\documentclass[11pt,reqno]{amsart}

\usepackage{amsmath}
\usepackage{amsthm}
\usepackage{amssymb}
\usepackage{enumerate}
\usepackage{amscd}
\usepackage{color}

\usepackage[latin1]{inputenc}

\title[]{Rigidity theorems for the area widths of Riemannian manifolds}

\author{Lucas Ambrozio, Fernando C. Marques and Andr\'e Neves}

\address{IMPA - Instituto Nacional de Matematica Pura e Aplicada, Rio de Janeiro, RJ, Brasil, 22460-320.}
\email{\texttt{l.ambrozio@impa.br}}

\address{Princeton University, Princeton, NJ, USA, 08544.}
\email{\texttt{coda@math.princeton.edu}}

\address{University of Chicago, Chicago, IL, USA, 60637.}
\email{\texttt{aneves@uchicago.edu}}

\thanks{The first author is supported  by CNPq (309908/2021-3 - Bolsa PQ) and by FAPERJ  (grant SEI-260003/000534/2023 - BOLSA E-26/200.175/2023 and grant SEI-260003/001527/2023 - APQ1 E-26/210.319/2023) . The second author is partly supported by NSF-DMS-2105557 and a Simons Investigator Grant. The third author is partly  supported by NSF-DMS-2005468 and a Simons Investigator Grant.}

\date{}

\newtheorem{thm}{Theorem}[section]
\newtheorem*{thmA}{Theorem A}
\newtheorem*{thmB}{Theorem B}
\newtheorem*{thmC}{Theorem C}
\newtheorem{lem}[thm]{Lemma}

\newtheorem{prop}[thm]{Proposition}
\theoremstyle{remark}
\newtheorem{rmk}{Remark}[section]

\def\RP{\mathbb{RP}}

\def\XXint#1#2#3{{\setbox0=\hbox{$#1{#2#3}{\int}$}
     \vcenter{\hbox{$#2#3$}}\kern-.5\wd0}}

\begin{document}


\begin{abstract}
The volume spectrum of a compact Riemannian manifold is a sequence of critical values for the area functional, defined in analogy with the Laplace spectrum by Gromov. In this paper we prove that the canonical metric on the two-dimensional projective plane is determined modulo isometries by its volume spectrum.  

We also prove  that the surface Zoll metrics on the three-dimensional sphere are characterized by the equality of the spherical area widths. These widths  generalize to the surface case the Lusternik-Schnirelmann lengths of closed geodesics. We prove a new sharp area systolic inequality for metrics on the three-dimensional projective space.
\end{abstract}

\maketitle



\section{Introduction}

We investigate inverse problems on determining or characterizing a Riemannian manifold by critical values of the area functional that are defined variationally. 
The settings are that of the volume functional on the space of modulo two closed hypersurfaces in a compact Riemannian manifold,  or that of the area functional on the space of smoothly embedded two-spheres in a three-dimensional Riemannian sphere.

In \cite{gromov-nonlinear}, Gromov introduced the notion of  volume spectrum of a compact Riemannian manifold. This is a sequence of numbers $\{\omega_k\}_{k\in \mathbb{N}}$ defined as a nonlinear analogue of the spectrum $\{\lambda_k\}_{k\in \mathbb{N}}$ of the Laplacian. This was part of a general framework proposed in \cite{gromov-nonlinear} of spectra of functionals defined on spaces that have a projective structure.

The $k$-th eigenvalue $\lambda_k$ is defined variationally by considering the Dirichlet energy on the projectivization of the space of scalar functions. The $k$-th eigenvalue of a compact Riemannian manifold $(M^{n+1},g)$  has a min-max definition:
$$
\lambda_k=\inf_{V\in \mathcal{V}_k} \sup_{f\in V, \int_M f^2=1}\int_M|\nabla f|^2,
$$
where $\mathcal{V}_k$ is the set of $k$-dimensional vector spaces of $W^{1,2}$ functions.

The $k$-width is defined similarly by
$$
\omega_k=\inf_{\mathcal{S}\in \mathcal{P}_k} \sup_{\Sigma\in \mathcal{S}}{\rm area}(\Sigma),
$$
where $\mathcal{P}_k$ denotes the set of $k$-sweepouts of $M$ by hypersurfaces, and $\Sigma\mapsto {\rm area}(\Sigma)$ is the area (or $n$-dimensional volume) functional. 

The eigenvalue $\lambda_k$ is the energy of a Laplace eigenfunction, whereas  the Almgren-Pitts min-max theory implies that the $k$-width $\omega_k$ is the area of a closed minimal hypersurface if $(n+1)\geq 3$, and the length of a stationary geodesic network if $(n+1)=2$.

The Weyl law states that the eigenvalues $\{\lambda_k\}$ satisfy the asymptotic formula:
$$\lim_{k \rightarrow \infty} \lambda_k k^{-\frac{2}{n+1}}  = c(n){\rm vol}(M)^{-\frac{2}{n+1}},$$
where $c(n)= 4 \pi^2 {\rm vol}(B)^{-\frac{2}{n+1}}$ and $B$ is the unit ball in $\mathbb{R}^{n+1}$. Hence the Laplace spectrum determines the dimension and the volume of the manifold. The inverse problem of determining geometric properties of a space from its  Laplace spectrum is a classic question (\cite{kac-hear}, \cite{berger-gauduchon-mazet}).

In \cite{liokumovich-marques-neves}, Liokumovich and the second and third authors proved that the volume spectrum obeys a Weyl law conjectured in \cite{gromov-maps}:
$$
\lim_{k\rightarrow \infty} \omega_k k^{-\frac{1}{n+1}}=a(n) {\rm vol}(M)^\frac{n}{n+1},
$$
where $a(n)$ is a positive dimensional constant. This generalized the previous estimate $c^{-1}k^\frac{1}{n+1}\leq \omega_k \leq c\,k^\frac{1}{n+1}$, with $c$ depending on the metric, due to Gromov \cite{gromov-maps} and Guth \cite{guth}.  

Thus, as before, the sequence $\{\omega_k\}$ determines the dimension and the volume of a Riemannian manifold. We would like to investigate to what extent its volume spectrum determines the geometry.

We prove that the two-dimensional constant curvature one projective plane  $(\mathbb{RP}^2,\overline{g})$ is determined by its volume spectrum within the class of compact Riemannian manifolds.  

\begin{thmA}\label{rigidity.projective.plane}
Let $(M^{n+1},g)$ be a compact Riemannian manifold. If 
$$\omega_k(M,g)=\omega_k(\mathbb{RP}^2,\overline{g})$$
 for any $k
\geq 1$, then $(M,g)$ is isometric to $(\mathbb{RP}^2,\overline{g})$.
\end{thmA}

A flat metric on the circle $S^1$ is determined by the dimension and the length, and hence by the volume spectrum. These metrics and the constant curvature metrics on $\mathbb{RP}^2$ are the only examples known with this property up to the results of this paper.

In \cite{chodosh-mantoulidis-widths}, Chodosh and Mantoulidis described the volume spectrum of the constant curvature one  two-sphere $(S^2,\overline{g})$ and hence proved that $a(1)=\sqrt{\pi}$ (the constant $a(n)$ 
is not known for  $n\geq 2$). This follows from the formula $$\omega_k(S^2,\overline{g})=2\pi\lfloor\sqrt{k}\rfloor,$$ 
for  $k\geq 1$. The main tool  to prove this is a min-max theory developed in \cite{chodosh-mantoulidis-widths} 
 which implies that, 
for each $k\geq 1$, the $k$-width of a Riemannian surface is the sum of the lengths of immersed closed geodesics.

In  \cite{marx-kuo}, Marx-Kuo used this to  prove that a Zoll metric $g$ on $S^2$ that is in the connected component of the round metric $\overline{g}$ in the space of Zoll metrics has the same volume spectrum as $(S^2,\overline{g})$. A Zoll metric is a Riemannian metric whose geodesics are all closed, embedded and have the same length, and there are plenty of them on $S^2$ (\cite{guillemin}, \cite{besse}). Hence the surface $(S^2,\overline{g})$ is flexible for the $k$-widths isospectral problem.

	In contrast with these results, both the two-sphere and the real projective plane with their constant curvature one metrics are rigid for the Laplace eigenvalues isospectral problem (\cite{berger-gauduchon-mazet}, Corollary III.E.IV.16).

	The class of Zoll metrics on $S^2$  can be characterized by a different length spectrum. The Lusternik and Schnirelmann min-max values of a Riemannian two-sphere are three positive numbers $\ell_1\leq \ell_2\leq \ell_3$, which are obtained by doing min-max for the length functional on the space of embedded closed curves in $S^2$. The  Lusternik-Schnirelmann theory (\cite{lusternik-schnirelmann}), by using Grayson's results on curve shortening flow (\cite{grayson}), shows that each $l_i$ is the length of an embedded closed geodesic. 
	
	In \cite{mazzucchelli-suhr}, Mazzucchelli and Suhr proved that the Zoll metrics on $S^2$ are characterized by the coincidence of the Lusternik-Schnirelmann min-max values.

	The min-max theory for the area functional in the Simon-Smith variant of the Almgren-Pitts theory finds closed minimal surfaces with genus bounds in closed three-dimensional Riemannian manifolds. If $g$ is a Riemannian metric on the three-sphere $S^3$, the spherical area widths are defined by
$$
\sigma_k(S^3,g)=\inf_{\mathcal{S} \in \mathcal{P}_k^{'}} \sup_{\Sigma \in \mathcal{S}} {\rm area}(\Sigma),
$$
where $1\leq k\leq 4$ and $\mathcal{P}_k'$ is the set of $k$-sweepouts by embedded two-spheres that are continuous in the smooth topology. Notice that $\mathcal{P}_k'\subset \mathcal{P}_k$, and hence $\omega_k(S^3,g)\leq \sigma_k(S^3,g)$ for $1\leq k\leq 4$.
The theory implies that $\sigma_k$ is the area of a closed minimal surface whose components are embedded two-spheres.

In \cite{ambrozio-marques-neves}, the authors found metrics on the sphere $S^{n+1}$ with continuous $(n+1)$-parameter families of closed minimal hypersurfaces, for   $n\geq 2$.  These metrics contain a Zoll  family $\{\Sigma_\sigma\}_{\sigma\in \mathbb{RP}^{n+1}}$ of $n$-dimensional minimal  spheres, i.e.  such that for any $x\in S^{n+1}$ and any $n$-dimensional space $\pi\subset T_xS^{n+1}$ there exists a unique $\sigma\in \mathbb{RP}^{n+1}$ such that $T_x\Sigma_\sigma=\pi$.  The examples found in \cite{ambrozio-marques-neves} include smooth perturbations of the canonical metric. 

In this paper  a Riemannian metric on $S^3$ with a Zoll family of minimal two-spheres will be called a  surface Zoll metric. We prove that this class of metrics can be characterized by the spherical area widths.

\begin{thmB}\label{zoll.characterization}
Let $g$ be a smooth Riemannian metric on $S^3$. Then 
$$
\sigma_1(S^3,g)=\sigma_2(S^3,g)=\sigma_3(S^3,g)=\sigma_4(S^3,g)
$$
if and only if $g$ is a surface Zoll metric.
\end{thmB} 

If $g$ is a Riemannian metric on the three-dimensional projective space $\mathbb{RP}^3$, one can consider smooth 
sweepouts by embedded projective planes. In this case, we define the projective area widths 
$$
\sigma_k(\mathbb{RP}^3,g)= \inf_{\mathcal{S} \in \mathcal{P}_k^{'}} \sup_{\Sigma \in \mathcal{S}} {\rm area}(\Sigma),
$$
where $0\leq k\leq 3$ and $\mathcal{P}_k'$ is the set of $k$-sweepouts by embedded projective planes  that are continuous in the smooth topology. 

The systole of a Riemannian manifold $(X,h)$, ${\rm sys}(X,h)$, is defined as the least length of a non-contractible loop in $X$.

We also prove an optimal area systolic inequality for metrics on $\mathbb{RP}^3$:
\begin{thmC}\label{systolic.inequality} Let $g$ be a smooth Riemannian metric on $\RP^3$. Then 
$$  {\rm sys}^2(\RP^3,g)\leq \frac{\pi}{2}\,\sigma_2(\RP^3,g),$$
with equality if and only if $g$ has constant sectional curvature.
\end{thmC}

The inequality of Theorem C follows since $ \text{sys}^2(\RP^3,g)\leq \frac{\pi}{2}\,\sigma_0(\RP^3,g)$. Theorem C
is sharp also in the sense that there are Riemannian metrics $g$ on $\mathbb{RP}^3$ with  ${\rm sys}^2(\RP^3,g)=\frac{\pi}{2}\,\sigma_1(\RP^3,g)$ that do not have constant sectional curvature.

\subsection{Main ideas and plan of the paper} The Weyl law for the volume spectrum implies that a compact Riemannian manifold with the same volume spectrum of $(\mathbb{RP}^2,\overline{g})$ is two-dimensional and has area $2\pi$. Since the first three widths of $(\mathbb{RP}^2,\overline{g})$ are equal to $2\pi$ (Proposition \ref{widths.projective.plane}), the strategy is to analyze the implications of the coincidence of $\{\omega_i\}_{1\leq i\leq 3}$ (Theorem \ref{rigidity.projective.plane2}).

	The heuristics of Lusternik-Schnirelmann theory suggests that the equality $\omega_1=\omega_2=\omega_3$ on a surface implies the existence of a two-sweepout by critical points of the length functional. Notice that the foliation $\{\gamma_s\}_{s\in S^1}$ of flat square two-tori and Klein bottles by simple closed two-sided geodesics generates such a sweepout, consisting of cycles $T_{\{s,t\}}=[\gamma_{s}]+[\gamma_{t}]$, for any unordered pair of points $s,t\in S^1$. A key step of our analysis is to rule out a case like this, by showing the existence of an optimal two-sweepout whose critical set accumulates only on stationary varifolds that are not  of the form $\gamma_s+\gamma_t=\partial U$ with $area(U)=\pi$ and $\gamma_s$, $\gamma_t$ elements of some foliation as above if it exists. See Proposition \ref{modification.prop}, and how this property is crucially used in Proposition \ref{no.disjoint.geodesic}.
	
	Then we show that the critical set of this two-sweepout contains only stationary integral varifolds that are supported on either a simple closed geodesic with multiplicity one, or on a figure-eight closed geodesic with multiplicity one, or on the sum of two simple closed geodesics with multiplicity one that intersect at a single point, or on a simple closed geodesic with multiplicity two (Propositions \ref{no.disjoint.geodesic}, \ref{connected.support.geodesic}, \ref{multiplicity.edges}, \ref{vertex} and \ref{tangent.cone}). While the second case is not possible, we argue that the first case only happens on a Zoll two-sphere whose geodesics have length $2\pi$, and the last two cases only happen on a projective plane whose systole is $\pi$. By a theorem of Weinstein \cite{weinstein}, the Zoll case is not compatible with area $2\pi$. Then the surface is a projective plane and  the equality case in Pu's systolic inequality \cite{pu} implies rigidity.

	The analysis of the equality of the spherical area widths of $S^3$ is similar and somehow simpler, because the critical set in this case consists of stationary integral varifolds that are supported on the disjoint union of embedded two-spheres with Morse index less than or equal to one, and because there is a bound on the maximum dimension of a family of embedded minimal two-spheres with index one in $S^3$ (\cite{cheng}). After showing that there are optimal three-sweepouts arising from the equality of the four spherical area widths that avoid minimal two-spheres  that are not part of a three-manifold of embedded minimal two-spheres (Proposition \ref{ruling.out.2}), we construct a three-sweepout that is a Zoll family of minimal two-spheres.
	
	The proof of Theorem C is by arguing   that  the equality case implies that  $\sigma_0(\mathbb{RP}^3,g)=\sigma_2(\mathbb{RP}^3,g)$. It follows by using  Lusternik-Schnirelmann  that there is a two-sweepout
	of $(\mathbb{RP}^3,g)$ by totally geodesic, constant curvature, embedded projective planes.  This implies that the metric $g$ is Zoll and hence has constant sectional curvature by Besse \cite{besse}.
	
	The plan of the paper is the following. In Section 2, we discuss some preliminaries. In  Section 3, we describe the first  widths of the canonical metric of $\mathbb{RP}^2$ and of  Zoll metrics on the sphere $S^2$. In Section 4, we prove Theorem A on the rigidity of the canonical projective plane for the volume spectrum. In Section 5, we prove the if part of Theorem B on the characterization of surface Zoll metrics by the spherical area widths. In Section 6, we finish the proof of Theorem B.  In Section 7, we prove Theorem C on the area systolic inequality for metrics on the three-dimensional projective space.

\section{Preliminaries}

Let $(M^{n+1},g)$ be a compact Riemannian manifold.
Let $\mathcal{Z}_n(M^{n+1},\mathbb{Z}_2)$ be the space of modulo two flat $n$-dimensional chains $T=\partial U$, where $U$ is a modulo two flat $(n+1)$-dimensional  chain  with support in $M$ ($U\in {\bf I}_{n+1}(M^{n+1},\mathbb{Z}_2)$). We denote by ${\bf I}_k(M,\mathbb{Z}_2)$ the space of $k$-dimensional modulo two flat chains with support in $M$. 
The  spaces $\mathcal{Z}_n(M^{n+1},\mathbb{Z}_2)$,  ${\bf I}_k(M,\mathbb{Z}_2)$ are topologized by the flat metric $\mathcal{F}$.

We denote by $\mathcal{V}_n(M)$ 
the closure of the space of $n$-dimensional rectifiable varifolds with support in $M$ in the varifold topology. The support of $V\in \mathcal{V}_n(M)$ is denoted by ${\rm spt}(V)$.  The varifold induced by a flat chain $T$ is denoted by $|T|$, and its support by ${\rm spt}(T)$. 
We denote by $\mathcal{Z}_n(M^{n+1},{\bf F}, \mathbb{Z}_2)$ the  space of modulo two boundaries with the ${\bf F}$-metric:
${\bf F}(S,T)=\mathcal{F}(S,T)+{\bf F}(|S|,|T|)$.  The inclusion
$\mathcal{Z}_n(M^{n+1},{\bf F}, \mathbb{Z}_2)\rightarrow \mathcal{Z}_n(M^{n+1},\mathbb{Z}_2)$ is a continuous map.  The mass of a chain $T$ is denoted by ${\bf M}(T)$, and the mass of a varifold $V$ is denoted by $||V||(M)$. We refer the reader to Section 2 of \cite{marques-neves-lower-bound} for a discussion of these definitions.

The space $\mathcal{Z}_n(M^{n+1},\mathbb{Z}_2)$ is weakly homotopically equivalent to $\mathbb{RP}^\infty$. Therefore the cohomology with $\mathbb{Z}_2$ coefficients is
$$
H^k(\mathcal{Z}_n(M^{n+1},\mathbb{Z}_2), \mathbb{Z}_2)=\mathbb{Z}_2=\{0, \overline{\lambda}^k\},
$$
where $\overline{\lambda}$ is the generator of $H^1(\mathcal{Z}_n(M^{n+1},\mathbb{Z}_2), \mathbb{Z}_2)=\mathbb{Z}_2$ and  $\overline{\lambda}^k=\overline{\lambda} \smile \cdots \smile \overline{\lambda}$ is the $k$-th cup power of $\overline{\lambda}$.

A $k$-sweepout is a continuous map $\Phi:X \rightarrow \mathcal{Z}_n(M^{n+1},{\bf F},\mathbb{Z}_2)$, where  $X$ is a finite-dimensional compact simplicial complex that depends on $\Phi$, such that
$$
\Phi^*(\overline{\lambda}^k)\neq 0\in H^k(X,\mathbb{Z}_2).
$$
We denote $\Phi \in \mathcal{P}_k$. 

For any integer $k\geq 1$, the $k$-width of $(M^{n+1},g)$ is the min-max number
$$
\omega_k(M,g)= \inf_{\Phi\in \mathcal{P}_k} \sup_{x\in {\rm dmn}(\Phi)} {\bf M}(\Phi(x)),
$$
where  ${\rm dmn}(\Phi)$ is the domain of $\Phi$.  The sequence $\{\omega_k(M,g)\}_k$ is called the volume spectrum of $(M,g)$.  The volume spectrum satisfies a Weyl law (\cite{liokumovich-marques-neves}):
\begin{equation}\label{weyl-law}
\lim_{k\rightarrow \infty} \omega_k(M,g) k^{-\frac{1}{n+1}}=a(n) {\rm vol}(M,g)^\frac{n}{n+1},
\end{equation}
where $a(n)$ is a positive dimensional constant. In \cite{chodosh-mantoulidis-widths}, it was proven that $a(1)=\sqrt{\pi}$.  The constant $a(n)$ for  $n\geq 2$
is not known. 

If  $\{\Phi_i\}_i \subset  \mathcal{P}_k$ is a sequence of $k$-sweepouts  such that 
$$
\sup_{x \in {\rm dmn}(\Phi_i)} {\bf M}(\Phi_i(x)) \rightarrow \omega_k(M,g),
$$
we say that the sequence $\{\Phi_i\}_i$ is optimal for $\omega_k$.
The image set of $\{\Phi_i\}$ is the set ${\bf \Lambda}(\{\Phi_i\})$ of varifolds $V\in \mathcal{V}_n(M)$ such that there are sequences $\{j\}\subset \{i\}$ and $\{x_j\in {\rm dmn}(\Phi_{j})\}$ with
$$
{\bf F}(|\Phi_{j}(x_j)|,V)\rightarrow 0.
$$
The critical set of $\{\Phi_i\}$ is the set ${\bf C}(\{\Phi_i\})$ of varifolds $V\in {\bf \Lambda}(\{\Phi_i\})$ such that
$||V||(M)=\omega_k(M,g)$.
If $(n+1)=2$, the Almgren-Pitts min-max theory implies that there is a stationary geodesic network with integer multiplicities $V\in {\bf C}(\{\Phi_i\})$.  If $(n+1)\geq 3$, the same min-max theory produces  $V\in {\bf C}(\{\Phi_i\})$ the varifold of  a  closed minimal hypersurface,  smoothly embedded outside a set of codimension seven, with integer multiplicities. 

The Simon-Smith variant of min-max theory is based on  sweepouts of smooth embedded surfaces (\cite{smith}, \cite{colding-delellis}, \cite{haslhofer-ketover}). Let $\mathcal{S}$ be the space  of smooth embedded two-spheres in $S^3$, say $\mathcal{\tilde{S}}$, with a point $\ast$ added. The open sets of $\mathcal{\tilde{S}}$ in the smooth topology together with the sets
$$
\{\ast\} \cup \{\Sigma \in \mathcal{\tilde{S}}:  {\rm area}(\Sigma,can)<\delta\},
$$
for $\delta>0$, form a basis for a topology $\mathcal{T}$ on $\mathcal{S}$. This is a modification of the space used in \cite{haslhofer-ketover}. 

 The inclusion $i:\mathcal{S}\rightarrow \mathcal{Z}_2(S^3,{\bf F},\mathbb{Z}_2)$, which sends $\ast$ to the trivial cycle, is a continuous map. 
A continuous map $\Phi:X\rightarrow \mathcal{S}$ is called  a smooth $k$-sweepout (by spheres) if the composition $i\circ \Phi$ is a $k$-sweepout. We denote $\Phi \in \mathcal{P}_k'$. Notice that the map
$$
\Psi_{\mathcal{S}}([a_0:a_1:a_2:a_3:a_4])= \{x\in S^3:a_0+a_1x_1+a_2x_2+a_3x_3+a_4x_4=0\}
$$
defines a smooth 4-sweepout $\Psi_{\mathcal{S}}:\mathbb{RP}^4 \rightarrow \mathcal{S}$.

The spherical area  widths of $(S^3,g)$ are the min-max numbers
$$
\sigma_k(S^3,g)=\inf_{\Phi\in \mathcal{P}_k'} \sup_{x\in {\rm dmn}(\Phi)} {\bf M}(\Phi(x)),
$$
for $1\leq k\leq 4$. Notice that $\mathcal{P}_k'\subset \mathcal{P}_k$ (if we identify $\Phi$ with $i\circ \Phi$), hence $\omega_k(S^3,g)\leq \sigma_k(S^3,g)$ for $1\leq k\leq 4$.
Also $\mathcal{P}_k'\subset \mathcal{P}_l'$  for $l\leq k$. 

Let $\overline{g}$ denote the canonical constant curvature one metric on $S^3$. Then 
$\sigma_1(S^3,\overline{g})=\sigma_2(S^3,\overline{g})=\sigma_3(S^3,\overline{g})=\sigma_4(S^3,\overline{g})=4\pi$.

Hatcher proved that $\tilde{\mathcal{S}}$ deformation retracts onto the space of equators (see appendix of \cite{hatcher}). The space of equators is homeomorphic to $\mathbb{RP}^3$.

\begin{prop}
There are no smooth $k$-sweepouts of $S^3$ by spheres for $k\geq 5$.
\end{prop}

\begin{proof}
Let $\Phi:X \rightarrow \mathcal{S}$ be a continuous map.  If $0<\delta<\omega_1(S^3,\overline{g})$,  the restriction of $(i \circ\Phi)^*(\overline{\lambda})$ to the open set
$$
Z_1=\{x \in X: {\rm area}(\Phi(x),can)<\delta\}
$$
vanishes in cohomology. The deformation retraction of $\tilde{\mathcal{S}}$ onto an $\mathbb{RP}^3$ implies that the restriction of $(i \circ\Phi)^*(\overline{\lambda}^4)$ to the open set
$$
Z_2=\{x\in X: \Phi(x)\in \tilde{\mathcal{S}}\}
$$
also vanishes in cohomology. Therefore, by properties of the cup product and since $X=Z_1\cup Z_2$,
$$
(i \circ\Phi)^*(\overline{\lambda}^5)_{|X}=0\in H^5(X,\mathbb{Z}_2).
$$
This proves the proposition.
\end{proof}

The Simon-Smith theory gives    genus bounds (\cite{smith}, \cite{delellis-pellandini}, \cite{ketover}) which imply that for each $1\leq k\leq 4$, there exist a disjoint family $\{\tilde{\Sigma}_1^{(k)}, \dots, \tilde{\Sigma}_{q_k}^{(k)}\}$ of smooth, embedded, minimal two-spheres for $(S^3,g)$ and $\{m_1^{(k)}, \dots, m_{q_k}^{(k)}\}\subset \mathbb{N}$ such that
$$
\sigma_k(S^3,g)=\sum_{i=1}^{q_k}m^{(k)}_{i}{\rm area}(\tilde{\Sigma}^{(k)}_{i}).
$$
The Morse index estimates of \cite{marques-neves-index}  also give
$$
\sum_{i=1}^{q_k} {\rm index}(\tilde{\Sigma}_i^{(k)})\leq k.
$$
These results follow from the corresponding existence theorem, with upper index bounds, for min-max invariants over homotopy classes.     This uses Sharp's compactness theorem (\cite{sharp}) and the fact that   $\sigma_k$ is the limit of  min-max invariants of a sequence of homotopy classes with $k$-dimensional domains. By  Wang-Zhou \cite{wang-zhou},    if $\tilde{\Sigma}^{(k)}_{i}$ is unstable then $m^{(k)}_{i}=1$.
\medskip

We now state and prove some results which will be used in the paper.
 The next proposition follows by the maximum principle after taking limits.

\begin{prop}\label{intersects.must.be.contained.geodesic}
Let $(M^2,g)$ be a compact Riemannian surface and $\gamma\subset M$ be a  smooth, closed, embedded geodesic. For each $\lambda, \eta>0$,  there exists $\tilde{\eta}>0$ such that every stationary integral one-varifold $V$ with ${\bf M}(V)\leq \lambda$, and with ${\rm spt}(V)$ connected, disjoint from $\gamma$,  intersecting the tubular neighborhood $B_{\tilde \eta}(\gamma)$, must satisfy ${\rm spt}(V) \subset B_\eta(\gamma)$.
\end{prop}

\begin{proof}
Suppose, by contradiction, that the proposition is not true. Then there is  a sequence $\{V_i\}_i$ of stationary integral one-varifolds with ${\bf M}(V_i)\leq \lambda$,  with ${\rm spt}(V_i)$ connected, disjoint from $\gamma$, not contained in $B_\eta(\gamma)$, such that there are points $p_i\in {\rm spt}(V_i)$ converging to a point $p\in \gamma$. By Allard-Almgren \cite{allard-almgren}, for each $i$ the varifold $V_i$ is induced by a stationary geodesic network with integer multiplicities. 

We suppose that $\gamma$ is two-sided. Let $V_i'$ be the varifold of the connected component of ${\rm spt}(V_i) \cap B_{\eta}(\gamma)$ passing by $p_i$ with the same multiplicities as $V_i$.  By compactness of stationary integral varifolds with mass bounds, $V_i'$ converges in varifold sense in compact sets of  $B_{\eta}(\gamma)$ to a stationary integral one-varifold $V$ in $B_{\eta}(\gamma)$. Notice that $p\in {\rm spt}(V)$. By the maximum principle of White (Theorem 4 of \cite{white-maximum-principle}), there exists $0<\eta'<\eta/2$ such that $V \llcorner B_{2\eta'}(\gamma)=k \cdot |\gamma|$ for some $k\in \mathbb{N}$. This implies that for sufficiently large $i$, the support of $V_i'$ does not
intersect $B_{\eta'}(\gamma)\setminus B_{\eta'/2}(\gamma)$. Since ${\rm spt}(V_i)$ is connected, it follows that for sufficiently large $i$ the 
support of $V_i$ is contained in $B_{\eta}(\gamma)$. This is a contradiction. If the geodesic $\gamma$ is one-sided, we can obtain  a contradiction similarly by lifting to a double cover. This finishes the proof of the proposition.
\end{proof}

Let $(M^{n+1},g)$ be a compact Riemannian manifold.  If $\Sigma\subset (M,g)$ is a closed, embedded, minimal hypersurface,  denote by $L_\Sigma$ its Jacobi operator.
 
\begin{prop}\label{approximate.minimal}
Let $\Sigma^n$ be a smooth, closed, two-sided, embedded minimal hypersurface in $(M^{n+1},g)$, and let  $k\geq 2$ be an integer and $\alpha\in (0,1)$. Suppose that ${\rm dim}({\rm Ker}\, L_\Sigma)=j$. Then there exists a neighborhood $\mathcal{W}\subset {\rm Ker} \, L_\Sigma$ of the origin, and a smooth embedding $\varphi:\mathcal{W}\rightarrow C^{k,\alpha}(\Sigma)$ with $\varphi(0)=0$, $D\varphi(0)=Id$, such that every closed, minimal hypersurface $\Sigma'$ that is sufficiently close in the varifold topology  to $\Sigma$ must be the graph over $\Sigma$ of the function $\varphi(z)$ for some $z\in \mathcal{W}$.

\end{prop}

\begin{proof}
This proposition follows  Theorem 1.3 of White \cite{white2}, since Allard implies that any
closed, minimal hypersurface $\Sigma'$ that is sufficiently close in the varifold topology  to $\Sigma$ is close to $\Sigma$
in the smooth topology.
\end{proof}

We are going to use the following proposition which is similar to Proposition \ref{intersects.must.be.contained.geodesic}.
\begin{prop}\label{intersects.must.be.contained}
Let $(M^{n+1},g)$ be a compact Riemannian manifold and $\Sigma^n\subset M$ be a smooth, closed, embedded minimal hypersurface. For each $\lambda, \eta>0$,  there is $0<\tilde{\eta}<\eta$ such that any connected, closed, smooth, embedded minimal hypersurface $\tilde{\Sigma}$, disjoint from $\Sigma$, with ${\rm area}(\tilde{\Sigma})+{\rm index}(\tilde{\Sigma})\leq \lambda$,  that intersects the tubular neighborhood $B_{\tilde \eta}(\Sigma)$ must be contained in $B_\eta(\Sigma)$.
\end{prop}

\begin{proof}
Suppose, by contradiction, that the result is  not true. Then there exists a sequence $\{\tilde{\Sigma}_i\}_i$ of connected, closed, smooth, embedded minimal hypersurfaces, disjoint from $\Sigma$, not contained in $B_\eta(\Sigma)$, with ${\rm area}(\tilde{\Sigma}_i)+{\rm index}(\tilde{\Sigma}_i)\leq \lambda$ and points $p_i\in \tilde{\Sigma}_i$ converging to a point $p\in \Sigma$. 

Let $\tilde{\Sigma}_i'$ be the connected component of $\tilde{\Sigma}_i\cap B_{2\eta}(\Sigma)$ passing by $p_i$.  By Sharp's compactness theorem (\cite{sharp}), because of the area and index bounds, $\tilde{\Sigma}_i'$ converges in varifold sense (and locally graphically outside finitely many points) in compact sets of  $B_{2\eta}(\Sigma)$ to a  smooth, embedded, minimal hypersurface $\overline{\Sigma}$ with integer multiplicities. Notice that $p\in \overline{\Sigma}$. If the connected component $\Sigma'$ of $\Sigma$ passing by $p$  is two-sided, we have that 
$\overline{\Sigma}$ is on one side of it  since
the  $\tilde{\Sigma}_i'$ are connected. By the maximum principle, $\overline{\Sigma}$ contains  $\Sigma'$. The graphical  convergence and connectedness of $\tilde{\Sigma}_i'$ imply that for sufficiently large $i$ we have $\tilde{\Sigma}_i'\subset B_\eta(\Sigma).$ Then $\tilde{\Sigma}_i'=\tilde{\Sigma}_i$ which is a contradiction because $\tilde{\Sigma}_i$ is not contained in
$B_\eta(\Sigma)$.  If $\Sigma'$ is one-sided, we can obtain  a contradiction similarly by lifting to the corresponding double cover.
\end{proof}

The next proposition also follows by the maximum principle.
 \begin{prop}\label{maximum.principle}
 Let $\Phi:S^1\rightarrow \mathcal{Z}_n(M^{n+1},{\bf F},\mathbb{Z}_2)$ be a one-sweepout so that $\Phi(\sigma)$ is a smooth, connected, embedded, closed minimal hypersurface  for $g$
 with multiplicity one for every $\sigma \in S^1$.  Assume there is $V$ an $n$-dimensional nontrivial stationary varifold in $(M,g)$ with ${\rm support}(V)$ disjoint from $\Phi(\sigma)$ for some $\sigma$. Then there exists $\sigma_0\in S^1$ such that ${\rm support}(V) \supset \Phi(\sigma_0)$. 
 \end{prop}
 
 \begin{proof}
 The condition of being a one-sweepout implies there is $t\in [0,2\pi]\mapsto U(t)\in {\bf I}_{n+1}(M^{n+1},\mathbb{Z}_2)$ continuous in the flat topology, $U(t)$ open sets, with $\partial U(t)=\Phi(e^{it})$ and $U(2\pi)=M - U(0)$.  Let  $\sigma=e^{it}$. We can suppose {$V'=V \llcorner U(t)$} is a nontrivial stationary varifold.

Let $p\in {\rm support}(V')$. For any open set $O\subset M$ containing $p$, there exists $\tilde{\sigma}\in S^1$ such that $\Phi(\tilde\sigma)$ intersects $O$ (\cite{guth}). {The map $\sigma\mapsto \Phi(\sigma)$ is necessarily  continuous in the Hausdorff distance and so   there exists $\tilde\sigma\in S^1$ such that $\Phi(\tilde\sigma)\cap {\rm support}(V')\neq \emptyset$. Combining with the fact that $ {\rm support}(V')\subset U(t)$, a standard continuity argument gives the existence of $\sigma_0=e^{it_0}$ so that ${\rm support}(V') \cap \overline{U(t_0)} \neq \emptyset$ and ${\rm support}(V') \subset \overline{U(t_0)}$.
The maximum principle for stationary varifolds of \cite{white-maximum-principle} implies that ${\rm support}(V') \supset \Phi(\sigma_0)=\partial U(t_0)$.}
 \end{proof}
 
\begin{lem}\label{convergence.supports}
 Let $\mathcal{K}$ be a compact set of stationary integral $n$-varifolds in $M$ and $K\subset M$ be a compact set such that
${\rm spt}(W)\subset K$ for every $W\in \mathcal{K}$. For each $\eta>0$ there exists $\zeta>0$ such that if a stationary integral $n$-varifold $V$ satisfies ${\bf F}(V,\mathcal{K})\leq \zeta$, then ${\rm spt}(V)\subset B_\eta(K)$.
\end{lem}

\begin{proof}
Suppose, by contradiction, that the lemma is false. Then there exist $\eta>0$ and a sequence $\{V_j\}$ of stationary integral varifolds converging to some $W\in \mathcal{K}$ with ${\rm spt}(V_j)\not\subset B_\eta(K)$ for every $j$. The monotonicity 
formula implies that ${\rm spt}(V_j)$ converges to ${\rm spt}(W)$ in the Hausdorff distance, which is a contradiction
since ${\rm support}(W)\subset K$.
\end{proof}

\begin{lem}\label{close.implies.homotopic}
 There exists $\delta>0$ such that if $\Phi\in \mathcal{P}_k$, then any
$\Phi':\text{dmn}(\Phi)\rightarrow \mathcal{Z}_n(M^{n+1},{\bf F},\mathbb{Z}_2)$ with
$
\mathcal{F}(\Phi'(x),\Phi(x))<\delta
$
for every $x\in \text{dmn}(\Phi)$ 
also satisfies $\Phi'\in \mathcal{P}_k$.

\end{lem}

\begin{proof}
There is $\delta>0$ such that any continuous map $c:S^1\rightarrow \mathcal{Z}_n(M^{n+1},\mathbb{Z}_2)$ with $\mathcal{F}(c(t),0)<\delta$ for every $t\in S^1$ is homotopically trivial in the flat topology. For a continuous map $\gamma:S^1 \rightarrow X$, we can write $\Phi'(\gamma(t))=\Phi(\gamma(t))+c(t)$ with $\mathcal{F}(c(t),0)<\delta$ for every $t\in S^1$. Since $c$ is homotopically trivial, $\Phi'\circ \gamma$ is homotopic in the flat topology to $\Phi \circ \gamma$.  Hence $\Phi'\circ \gamma$ is a one-sweepout if and only if $\Phi \circ \gamma$ is a sweepout. This implies $(\Phi')^*(\overline{\lambda})$ and $(\Phi)^*(\overline{\lambda})$ agree when
evaluated at $\gamma$. Hence $(\Phi')^*(\overline{\lambda})=(\Phi)^*(\overline{\lambda})$ and the lemma is proved.
\end{proof}

For $\zeta>0$, we define a $\zeta$-chain to be a sequence of $n$-varifolds $V_1,\dots,V_l$ with ${\bf F}(V_i,V_{i+1})\leq \zeta$ for each $1\leq i\leq l-1$.

\section{The widths of the projective plane and of Zoll two-spheres}

Let $\overline{g}$ be the canonical metric of constant curvature one on $\mathbb{RP}^2$. We denote by $\pi:S^2 \rightarrow \mathbb{RP}^2$ the two-cover projection, where $S^2$ is the unit two-sphere.

\begin{prop} \label{widths.projective.plane} The metric $\overline{g}$ on $\mathbb{RP}^2$ satisfies:
$$
\omega_1(\mathbb{RP}^2, \overline{g})=\omega_2(\mathbb{RP}^2, \overline{g})=\omega_3(\mathbb{RP}^2, \overline{g})=2\pi.
$$
\end{prop}
\begin{proof}
For $v\in \mathbb{R}^3$, $v\neq 0$, let $[v]$ denote the one-dimensional vector space generated by $v$ and let
$$
[v]^\perp=\{[w]\in \mathbb{RP}^2: \langle w,v\rangle=0\}.
$$
Let $\Phi:\mathbb{RP}^2\times\mathbb{RP}^2 \rightarrow \mathcal{Z}_1(\mathbb{RP}^2,\mathbb{Z}_2)$ be the map defined by
$$
\Phi([v],[w])=[v]^\perp + [w]^\perp.
$$
The map is well-defined since $[v]^\perp + [w]^\perp$ is a modulo two boundary for any $[v],[w]\in \mathbb{RP}^2$. It is continuous 
in the flat topology and has no concentration of mass (e.g.  Section 2 of \cite{liokumovich-marques-neves}). 
The restrictions of $\Phi$ to
$$
\{([(\cos\theta) e_1+(\sin \theta) e_2],[e_1])\in \mathbb{RP}^2\times\mathbb{RP}^2: \theta\in \mathbb{R}\}
$$
and
$$
\{([e_1],[(\cos\theta) e_1+(\sin \theta) e_2])\in \mathbb{RP}^2\times\mathbb{RP}^2: \theta\in \mathbb{R}\}
$$
are  one-sweepouts of $\mathbb{RP}^2$. Here $\{e_1,e_2,e_3\}$ denote the canonical basis of $\mathbb{R}^3$. 

Therefore $\Phi^*(\overline{\lambda})=\alpha +\beta$, where $\alpha$ is the generator 
of $H^1(\mathbb{RP}^2 \times \{y\},\mathbb{Z}_2)$ and $\beta$ is the generator of $H^1(\{x\}\times \mathbb{RP}^2,\mathbb{Z}_2)$.
Since $(\alpha+\beta)^3=\alpha^2 \beta +\alpha  \beta^2\neq 0$, the map $\Phi$ is a three-sweepout of $\mathbb{RP}^2$.
By Corollary 2.13 of 
\cite{liokumovich-marques-neves}, the same $k$-widths can be defined  using maps which are continuous in the flat topology and that  have no concentration of mass. Hence
$$
\omega_3(\mathbb{RP}^2,\overline{g})\leq \sup_{(x,y)\in \mathbb{RP}^2 \times \mathbb{RP}^2}{\bf M}(\Phi(x,y))= 2\pi.
$$

The Almgren-Pitts min-max theory implies that there is a stationary integral varifold  $V$ in $(\mathbb{RP}^2,\overline{g})$, which is a stationary geodesic network by \cite{allard-almgren}, such that
$\omega_1(\mathbb{RP}^2,\overline{g})=||V||(\mathbb{RP}^2)$.  By Aiex \cite{aiex}, we can also suppose that the density of $V$ is an integer at the vertices in ${\rm spt}(V)$. Hence $\tilde{V}=\pi^{-1}(V)$ is a stationary integral varifold in $S^2$ with integer densities at the vertices and such that
$||\tilde{V}||(S^2)=2\,||V||(\mathbb{RP}^2)$. The cone $C$ such that $C\cap S^2=\tilde{V}$ is a stationary integral varifold.  By the monotonicity formula, if there is $p \in {\rm spt}(\tilde{V})$ with density greater than or equal to $2$, then
$||\tilde{V}||(S^2)\geq 4\pi$. In this case $||V||(\mathbb{RP}^2)\geq 2\pi$, and then
$2\pi\leq \omega_1(\mathbb{RP}^2,\overline{g})\leq  \omega_2(\mathbb{RP}^2,\overline{g})\leq  \omega_3(\mathbb{RP}^2,\overline{g})\leq 2\pi$ which proves the proposition.

If $\tilde{V}$ is a density one stationary integral varifold, then it is the varifold induced by a closed geodesic of $S^2$. Hence  $V$ is the density one
varifold with support a closed geodesic of $(\mathbb{RP}^2,\overline{g})$. 
Let $\Phi_i:X_i\rightarrow \mathcal{Z}_1(\mathbb{RP}^2,\mathbb{Z}_2)$  be a sequence of one-sweepouts such that
$$
\sup_{x\in X_i}{\bf M}(\Phi_i(x))\rightarrow \omega_1(\mathbb{RP}^2,\overline{g}).
$$
Since there is a loop $\sigma_i\subset X_i$ such that $(\Phi_i)^*(\overline{\lambda})\cdot \sigma_i=1$,  we can suppose that $X_i=S^1$. By the pull-tight argument, we can suppose that any varifold $V \in {\bf C}(\{\Phi_i\})$ is stationary.  It follows by Pitts
\cite{pitts}  that there is  $V\in {\bf C}(\{\Phi_i\})$ which is $\mathbb{Z}_2$ almost minimizing in annuli (\cite{marques-neves-lower-bound}, Section 4). By Section 3.13 of Pitts \cite{pitts}, $V$ is a stationary integral varifold. The discussion of the previous paragraph implies that we can
suppose that any varifold $V\in {\bf C}(\{\Phi_i\})$ which is $\mathbb{Z}_2$ almost minimizing in annuli must be the density one varifold induced
by a closed geodesic of $(\mathbb{RP}^2,\overline{g})$. 

Let $\gamma$ be a closed geodesic of $(\mathbb{RP}^2,\overline{g})$, and let $|\gamma|$ be the corresponding density one varifold.
Let $\{T_i\in \mathcal{Z}_1(\mathbb{RP}^2,\mathbb{Z}_2)\}_i$ be a sequence of modulo two flat boundaries such that $|T_i|\rightarrow |\gamma|$ in the varifold topology.  By the Constancy Theorem of modulo two flat cycles, $T_i\rightarrow 0$ in the flat topology (because the flat cycle $\gamma$ is not a boundary). We can apply Proposition 4.10 of \cite{marques-neves-lower-bound} to $V=|\gamma|$ and $\Sigma=0$ even if the ambient dimension is two, since $\gamma$ has smooth support. Hence one can choose annuli for $V=|\gamma|$ as in Part 1 of the proof of Theorem 4.10 of \cite{pitts}, so that the Part 2 of that proof is satisfied. Since Parts 1 and 2 of the proof of Theorem 4.10 of \cite{pitts} can also be
satisfied for those varifolds $V\in {\bf C}(\{\Phi_i\})$ which are not $\mathbb{Z}_2$ almost minimizing in annuli, the combinatorial argument of Pitts
can be implemented. This argument was formulated for discrete sequences of maps which are fine in the mass topology. To adapt it to our context we need the interpolation machinery (e.g. Section 3 of \cite{marques-neves-lower-bound}, Section 2 of \cite{liokumovich-marques-neves}). It produces a sequence of one-sweepouts with masses bounded by $\omega_1-\eta$ for some $\eta>0$. This is a contradiction and hence the proposition is proved.  
\end{proof}

\begin{prop}\label{prop.zollsurface_w1=w3}
	Let $g$ be a Zoll metric on the sphere $S^2$ whose geodesics have length $\ell$. Then
	\begin{equation*}
		\omega_1(S^2,g)=\omega_2(S^2,g)=\omega_3(S^2,g)=\ell.
	\end{equation*}
\end{prop}

\begin{proof}
	By (\cite{besse}, 2.10), the space $\tilde{X}$ of closed oriented geodesics of $(S^2,g)$ is diffeomorphic to $S^2$.
	Any  geodesic $\gamma$ of $(S^2,g)$ separates $S^2$ into two components with boundary $\gamma$ by Jordan's theorem. Let $X^+$ be the set  of open subsets $\Omega\subset S^2$ such that $\partial \Omega$ is some geodesic of $(S^2,g)$. Since $S^2$ has an orientation, we can identify $X^+$ with $\tilde{X}$ by the oriented boundary map $\Omega \mapsto \partial \Omega$. The map $i:X^+\rightarrow X^+$ that sends $\Omega$ to $S^2 \setminus \overline{\Omega}$ is an involution, and hence the space $X$ of closed unoriented geodesics of $(S^2,g)$ is diffeomorphic to $\mathbb{RP}^2$.
	
	Any geodesic $\gamma$ of $(S^2,g)$ is unstable, because $\gamma$ is  two-sided of  nullity  at least two, due to the variations by geodesics within $X$. Given $\Omega\in X_+$ with $\partial \Omega=\gamma$, let $N_\Omega$ be the unit normal vector field of $\gamma$ that points towards $S^2\setminus \overline{\Omega}$, and let $\phi_\Omega$ denote the  positive first eigenfunction of the Jacobi operator of $\gamma$ that has $L^2$-norm equal to one. There exists $0<\tilde{s}<1$, which does not depend on $\Omega$, and a foliation $\{\gamma_s\}$, $s\in [0,\tilde{s}]$, of a neighborhood of $\gamma$ in $S^2\setminus \Omega$ by embedded curves with curvature vector pointing away from $\gamma$, produced by applying the normal exponential map of $\gamma$ to $s\phi_\Omega N_\Omega$. Let $\Omega_s$ be the component of $S^2\setminus \gamma_s$ that contains $\Omega$. Since $X$ is compact, we can define a map $\Psi: X_+\times [0,\tilde{s}] \rightarrow \mathcal{Z}_1(S^2,{\bf F},\mathbb{Z}_2)$ by $\Psi(\Omega,s)=\partial \Omega_{s}$. 
	
	The construction will require  further properties of geodesics of Zoll metrics on $S^2$. For instance, any two geodesics of $(S^2,g)$ intersect. Suppose by contradiction that this is not the case. Then there are disjoint geodesics $\gamma_0$ and $\gamma_1$ on $(S^2,g)$. Let $\Omega_0$ be the open component of $S^2\setminus \gamma_0$ that does not contain $\gamma_1$. Fix $p\in \gamma_0$. The continuous loop $t\in [0,1]\mapsto \partial \Omega_t\in X$, that switches between $\Omega_0$ with $\partial \Omega_0=\gamma_0$ and $\Omega_1=S^2\setminus \overline{\Omega}_0$, by moving along geodesics that contain $p$, produces a one-sweepout of $S^2$. In particular, some $\partial\Omega_t$ intersects $\gamma_1$, for some $0<t<1$. At the first time $t$ of contact, the intersection between the geodesics $\partial \Omega_{t}$ and $\gamma_1$ is tangential. By uniqueness of geodesics, $\gamma_1=\partial \Omega_{t}$. But this is a contradiction, since $\gamma_1\cap \gamma_0=\emptyset$, while $\partial \Omega_{t}$ contains $p$, by construction.
	
	By \cite{chodosh-mantoulidis-widths} (also \cite{calabi-cao}), there are closed geodesics $\gamma_1,\dots,\gamma_k$ such that   $\omega_1(S^2,g)=\sum_{j=1}^kl(\gamma_j)$. Since $(S^2,g)$ is Zoll, $l(\gamma_j)=l$ for any $j$ and hence $l\leq \omega_1(S^2,g)$.  Hence the proposition is proved if we can show that $\omega_3(S^2,g)\leq l$. 
	
	By Grayson's theorem \cite{grayson}, the curve shortening flow starting at each $\partial \Omega_{\tilde{s}}$ is such that either it flows to a closed embedded geodesic in $S^2\setminus \overline{\Omega}_{\tilde{s}}$, or the flow ends in finite time at a point. The first possibility does not happen because no geodesic of $(S^2,g)$ is disjoint from $\gamma=\partial \Omega$. Thus, for every $\Omega\in X_+$ there exists $T_{max}<\infty$ and a smooth curve shortening flow $\{\partial \Omega_s\}$, $s\in [\tilde{s},T_{max}]$, starting at $\partial \Omega_{\tilde{s}}$, ending at a round point, such that $\Omega_s\subset \Omega_t$ for $s<t$ and such that ${\bf M}(\partial \Omega_s)$ is strictly decreasing in $s$. 
The estimates of curve shortening flow can be used to extend the map $\Psi$ to a continuous map	
$\Psi: X_+\times [0,1] \rightarrow \mathcal{Z}_1(S^2,{\bf F}, \mathbb{Z}_2)$ such that for any $\Omega\in X^+$,   $\{\Psi(\Omega,s)\}_{s\in [\tilde{s},1]}$ is a reparametrization of the curve shortening flow $\{\partial \Omega_s\}_{s\in [\tilde{s},T_{max}(\Omega)]}$. Then $\Psi(\Omega,1)=0$ for any $\Omega\in X^+$. The space $X^+\times [0,1]$ if we identify the points of $X\times\{1\}$ is homeomorphic to the unit three-dimensional ball $B^3$. Since
$\Psi(\Omega,0)=\Psi(i(\Omega),0)=\partial \Omega$ for any $\Omega\in X^+$, the map $\Psi$ induces a continuous
map $\tilde{\Psi}:\mathbb{RP}^3 \rightarrow \mathcal{Z}_1(S^2,{\bf F}, \mathbb{Z}_2)$ with ${\bf M}(\tilde{\Psi}(z))\leq l$ for any $z\in \mathbb{RP}^3$.

The map $\tilde{\Psi}$ sends a nontrivial loop in $X \approx \mathbb{RP}^2$ into a one-sweepout of $S^2$,  since its lift to $X^+$ switches the components of $S^2\setminus \gamma$ for some geodesic $\gamma$.  Therefore $\tilde{\Psi}^*(\overline{\lambda})$ is the generator $\lambda$ of $H^1(\mathbb{RP}^3,\mathbb{Z}_2)$. Since $\lambda^3\neq 0$, the map $\tilde{\Psi}$ is a three-sweepout of $(S^2,g)$. Hence $\omega_3(S^2,g)\leq l$ as we wanted to prove.
\end{proof}

\section{Rigidity for the volume spectrum}

In this section we will prove Theorem A. 
We start by proving a metric characterization of compact Riemannian surfaces such that the widths $\omega_1,\omega_2,\omega_3$ are equal.

\begin{thm}\label{rigidity.projective.plane2}
Suppose that $(M^2,g)$ is a compact Riemannian surface such that
$$
\omega_1(M,g)=\omega_2(M,g)=\omega_3(M,g)=2\pi.
$$
Then either $(M,g)$ is isometric to a Zoll sphere $(S^2,g)$ whose geodesics have length $2\pi$, or $M$ is diffeomorphic to $\mathbb{RP}^2$ and  
${\rm sys}(\mathbb{RP}^2,g)=\pi$.  
\end{thm}

\begin{proof}
Let $(M^2,g)$ be a compact Riemannian surface which satisfies the hypotheses of Theorem \ref{rigidity.projective.plane2}. 
Let $\{\Psi_i\}_{i\in \mathbb{N}}\subset \mathcal{P}_3$ be such that
$$
\sup_{x\in X_i}{\bf M}(\Psi_i(x))\rightarrow \omega_3(M^2,g)=2\pi,
$$  
where $X_i={\rm dmn}(\Psi_i)$. As in the Proposition 2.2 of \cite{irie-marques-neves}, $X_i$ can be supposed to be a three-dimensional simplicial complex.

	Assume there exists a foliation $\{\gamma_t\}_{t\in [-s,s]}$ of $M$ by two-sided, smooth, embedded, closed geodesics of length $\pi$, with $\gamma_{-s}=\gamma_s$. This case can only happen if $M$ is diffeomorphic to a torus or to a Klein bottle. Let $\mathcal{R}$ be the set of varifolds of the form $|\gamma_{t_1}|+|\gamma_{t_2}|$, for $t_1,t_2\in [-s,s]$. The set $\mathcal{R}$ is compact in the varifold topology.  We also define $\mathcal{T}$ to be the set of modulo two one-cycles of the form $\gamma_{t_1}+\gamma_{t_2}$, where $t_1,t_2\in [-s,s]$. 
 The set $\mathcal{T}$ is compact in the flat topology.

\begin{prop}\label{modification.prop}
Let $T'\in \mathcal{T}$ be such that it divides $M$ into two regions of equal area. For any $\eta>0$, the sequence $\{\Psi_i\}$ can be modified  into a sequence of three-sweepouts $\{\Psi_i'\}$ that is optimal for $\omega_3$, and 
\begin{itemize}
 \item[(i)] there is a sequence $\delta_i\rightarrow 0$ such that for each $x\in {\rm dmn}(\Psi_i')$, either ${\bf F}(\Psi_i'(x), \Psi_i ({\rm dmn}(\Psi_i))\leq \delta_i$ or $\mathcal{F}(\Psi_i'(x),0)<2\eta$,
\item[(ii)]  there is  $\alpha>0$ so that
$$
{\bf F}(\Psi_i'(x),T')\geq \alpha
$$
for any $x\in {\rm dmn}(\Psi_i')$.
\end{itemize}
\end{prop}

\begin{proof}

We are going to use the following  mass-minimization property of nontrivial cycles in the set $\mathcal{T}$:
\begin{prop}\label{length.minimizing}
Let $T=\gamma_{t_1}+\gamma_{t_2}\in \mathcal{Z}_1(M,\mathbb{Z}_2)$, where $\gamma_{t_1}\neq \gamma_{t_2}$. There is $0<\delta=\delta(T)$   such that if $S\in \mathcal{Z}_1(M,\mathbb{Z}_2)$ satisfies $\mathcal{F}(S,T)\leq\delta$, then either ${\bf M}(S)>{\bf M}(T)=2\pi$ or $S\in \mathcal{T}$.
\end{prop}

\begin{proof}

Suppose, by contradiction, that the proposition is not true. Then there is a sequence $\{S_i \in  \mathcal{Z}_1(M,\mathbb{Z}_2)\}_i$ such that $S_i\rightarrow T$ in the flat topology, ${\bf M}(S_i)\leq 2\pi$ and $S_i\notin \mathcal{T}$ for each $i$.  Let $U'$ be a  neighborhood of 
$\gamma_{t_1}$, and $V'$ be a neighborhood of $\gamma_{t_2}$, such that for some open sets $U,V$ one has $\overline{U'}\subset U$, $\overline{V'}\subset V$ and $\overline{U}\cap \overline{V}= \emptyset$.  By lower semicontinuity of mass, and using that ${\bf M}(S_i)\leq 2\pi$, it follows that  ${\bf M}(S_i\llcorner U')\rightarrow \pi$ and ${\bf M}(S_i\llcorner V')\rightarrow \pi$. Hence ${\bf M}(S_i\llcorner (U'\cup V')^c)\rightarrow 0$. 

Let
$$
m_i=\inf \{{\bf M}(T'): {\rm spt}(T'-S_i)\subset (U'\cup V')^c, T'\in \mathcal{Z}_1(M,\mathbb{Z}_2)\}.
$$
By lower semicontinuity of mass there is $S_i'\in \mathcal{Z}_1(M,\mathbb{Z}_2)$, with ${\rm spt}(S_i'-S_i)\subset (U'\cup V')^c$
and ${\bf M}(S_i')=m_i$. Notice that $S_i'\llcorner (U'\cup V')=S_i\llcorner (U'\cup V')$, hence ${\bf M}(S_i'\llcorner U')\rightarrow \pi$ and ${\bf M}(S_i'\llcorner V')\rightarrow \pi$.  But  ${\bf M}(S_i')\leq {\bf M}(S_i) \leq 2\pi$.  Therefore ${\bf M}(S_i'\llcorner (U'\cup V')^c)\rightarrow 0$. Since 
$S_i'$ is mass-minimizing in $(\overline{U'}\cup\overline{V'})^c$, the monotonicity formula for geodesics implies 
${\rm spt}(S_i')\subset U \cup V$ for sufficiently large $i$. We can decompose $S_i'=S_{i,U}'+S_{i,V}'$, with $\partial S_{i,U}'=\partial S_{i,V}'=0$,  ${\rm spt}(S_{i,U}')\subset U$, ${\rm spt}(S_{i,V}')\subset V$,  $S_{i,U}'\rightarrow \gamma_{t_1}$, $S_{i,V}'\rightarrow \gamma_{t_2}$. 

We can suppose that $U$ is a foliated open set. In particular, $\partial U$ is a union of two closed geodesics.
By minimizing the length inside $U$ in the homology class of $\gamma_{t_1}$, and then applying the maximum principle, it follows that the closed geodesics of the foliation of $U$ are the only length-minimizing cycles modulo two.  The analogous statement is true for $V$ and $\gamma_{t_2}$.
Then ${\bf M}(S_{i,U}')\geq \pi$ and ${\bf M}(S_{i,V}')\geq \pi$.  Since ${\bf M}(S_i')\leq 2\pi$, it follows that ${\bf M}(S_{i,U}')= \pi$ and ${\bf M}(S_{i,V}')= \pi$. Therefore $S_{i,U}'$ and $S_{i,V}'$ are closed geodesics of the foliation.  Since $S_i\llcorner (U'\cup V')=S_i'\llcorner (U'\cup V')$ and ${\bf M}(S_i)\leq 2\pi$, it follows that $S_i=S_i'$ for sufficiently large $i$.  This is a contradiction as $S_i\notin\mathcal{T}$. This finishes the proof of the proposition.
\end{proof}

	Given $\eta>0$, we denote by $\mathcal{T}_\eta\subset \mathcal{Z}_1(M,\mathbb{Z}_2)$ the set of cycles  $T=\gamma_{t_1}+\gamma_{t_2}$ such that: if $U$ is a two-chain with $\partial U=T$, then ${\bf M}(U)\geq \eta$ and ${\bf M}(M-U)\geq \eta$. 
	
	There is $\delta'(\eta)>0$ such that if $S\in \mathcal{T}$ satisfies $\mathcal{F}(S,\mathcal{T}_\eta)\leq \delta'(\eta)$, then $S\in \mathcal{T}_{\eta/2}$.

\begin{prop}\label{length.minimizing.eta}
Let $\eta>0$. There is $0<\delta=\delta(\eta)<\delta'(\eta)$ such that if $S\in \mathcal{Z}_1(M,\mathbb{Z}_2)$ satisfies $\mathcal{F}(S,\mathcal{T}_\eta)\leq \delta$, then either ${\bf M}(S)>2\pi$ or $S\in \mathcal{T}_{\eta/2}$.
\end{prop}

\begin{proof}
Suppose that the proposition is false.  Since $\mathcal{T}_\eta$ is compact in the flat topology,  we can find a sequence $\{S_i \in  \mathcal{Z}_1(M,\mathbb{Z}_2)\}_i$ such that $S_i\rightarrow T\in \mathcal{T}_\eta$ in the flat topology, ${\bf M}(S_i)\leq 2\pi$ and $S_i\notin \mathcal{T}_{\eta/2}$ for each $i$. By applying Proposition \ref{length.minimizing} to $T$, we get that $S_i\in \mathcal{T}$ for sufficiently large $i$. 
Since $\mathcal{F}(S_i,T)<\delta'(\eta)$ for sufficiently large $i$, it follows that $S_i\in \mathcal{T}_{\eta/2}$ for such $i$
which is a contradiction. This proves the proposition.
\end{proof}

\begin{prop}\label{close.Fmetric}
Let $\eta>0$, $0<\delta<\delta(\eta)$, and $\beta>0$.  There is $0<\theta=\theta(\eta,\delta,\beta)$ such that if $S\in \mathcal{Z}_1(M,\mathbb{Z}_2)$ satisfies $\mathcal{F}(S,\mathcal{T}_\eta)\leq \delta$ and ${\bf M}(S)\leq 2\pi+\theta$, then ${\bf F}(S, \mathcal{T}_{\eta/2})< \beta.$
\end{prop}

\begin{proof}
Suppose the proposition is false, by contradiction. Then there is a sequence $\{S_i \in  \mathcal{Z}_1(M,\mathbb{Z}_2)\}_i$ such that $\mathcal{F}(S_i,\mathcal{T}_\eta)\leq \delta$, ${\bf M}(S_i)\leq 2\pi+\theta_i$, $\theta_i\rightarrow 0$,  and 
${\bf F}(S_i, \mathcal{T}_{\eta/2})\geq  \beta$
for each $i$.   By compactness of the space of flat chains with mass bounds, there is a sequence $\{S_j\}_j\subset \{S_i\}_i$ and $S\in \mathcal{Z}_1(M,\mathbb{Z}_2)$ such that $S_j\rightarrow S$ in the flat topology. Then $\mathcal{F}(S,\mathcal{T}_\eta)\leq \delta$, and ${\bf M}(S)\leq 2\pi$ by lower semicontinuity of mass.  Since $\delta<\delta(\eta)$, it follows by Proposition \ref{length.minimizing.eta} that $S\in \mathcal{T}_{\eta/2}$.
Then $\mathcal{F}(S_j,S)\rightarrow 0$ and ${\bf M}(S_j)\rightarrow {\bf M}(S)$, which implies ${\bf F}(S_j,S)\rightarrow 0$ (e.g. Section 2.1 of \cite{pitts}). This is a contradiction as ${\bf F}(S_i, \mathcal{T}_{\eta/2})\geq  \beta$ for each $i$.  This finishes the proof of the proposition.
\end{proof}

Let $0<\eta<{\rm area}(M)/5$, and $0<\delta<\delta(\eta)$. Let $\Omega_1$, $\Omega_2$ and $\Omega_3$ be open sets of $M$, with mutually disjoint closures, such that for any $T\in \mathcal{T}$, there is $1\leq k\leq 3$ so that ${\rm spt}(T)\cap \overline{\Omega}_k=\emptyset$.
We choose $\delta$ sufficiently small such that if $S\in \mathcal{Z}_1(M,\mathbb{Z}_2)$ satisfies ${\mathcal F}(S,\mathcal{T})<2\delta$, then
there is $1\leq k\leq 3$ so that $S$ cannot bisect $\Omega_k$ into two regions of equal area. This is possible because of the isoperimetric inequality.

Recall that  $X_i$ is a three-dimensional simplicial complex.
 Since $\Psi_i$ is a three-sweepout,  there is $\sigma_i\in H_3(X_i,\mathbb{Z}_2)$ such that $\Psi_i^*(\overline{\lambda})^3\cdot \sigma_i=1$. Hence $X_i$ can be replaced by $\sigma_i$ as the domain of $\Psi_i$. Therefore $X_i$ can be supposed to satisfy $\partial X_i=0$.  
 
 Let $\mathcal{R}_\eta$ be the set of varifolds induced by elements of $\mathcal{T}_\eta$. We can suppose that ${\bf C}(\{\Psi_i\})$ contains $|T'|\in \mathcal{R}_\eta$, otherwise
no modification is needed. By subdividing $X_i$ we can find three-dimensional subcomplexes $Y_i,Z_i\subset X_i$, with $X_i=Y_i\cup Z_i$, ${\rm int}(Y_i)\cap {\rm int}(Z_i)=\emptyset$, such that for any $x\in Y_i$, ${\bf F}(\Psi_i(x),\mathcal{T}_\eta)<\delta$, and for any $x\in Z_i$, ${\bf F}(\Psi_i(x),\mathcal{T}_\eta)\geq\delta/2$. In particular, $\mathcal{F}(\Psi_i(x),\mathcal{T}_\eta)<\delta$ for any $x\in Y_i$.

Let $\beta>0$, and $\theta=\theta(\eta,\delta,\beta)>0$. For sufficiently large $i$, 
$$
\sup_{x\in X_i}{\bf M}(\Psi_i(x)) \leq 2\pi+\theta.
$$
Hence, for sufficiently large $i$,  by Proposition \ref{close.Fmetric}, 
$$
{\bf F}(\Psi_i(x),\mathcal{T}_{\eta/2}) < \beta
$$
for any $x\in Y_i$. Notice that $\mathcal{T}_{\eta/2}\subset \mathcal{Z}_1(M,{\bf F},\mathbb{Z}_2)$ is a compact set. Since $\beta$ is arbitrarily small, we can use Theorem 3.8 of \cite{marques-neves-lower-bound} with $\mathcal{K}=\mathcal{T}_{\eta/2}$ to construct homotopies
$$
H_i:[0,1]\times \partial Y_i \rightarrow \mathcal{Z}_1(M,{\bf F},\mathbb{Z}_2),
$$
with $H_i(0,\cdot)=\Psi_i$, $H_i(1,x) \in \mathcal{T}_{\eta/5}$ for any $x\in \partial Y_i$, and
$$
 {\bf F}(H_i(t,x),\Psi_i(x))\leq \alpha_i\rightarrow 0
$$
for any $(t,x)\in [0,1]\times \partial Y_i$.

Since $\partial Y_i\subset Z_i$, ${\bf F}(\Psi_i(x),\mathcal{T}_\eta)\geq\delta/2$ for any $x\in \partial Y_i$. Hence for sufficiently large
$i$, $H_i(1,x)\notin\mathcal{T}_\eta$ for any $x\in \partial Y_i$.  Since $\eta<{\rm area}(M)/5$, for each $x\in \partial Y_i$ there is a unique
two-dimensional chain $U_i(x)$ with $\partial U_i(x)=H_i(1,x)$ and ${\bf M}(U_i(x))< \eta$.  Notice that $\partial U_i(x)$ is a union of two closed geodesics of the foliation.  We can define a continuous path $t\in [1,2]\mapsto U_i(t,x)$ of two-chains $U_i(t,x)\subset U_i(x)$, with $U_i(1,x)=U_i(x)$ and $U_i(2,x)=0$, by letting the geodesics in $\partial U_i(x)$ converge, as closed geodesics of the foliation, to the closed geodesic that divides in half the area of $U_i$. Along this deformation, $\partial U_i(t,x)\in \mathcal{T}$.

We can construct then a homotopy 
$$
H_i':[1,2]\times \partial Y_i \rightarrow \mathcal{Z}_1(M,{\mathcal F},\mathbb{Z}_2),
$$
continuous in the flat topology and with no concentration of mass, such that
$H_i'(1,x)=H_i(1,x)$ $H_i'(2,x)=0$ for any $x\in \partial Y_i$, and $H_i'(t,x)\in \mathcal{T}$ for any $(t,x)\in [1,2]\times \partial Y_i$.

We glue the homotopies into a map $\tilde{H}_i:[0,2]\times \partial Y_i\rightarrow \mathcal{Z}_1(M,{\mathcal F},\mathbb{Z}_2),$ with
$\tilde{H}_i(t,x)=H_i(t,x)$ if $t\in [0,1]$ and $\tilde{H}_i(t,x)=H_i'(t,x)$ if $t\in [1,2]$.  Since $\tilde{H}_i(2,x)=0$ for any $x\in \partial Y_i$,
we can identify $(2,x)\sim (2,x')$ for any $x,x'\in \partial Y_i$ so that $\tilde{H}_i$ induces a continuous map
$$
\Gamma_i: C(\partial Y_i)\rightarrow \mathcal{Z}_1(M,{\mathcal F},\mathbb{Z}_2)
$$
on the cone $C(\partial Y_i)$ over $\partial Y_i$. The map $\Gamma_i$ has no concentration of mass and 
$$
\Gamma_i(C(\partial Y_i))\subset \tilde{H}_i([0,2]\times \partial Y_i)
$$
for any $i$.

We define the spaces $\tilde{Y}_i=Y_i \cup C(\partial Y_i)$ and $\tilde{Z}_i=Z_i \cup C(\partial Y_i)$. Since $\partial (C(\partial Y_i))=\partial Y_i$, and $\partial Z_i=\partial Y_i$, it follows that $\partial \tilde{Y}_i=\partial \tilde{Z}_i=0$. The homology is with coefficients modulo two, hence
$$\tilde{Y}_i+\tilde{Z}_i=Y_i+Z_i+2C(\partial Y_i)=Y_i+Z_i=X_i.$$

We consider the map $\Psi_{i,1}:\tilde{Y}_i\rightarrow \mathcal{Z}_1(M,{\mathcal F},\mathbb{Z}_2)$ defined by $(\Psi_{i,1})_{|Y_i}=(\Psi_i)_{|Y_i}$ and  $(\Psi_{i,1})_{|C(\partial Y_i)}=\Gamma_i$, and $\Psi_{i,2}:\tilde{Z}_i\rightarrow \mathcal{Z}_1(M,{\mathcal F},\mathbb{Z}_2)$ defined by $(\Psi_{i,2})_{|Z_i}=(\Psi_i)_{|Z_i}$ and  $(\Psi_{i,2})_{|C(\partial Y_i)}=\Gamma_i$.
These maps are well defined because $H_i(0,x)=\Psi_i(x)$ for any $x\in \partial Y_i$. Notice that by modulo two cancellation of the maps on the cone, one gets
$$
(\Psi_{i,1})_*(\tilde{Y}_i)+(\Psi_{i,2})_*(\tilde{Z}_i)=(\Psi_i)_*(X_i) \in H_3(\mathcal{Z}_1(M,\mathbb{Z}_2),\mathbb{Z}_2).
$$
Since
$$
\overline{\lambda}^3 \cdot [(\Psi_i)_*(X_i)]=1,
$$
we get that either $\overline{\lambda}^3 \cdot [(\Psi_{i,1})_*(\tilde{Y}_i)]=1$ or $\overline{\lambda}^3 \cdot [(\Psi_{i,2})_*(\tilde{Z}_i)]=1$.
This implies that either   $\Psi_{i,1}$ or $\Psi_{i,2}$ is a three-sweepout.

Since ${\bf F}(\Psi_i(x),\mathcal{T}_\eta)< \delta$ for any $x\in Y_i$,  ${\bf F}(\tilde{H}_i(t,x),\Psi_i(x))\leq \alpha_i <\delta$ for any
$(t,x)\in [0,1]\times \partial Y_i$ and sufficiently large $i$, and $\tilde{H}_i(t,x)\in \mathcal{T}$ for any $(t,x)\in [1,2]\times \partial Y_i$, we
have that
$$
{\mathcal F}(\Psi_{i,1}(x),\mathcal{T})<2\delta
$$
for any $x\in \tilde{Y}_i$. Therefore by the choice of $\delta$, no $S=\Psi_{i,1}(x)$, $x\in \tilde{Y}_i$, can bisect each $\Omega_k$, $1\leq k\leq 3$ into regions of equal area. This proves that $\Psi_{i,1}$ is not a three-sweepout.
Hence $\Psi_{i,2}:\tilde{Z}_i\rightarrow \mathcal{Z}_1(M,\mathcal{F}, \mathbb{Z}_2)$ is a three-sweepout for sufficiently large $i$.

Recall $T'\in \mathcal{T}$ is such  that it divides $M$ into two regions of equal area.  Hence $T'\in \mathcal{T}_\eta$. This implies ${\bf F}(\Psi_i(x),T')\geq \delta/2$ for any $x\in Z_i$. Since  ${\bf F}(\tilde{H}_i(t,x),\Psi_i(x))\leq \alpha_i \rightarrow 0$ for any $(t,x)\in [0,1]\times\partial Y_i$, and $\partial Y_i\subset Z_i$, one has that ${\bf F}(\tilde{H}_i(t,x),T')\geq \delta/5$ for any $(t,x)\in [0,1]\times\partial Y_i$ and sufficiently large $i$. Any
$\tilde{H}_i(t,x)$, $(t,x)\in [1,2]\times \partial Y_i$, bounds a region of area less than ${\rm area}(M)/5$. It follows that for some $\alpha>0$,
$$
{\bf F}(\Psi_{i,2}(x),T')\geq \alpha
$$
for any $x\in \tilde{Z}_i$ and sufficiently large $i$.

Since $\lim_{i\rightarrow \infty}\sup_{x\in \tilde{Z}_i}{\bf M}(\Psi_{i,2}(x))\leq  2\pi$, and $\Psi_{i,2}$ is a three-sweepout, we have that
$\sup_{x\in \tilde{Z}_i}{\bf M}(\Psi_{i,2}(x))\rightarrow  2\pi$ and the sequence $\{\Psi_{i,2}\}_i$ is optimal for $\omega_3$.
The maps $\Psi_{i,2}$ are continuous in the flat topology and have no concentration of mass.  By Corollary 2.13 of \cite{liokumovich-marques-neves}, for each $i$ there is a sequence of continuous maps   $\Psi_{i,2}^{(j)}:\tilde{Z}_i\rightarrow\mathcal{Z}_1(M,{\bf M},\mathbb{Z}_2)$ 
which are three-sweepouts and such that for some sequence $\delta_{i,j}\rightarrow 0$,
$$
\sup_{x\in \tilde{Z}_i}{\bf M}(\Psi_{i,2}^{(j)}(x))\leq \sup_{x\in \tilde{Z}_i}{\bf M}(\Psi_{i,2}(x))+\delta_{i,j},
$$
and
$$
\sup_{x\in \tilde{Z}_i}\mathcal{F}(\Psi_{i,2}^{(j)}(x), \Psi_{i,2}(x))\leq \delta_{i,j}.
$$
Since the restriction of the map $\Psi_{i,2}$ to $Z_i \cup ([0,1]\times \partial Y_i)$ is continuous in the ${\bf F}$-metric,  by the proof of Proposition A.5 of \cite{marques-neves-index} we can suppose  that
$$
\sup_{x\in Z_i \cup ([0,1]\times \partial Y_i)}{\bf F}(\Psi_{i,2}^{(j)}(x), \Psi_{i,2}(x))\rightarrow 0
$$
as $j\rightarrow \infty$. If $x\in \big(([1,2]\times \partial Y_i)/\sim\big) \subset C(\partial Y_i)$, then $\mathcal{F}(\Psi_{i,2}(x),0)<\eta$. Hence, for
sufficiently large $j$,
$$
\sup_{x\in \big(([1,2]\times \partial Y_i)/\sim\big) }{\mathcal F}(\Psi_{i,2}^{(j)}(x), 0)<2\eta.
$$
By a diagonal argument we can extract a sequence of three-sweepouts $\{\Psi_i'=\Psi_{i,2}^{(j_i)}\}_i$ which is optimal for $\omega_3$ and satisfies item (i) of the proposition.

Suppose  there is a sequence $\{k\}\subset \{i\}$ and $x_k\in \tilde{Z}_k$ such that
$${\bf F}(\Psi_k'(x_k),T')\rightarrow 0.$$
Then it follows that $\mathcal{F}(\Psi_{k,2}(x_k),T')\rightarrow 0$ and ${\bf M}(\Psi_{k,2}(x_k))\rightarrow 2\pi$.
This implies ${\bf F}(\Psi_{k,2}(x_k),T')\rightarrow 0$, which is a contradiction.  Therefore for some $\tilde{\alpha}>0$, 
$$
{\bf F}(\Psi_i'(x),T')\geq \tilde{\alpha}
$$
for any $x\in \tilde{Z}_i$ and sufficiently large $i$. This finishes the proof of the Proposition \ref{modification.prop}.

\end{proof}

{\bf Claim:} The number of foliations of $M$ by two-sided, smooth, embedded, closed geodesics of length $\pi$, is finite.

\medskip

Suppose, by contradiction, that the claim is false. Then  there is a sequence of   foliations $\{\gamma_t^{(j)}\}_{t\in [-s,s]}$ of $M$  by two-sided, smooth, embedded, closed geodesics of length $\pi$, with $\gamma_{-s}^{(j)}=\gamma_s^{(j)}$, such that  
$\{\gamma_t^{(j)}\}_{t}\neq \{\gamma_t^{(k)}\}_{t}$ as foliations for any $j\neq k$. Then by passing to a subsequence we can suppose that $\gamma_0^{(j)}$ converges to a two-sided  simple closed geodesic with multiplicity one, or to a one-sided simple closed geodesic with multiplicity two. The latter case cannot happen as $\gamma_0^{(j)}$ is homologically nontrivial. Hence there is a two-sided embedded closed geodesic $\gamma$ such that $\gamma_0^{(j)}\rightarrow \gamma$ smoothly.  Suppose that for each $j$ there is $k_j\geq j$ and $s_j\in[-s,s]$ such that the geodesic $\gamma_{s_j}^{(k_j)}$ intersects any leaf of $\{\gamma_t^{(j)}\}_{t\in [-s,s]}$. Since $\gamma_{s_j}^{(k_j)}$ is disjoint from $\gamma_{0}^{(k_j)}$, there 
is a sequence $\{i\}\subset \{j\}$ such that $\gamma_{s_i}^{(k_i)}$ converges smoothly to a simple closed geodesic $\gamma'$ that is either disjoint from $\gamma$ or equal to $\gamma$. Since $\gamma_{s_i}^{(k_i)}$ intersects $\gamma_0^{(i)}$, and $\gamma_0^{(i)}\rightarrow \gamma$, it follows that $\gamma'=\gamma$. Hence $\gamma_{s_i}^{(k_i)}\rightarrow \gamma$. This is a contradiction as $\gamma_{s_i}^{(k_i)}$
also intersects the geodesic $\gamma_{t'}^{i}$ such that  $\gamma_0^{i}\cup \gamma_{t'}^{i}$ divides $M$ into two regions of equal area.
Therefore there is $j$ such that for any $k\geq j$, any geodesic $\gamma_t^{(k)}$ is disjoint from one of the leaves  $\{\gamma_t^{(j)}\}_{t\in [-s,s]}$. By the maximum principle, for any $k\geq j$ any leaf of the foliation $\{\gamma_t^{(k)}\}_t$ is  a leaf of the foliation $\{\gamma_t^{(j)}\}_t$. Hence there is $j\in \mathbb{N}$ such that for any $k\geq j$, $\{\gamma_t^{(k)}\}_t=\{\gamma_t^{(j)}\}_t$ as foliations. This is a contradiction and hence the claim is proved.

\medskip

For any foliation of $M$ by two-sided, smooth, embedded, closed geodesics of length $\pi$, if there are such foliations, we can choose a cycle $T$ which is the sum of two leaves and divides the surface into regions of equal area. Since the number of such foliations is finite, we can denote the set of such cycles by $\{T_1,\dots,T_{k}\}$, where $k\in \mathbb{N}$. By $k$ repeated applications of Proposition \ref{modification.prop}, we can suppose that the sequence 
$\{\Psi_i\}\subset \mathcal{P}_3$ is such that for some $\alpha>0$,
$$
{\bf F}(\Psi_i(x),T_{k'})\geq \alpha
$$
for any $x\in X_i$ and $1\leq k'\leq k$. We can also suppose that the sequence is pulled-tight, so any $V\in {\bf C}(\{\Psi_i\}_i)$ is stationary. This is because for any $\xi>0$ the pull-tight argument can be done  by not changing each $\Psi_i(x)$ more than $\xi$ in the ${\bf F}$-metric. 

Let $\mathcal{V}$ be the set of stationary geodesic networks with integer multiplicities and mass equal to $2\pi$. 
For $V\in \mathcal{V}$, let $T_V\in {I}_1(M,\mathbb{Z}_2)$ be the sum of the edges of $V$ which have an odd multiplicity.
If $S\in \mathcal{Z}_1(M,\mathbb{Z}_2)$, we define
$$
d(S,\mathcal{V})=\inf_{V\in \mathcal{V}} \{\mathcal{F}(S,T_V)+{\bf F}(|S|,V) \}.
$$
Notice that, for $T\in \mathcal{Z}_1(M,\mathbb{Z}_2)$,
$$
d(T,\mathcal{V}) \leq d(S,\mathcal{V})+{\bf F}(T,S).
$$

For each $\delta>0$, we can suppose by refining $X_i$ that for any $j$-dimensional simplex $t\in X_i$,
$$
{\bf F}(\Psi_i(x),\Psi_i(x'))< \delta/5
$$
for every $x,x'\in t$.
Let $Z_i$ be the union of the simplices $t\in X_i$ such  that  
$d(\Psi_i(z),\mathcal{V})\leq 2\delta/5$ for every $z\in t$, and $Y_i$ to be  the union of the simplices
$t\in X_i$ such that  $d(\Psi_i(y),\mathcal{V})\geq \delta/5$ for every $y\in t$. It follows that  $Z_i,Y_i$ are subcomplexes of $X_i$ with
$X_i=Z_i \cup Y_i$. 

Suppose there is  a sequence $\{k\}\subset \{i\}$ such that $({\Psi_k}_{|Y_k})^*(\overline{\lambda}) \neq 0 \in H^1(Y_k,\mathbb{Z}_2)$.
Since $\omega_1(M,g)=2\pi$ and $\limsup_{k\rightarrow \infty} \sup_{y\in Y_k}{\bf M}(\Psi_k(y))\leq 2\pi$, we would have that $\{{\Psi_k}_{|Y_k}\}\subset \mathcal{P}_1$ is an optimal sequence for $\omega_1$. 
We can find  $\sigma_k\subset Y_k$, $\sigma_k$ homeomorphic to $S^1$ ($\sigma_k\approx S^1$), such that $\Psi_k'=(\Psi_k)_{|\sigma_k}$ is a one-sweepout for each $k$. 

Let $V\in {\bf C}(\{\Psi_k'\})$. If $V$ is not $\mathbb{Z}_2$-almost minimizing in annuli, we can proceed as in Parts 1 and 2 of the proof of Theorem 4.10 in \cite{pitts}. This gives a set of concentric annuli so that if a  cycle $T$ is such that $T$ is sufficiently close to $V$, say if ${\bf F}(|T|,V)<\eta=\eta(V)$, then $T$  admits mass decreasing deformations supported in each annulus. 

 If $V$ is $\mathbb{Z}_2$-almost minimizing in annuli, since  $V$ is stationary it follows by Pitts (\cite{pitts}, 3.13) that $V$ is a stationary integral varifold ($V\in \mathcal{V}$).
Let $\mathcal{T}_V$ be the set of cycles $T\in \mathcal{Z}_1(M,\mathbb{Z}_2)$ such that ${\rm spt}(T)\subset {\rm spt}(V)$. By the Constancy Theorem of flat chains, if $T\in \mathcal{T}_V$ then either $T=0$ or $T$ is a sum of edges of ${\rm support}(V)$. Hence $\mathcal{T}_V$ is a finite set.
 If $T\in \mathcal{T}_V$ is such that $T\neq T_V$, then there is an edge $\gamma=\gamma(T,V)$ of $V$ with integer multiplicity $k\in \mathbb{N}$ such that $T\llcorner \gamma= r\cdot \gamma$, $r \in \{0,1\}$, with $r \neq k \,  {\rm mod} \, 2$. The proof of the Proposition 4.10 of \cite{marques-neves-lower-bound} is local, so it can be applied with $\Sigma=T$ to points $p\in \gamma$ near the midpoint of the edge. That proposition can be used here even if $(n+1)=2$ since the edge $\gamma$ is smooth.

For $\rho>0$, there is $0<\eta<\rho$ such that if $S\in \mathcal{Z}_1(M,\mathbb{Z}_2)$ is such that ${\bf F}(|S|,V)<\eta$ then for some $T\in \mathcal{T}_V$, one has $\mathcal{F}(S,T)<\rho$. If $2\rho<\delta/5$, and $S=\Psi_k'(x_k)$, $x_k\in \sigma_k$, then $T\neq T_V$ by the definition of $Y_k$. We choose $\rho=\rho(V)>0$ sufficiently small so that Proposition 4.10 of \cite{marques-neves-lower-bound} can be applied to $V$ and $\Sigma=T$ near the
midpoint of the edge $\gamma(T,V)$ for any $T\in \mathcal{T}_V$, $T\neq T_V$.  This gives for each  edge of $V$ a set of annuli centered at its midpoint such that if $S=|\Psi_k'(x_k)|$, $x_k\in \sigma_k$, satisfies ${\bf F}(|S|,V)<\eta(V)$  then for one of the edges $S$ admits mass decreasing deformations supported in each annulus.

Since ${\bf C}(\{\Psi_k'\}_k)$ is a compact set in the varifold topology, we can find $\{V_1, \dots, V_h\} \subset {\bf C}(\{\Psi_k'\}_k)$ such that
${\bf C}(\{\Psi_k'\}) \subset \cup_{k=1}^h B_{\bf F}(V_k,\eta(V_k)/2)$. Hence there is $\xi>0$ such that for sufficiently large $k$, if $x\in \sigma_k$ is such that ${\bf M}(\Psi_k'(x))\geq 2\pi-\xi$, then  $|\Psi_k'(x)|\in \cup_{k=1}^h B_{\bf F}(V_k,\eta(V_k))$. Therefore $\Psi_k'(x)$ admits mass decreasing deformations in each annulus of one of the sets of concentric annuli described. The number of such sets is finite. We can suppose by approximation that the maps $\Psi_k'$ are continuous in the mass topology. We identify the domain $\sigma_k$ of $\Psi_k'$ with $X=S^1$. Hence we can apply Theorem 4.6 of \cite{marques-neves-lower-bound} with the interpolation machinery and $\mathcal{W}=\emptyset$   to deform $\{\Psi_k'\}_k$ into a sequence of one-sweepouts $\{\Psi_k''\}$ such that
$$
\limsup_{k\rightarrow \infty} \sup_{x\in \sigma_k} {\bf M}(\Psi_k''(x))< 2\pi,
$$
which is a contradiction.

Hence for sufficiently large $i$,
 $({\Psi_i}_{|Y_i})^*(\overline{\lambda}) = 0$. By the properties of the cup product,  for such $i$
$$
({\Psi_i}_{|Z_i})^*(\overline{\lambda}^2)  \neq 0.
$$

Since $\delta>0$ is arbitrary, this proves that there is a sequence $\{\Phi_i\}_{i\in \mathbb{N}} \subset \mathcal{P}_2$ 
such that
\begin{equation}\label{flat.varifold}
\lim_{i\rightarrow \infty} \sup_{x\in X_i} d(\Phi_i(x),\mathcal{V})=0,
\end{equation}
where $X_i={\rm dmn}(\Phi_i)$. We can suppose the $X_i$ are connected and two-dimensional. Moreover, if $M$ admits foliations as discussed above, by construction of the three-sweepouts $\{\Psi_i\}\subset \mathcal{P}_3$ we may also assume that
$$
{\bf F}(\Phi_i(x),T_{k'})\geq \alpha
$$
for any $x\in X_i$ and $1\leq k'\leq k$.
Notice that this sequence is  optimal for $\omega_2$ because $\omega_2(M,g)=2\pi$. The critical set ${\bf C}(\{\Phi_i\})$ coincides with the image set ${\bf \Lambda}(\{\Phi_i\})$, which is the set of varifold limits of sequences $\{|\Phi_{i_j}(x_{i_j})|\}$, and ${\bf C}(\{\Phi_i\})\subset \mathcal{V}$.

	We now begin the analysis of the structure of the varifolds belonging to the critical set of the optimal sequence of two-sweepouts constructed above.

\begin{prop}\label{no.disjoint.geodesic}
Let $V\in {\bf C}(\{\Phi_i\})\subset \mathcal{V}$. There is no  smooth, embedded, closed geodesic  disjoint from ${\rm support}(V)$. Moreover, any component $\tilde{\gamma}$ of the support of $V$ that is a smooth, embedded, closed geodesic  is part of a local foliation whose leaves have  curvature vector pointing strictly away from  $\tilde{\gamma}$.
\end{prop}

\begin{proof}
The proof is by contradiction. Suppose $\gamma$ is a smooth, two-sided, closed geodesic  disjoint from ${\rm spt}(V)$. The geodesic $\gamma$ is part of a local foliation such that the leaves have vanishing, positive or negative curvature. We use a local orientation to give a sign to the curvature function, or alternatively we can say that each leaf has either vanishing curvature vector, curvature vector pointing strictly away from $\gamma$ or curvature vector pointing strictly towards $\gamma$. Let $\{\gamma_t\}_{t\in (-2\delta,2\delta)}$ be such a foliation, $\gamma_0=\gamma$.  We can suppose that $\overline{\bigcup \gamma_t}$ is disjoint from ${\rm spt}(V)$.

Suppose that, for some $-\delta<t'<\tilde{t}<\delta$, the complement $\Omega$ of $\big(\cup_{t\in [t',\tilde{t}]} \gamma_t\big)$ has strictly convex boundary. By Theorem 5 of \cite{white-maximum-principle}, there exists $\eta>0$ such that every stationary $1$-varifold of $\Omega$ is disjoint from $B_\eta(\partial \Omega)$.
Consider $K=\overline{\Omega}\setminus B_\eta(\partial\Omega)$ and let $\mathcal{K}$ be the set of elements in $\mathcal{V}$ with support contained in $K$. Let $\zeta>0$ be as in Lemma \ref{convergence.supports} for $K$, $\mathcal{K}$, $\eta$.  If $W$ is an element of $\mathcal{V}$ 
with support contained in $\overline{\Omega}$, then
${\rm spt}(W)\subset K$ ($W\in \mathcal{K}$).
Therefore any element of $\mathcal{V}$ that is $\zeta$-close in the ${\bf F}$-metric to $W$  must have support contained
in $B_\eta(K)\subset \overline{\Omega}$.  Hence its support must be contained in $K$. 

 Since $V\in \mathcal{K}$, we can conclude that any  $\zeta$-chain $\{V_1,\dots,V_l\}$ of one-varifolds in $M$  with $V_1=V$ and $V_i\in \mathcal{V}$ for every $1\leq i\leq l$ must satisfy ${\rm spt}(V_i)\subset \overline{\Omega}$ for every $i$. Because
 $V \in {\bf C}(\{\Phi_i\})$, there exist sequences $\{j\} \subset \{i\}$ and $\{x_j\in X_j\}$ with $|\Phi_j(x_j)|$ converging to $V$ in the
 ${\bf F}$-metric. For sufficiently large $j$ we have that ${\bf F}(|\Phi_j(x)|,\mathcal{V})\leq \zeta/5$ for every $x\in X_j$. Given
 $x'\in X_j$, the connectedness of $X_j$ implies that we can find $\{y_1,\dots, y_{q_j}\}\subset X_j$ with $y_1=x_j$, $y_{q_j}=x'$ and
 ${\bf F}(\Phi_j(y_k),\Phi_j(y_{k+1}))\leq \zeta/5$ for every $1\leq k\leq q_j-1$. We choose $V_{j,k}\in \mathcal{V}$ so that
 ${\bf F}(|\Phi_j(y_k)|,V_{j,k})\leq \zeta/5$ and $V_{j,1}=V$. Hence $\{V_{j,1}, \dots, V_{j,q_j}\}$ is a $\zeta$-chain in $\mathcal{V}$ starting at $V$. We conclude that ${\rm spt}(V_{j,k})\subset \overline{\Omega}$ for every $1\leq k\leq q_j$. 
 
  This proves that for every such $j$ and every  $x'\in X_j$ the varifold $|\Phi_j(x')|$ is $\zeta/5$-close in the ${\bf F}$-metric to an element of $\mathcal{V}$ with support contained in $\overline{\Omega}$.  For any $\gamma>0$ and $t'<s'<\tilde{s}<\tilde{t}$, we can choose $\zeta$ sufficiently small so that 
 $$
 {\bf M}(\Phi_j(x') \llcorner \big(\cup_{t\in (s',\tilde{s})} \Sigma_t\big))<\gamma
 $$
 for every $x'\in X_j$. If $\gamma$ is sufficiently small, this contradicts the property that $\Phi_j$ is a one-sweepout (since it is a two-sweepout), because of the isoperimetric inequality and of the fact that any one-sweepout contains a cycle that divides a prescribed region in two parts of equal area. This proves that either $\gamma_t$,  $t\in (-\delta,\delta)$, has curvature vector pointing towards $\gamma$ or $\gamma_t$  has curvature of the same sign for $t\in (-\tilde{\delta},\tilde{\delta})$, with $\tilde{\delta}<\delta$ (in this case either any $\gamma_t$ has curvature vector pointing towards $\gamma_{\tilde{\delta}}$, or any $\gamma_t$ has
 curvature vector pointing towards $\gamma_{-\tilde{\delta}}$).
 
 Suppose then that for any $t\in (-\delta,\delta)$ the curve $\gamma_t$ has  curvature vector pointing towards $\gamma$. By the maximum principle of White \cite{white-maximum-principle} (Theorem 4), any stationary integral one-varifold $V'$ with   ${\rm spt}(V') \subset \big(\cup_{t\in (-\delta,\delta)} \gamma_t\big)$ must be a sum of leaves each of which is a closed geodesic.  Let $\eta>0$ be such that 
$$B_\eta(\gamma)\subset \big(\cup_{t\in (-\delta,\delta)} \gamma_t\big),
$$
$\lambda=2\,\omega_1(M,g)$ and $\tilde{\eta}$ be as in Proposition \ref{intersects.must.be.contained.geodesic}. Therefore any stationary integral one-varifold $\tilde{V}$, with ${\rm spt}(\tilde{V})$ disjoint from $\gamma$, with ${\bf M}(\tilde{V})\leq \lambda$, that intersects $B_{\tilde{\eta}}(\gamma)$ must be contained in $B_\eta(\gamma)$ and hence be a sum of  leaves $\gamma_t$ each of which is a closed geodesic.

Let $0<\tilde{\delta}<\delta$ be such that $\big(\cup_{t\in [-\tilde{\delta},\tilde{\delta}]} \gamma_t\big)\subset B_{\tilde{\eta}}(\gamma)$.
Let $\mathcal{V}'\subset \mathcal{V}$ be the set of varifolds $W\in \mathcal{V}$ such that every connected component of ${\rm spt}(W)$ that intersects 
$\big(\cup_{t\in [-\tilde{\delta},\tilde{\delta}]} \gamma_t\big)$ is a leaf. 
We claim that $\mathcal{V}'$ is open and closed in $\mathcal{V}$. Let $W\in \mathcal{V}'$, and consider a sequence $\{W_j\}\subset \mathcal{V}$ that converges to $W$. Each component of $W$ is either disjoint from  $\big(\cup_{t\in [-\tilde{\delta},\tilde{\delta}]} \gamma_t\big)$ or coincides with a leaf in $\big(\cup_{t\in [-\tilde{\delta},\tilde{\delta}]} \gamma_t\big)$. By Lemma \ref{convergence.supports},  there exists $\alpha>0$ such that, for sufficiently large $j$, ${\rm spt}(W_j)$ is disjoint from 
$$\big(\cup_{t\in [-\tilde{\delta}-2\alpha,-\tilde{\delta}-\alpha]} \gamma_t\big) \cup \big(\cup_{t\in [\tilde{\delta}+\alpha,\tilde{\delta}+2\alpha]} \Sigma_t\big),
$$
and $-\delta<-\tilde{\delta}-2\alpha< \tilde{\delta}+2\alpha<\delta$.
Therefore each connected component of ${\rm spt}(W_j)$ that intersects $\big(\cup_{t\in [-\tilde{\delta},\tilde{\delta}]} \gamma_t\big)$ must be contained in
$\big(\cup_{t\in [-\tilde{\delta}-\alpha,\tilde{\delta}+\alpha]} \gamma_t\big)$ and hence be a leaf. This proves $\mathcal{V'}$ is open.  Now suppose $\{W_j\}\subset \mathcal{V}'$ is a sequence converging to $W\in \mathcal{V}$. By the monotonicity formula for stationary varifolds, there is a bound on the number of connected components of $W_j$ that is uniform in $j$. Let $\beta>0$ be such that $\big(\cup_{t\in [-\tilde{\delta}-\beta,\tilde{\delta}+\beta]} \gamma_t\big)\subset B_{\tilde{\eta}}(\gamma)$.   If a connected component of $W_j$ is disjoint from $\big(\cup_{t\in [-\tilde{\delta},\tilde{\delta}]} \gamma_t\big)$ but intersects $\big(\cup_{t\in [-\tilde{\delta}-\beta,\tilde{\delta}+\beta]} \gamma_t\big)$, it must be a leaf (since it is disjoint from $\gamma$). Each connected component of $W_j$ is either disjoint from $\big(\cup_{t\in [-\tilde{\delta}-\beta,\tilde{\delta}+\beta]} \gamma_t\big)$ or is a leaf $\gamma_t$.
This implies each connected component of $W$ that intersects $\big(\cup_{t\in [-\tilde{\delta},\tilde{\delta}]} \gamma_t\big)$ is a leaf. Hence $W\in \mathcal{V'}$. This proves that $\mathcal{V'}$ is closed. 

Since $\mathcal{V}$ is compact, $\mathcal{V}'$ is compact. Therefore there exists $\zeta>0$ such that if $W\in \mathcal{V}$ satisfies ${\bf F}(W,\mathcal{V}')\leq \zeta$ then $W\in \mathcal{V}'$. Consider $\{V_1,\dots,V_l\} \subset \mathcal{V}$ be a $\zeta$-chain with $V_1=V$. We have $V\in \mathcal{V}'$, since ${\rm spt}(V)$ is disjoint from $\big(\cup_{t\in [-\tilde{\delta},\tilde{\delta}]} \gamma_t\big)$. Hence $V_i\in \mathcal{V}'$ for every $1\leq i\leq l$. Suppose that for some interval $[a,b]\subset (-\tilde{\delta},\tilde{\delta})$, one has that $\gamma_t$ is not a closed geodesic for any $t\in [a,b]$.
It follows
that ${\rm spt}(V_i)$ is disjoint from $\big(\cup_{t\in [a,b]} \gamma_t\big)$ for every $1\leq i\leq l$. As before this contradicts the fact that $\Phi_j$ is a one-sweepout. Therefore we have proved that $\gamma_t$ is a closed geodesic for any $t\in [-\tilde{\delta},\tilde{\delta}]$.

  In the case that we can suppose $\gamma_t$, for $t\in (-\delta,\delta)$, has curvature of the same sign, again by the maximum 
  principle of \cite{white-maximum-principle}, any stationary integral varifold with support contained in  $\big(\cup_{t\in (-\delta,\delta)} \gamma_t\big)$ must be a sum of  leaves each of which is a closed geodesic. We can proceed as before to get a contradiction if there is no foliation of $\gamma$ by closed geodesics.

   Notice that these arguments can be applied to a component
$\tilde{\gamma}$ of the support of $V$ that is a smooth, closed geodesic, two-sided or one-sided. They prove then that  $\tilde{\gamma}$  is part of a local foliation such that the curvature vector of any leaf points away from $\tilde{\gamma}$.  This includes the case in which $\tilde{\gamma}$ is part of a local foliation by closed geodesics.
  
	We proved that  $\gamma$ is part of a local foliation $\{\gamma_t\}_{t\in [-s,s]}$ by closed geodesics. These closed geodesics have the same length. The foliation can be extended until it touches ${\rm spt}(V)$, in which case the extreme leaf coincides with a connected component of 
${\rm spt}(V)$ by the maximum principle, or until the closed geodesics converge to a one-sided closed geodesic with multiplicity two. This extension can be done on both sides of $\gamma$. As the curve $\gamma$ is not necessarily separating, it could be that the foliations glue smoothly on both ends into a foliation of the whole surface. If $\gamma_s$ is a two-sided closed geodesic in ${\rm spt}(V)$, the foliation can be extended to a foliation $\{\gamma_t\}_{t\in [-s,s+\eta]}$ in which the curvature vector of any leaf points towards $\gamma_{s+\eta}$.  By proceeding as before we get that any $\gamma_t$, $t\in [-s,s+\eta]$, is a closed geodesic. Therefore the foliation
by closed geodesics can be extended past any smooth, two-sided, closed geodesic in ${\rm spt}(V)$.

	This gives  a foliation $\{\gamma_t\}_{t\in [-s,s]}$, $\gamma_0=\gamma$, of $M$ by closed geodesics  such that either both $\gamma_{-s}$ and $\gamma_s$ are one-sided closed geodesics covered with multiplicity two, or $\gamma_{s}=\gamma_{-s}$. The varifold $V$ is an integral linear combination of leaves.

Suppose that both $\gamma_{-s}$ and $\gamma_s$ are one-sided closed geodesics covered with multiplicity two, in which case $M$ is diffeomorphic to a Klein bottle.  We denote by $([a,b]/\sim)$ the interval $[a,b]$ with the endpoints identified. Since $t\in ([-1,1]/\sim) \mapsto  \gamma_{ts}$ is   a one-sweepout of $M$, where $\gamma_{-s}$ and $\gamma_s$ are considered with multiplicity two, it follows then that  $||V||(M)=\omega_1(M,g)\leq l(\gamma)$. It follows from the previous arguments that the set $\mathcal{V}'\subset \mathcal{V}$ of varifolds whose connected components are leaves of the foliation is open and closed. By considering $\zeta$-chains of varifolds in $\mathcal{V}$ starting at $V\in \mathcal{V}'$,
we can argue that any element of ${\bf C}(\{\Phi_i\})$ is in $\mathcal{V}'$. But each $\Phi_i$ is a two-sweepout, hence for any $p$, $q\in M$ there is 
$\tilde{V}\in {\bf C}(\{\Phi_i\})$ with $p$, $q \in {\rm spt}(\tilde{V})$. By choosing $p$, $q$ in different two-sided leaves of the foliation, we get that $||\tilde{V}||(M)\geq 2l(\gamma)$. This is a contradiction, since $||\tilde{V}||(M)=\omega_1(M,g)\leq l(\gamma)$.

Suppose then that $\gamma_{s}=\gamma_{-s}$. Since $t\in ([-1,1]/\sim) \mapsto  \gamma_0+\gamma_{ts}$ is   a one-sweepout of $M$, it follows that $||V||(M)=\omega_1(M,g)\leq 2l(\gamma)$. As before we get that any element of ${\bf C}(\{\Phi_i\})$ is an integral linear combination of leaves. Since each $\Phi_i$ is a two-sweepout, for any $p$, $q\in M$ there is 
$V'\in {\bf C}(\{\Phi_i\})$ with $p$, $q \in {\rm spt}(V')$. Hence $||V||(M)=||V'||(M)\geq 2l(\gamma)$. Therefore $\omega_1(M,g)=2l(\gamma)$,  and every $V'\in {\bf C}(\{\Phi_i\})$ is a sum of two leaves. In particular, the leaves of the foliation have length $\pi$.

There is $T_{k'}\in \{T_1,\dots,T_{k}\}$ that is of the form $\gamma_{t_1}+\gamma_{t_2}$. By choosing $p\in \gamma_{t_1}$ and $q\in \gamma_{t_2}$, we get that $|T_{k'}|\in {\bf C}(\{\Phi_i\})$.  Recall that 
$$
{\bf F}(\Phi_i(x),T_{k'})\geq \alpha
$$
for any $x\in X_i$, by construction. Let $D$, $D'$ be disks centered at $p$, $q$, respectively, such that no leaf of the foliation intersects $D$ and $D'$.  Since $\Phi_i$ is a two-sweepout, there is $x_i\in X_i$ such that 
$\Phi_i(x_i)$ bisects $D$ and $D'$ in regions of equal area.  There is $\{j\}\subset\{i\}$ such that  $|\Phi_j(x_j)|$ converges 
in varifold sense to a varifold of the form $|\gamma_{t_1'}|+|\gamma_{t_2'}|$. Since each $\gamma_t$ is homologically nontrivial, by the Constancy Theorem any convergent subsequence $\{\Phi_{k}(x_{k})\}_{k}\subset \{\Phi_{j}(x_{j})\}_{j}$ in the flat topology satisfies  either 
$\mathcal{F}(\Phi_{k}(x_{k}), \gamma_{t_1'}+\gamma_{t_2'})\rightarrow 0$ or $\mathcal{F}(\Phi_{k}(x_{k}), 0)\rightarrow 0$. But the latter is not possible by the bisecting property. Hence  $\Phi_j(x_j)\rightarrow \gamma_{t_1'}+\gamma_{t_2'}$ in the flat topology. By letting $D\rightarrow p$ and $D' \rightarrow q$, we can by using a diagonal argument to find $\{j\}\subset \{i\}$ and $\{x_j\in X_j\}_j$ such that ${\bf F}(\Phi_j(x_j),T_{k'})\rightarrow 0$, which is a contradiction. 

	Hence there is no two-sided, embedded closed geodesic disjoint from ${\rm spt}(V)$. If there is a one-sided closed geodesic disjoint from ${\rm spt}(V)$, the foliation arguments give that there are two-sided closed geodesics which converge to $\gamma$.  This implies there is a two-sided closed geodesic disjoint from ${\rm support}(V)$, which is a contradiction. 

	If $\tilde{\gamma}$ is a component of ${\rm spt}(V)$ which is a smooth, embedded, closed geodesic, we argued that it is part of a local foliation so that each leaf has curvature vector pointing away from $\tilde{\gamma}$. The curvature vector of any leaf near $\tilde{\gamma} $ which is different from $\tilde{\gamma}$ has to point strictly away from it, as there is no closed geodesic disjoint from ${\rm spt}(V)$. This finishes the proof of the proposition. 
\end{proof}

\begin{prop}\label{connected.support.geodesic}
Let $V\in {\bf C}(\{\Phi_i\})\subset \mathcal{V}$. Then ${\rm spt}(V)$ is connected. 
\end{prop}

\begin{proof}
Suppose that ${\rm spt}(V)$ is not connected. Let $K$ be a connected component of ${\rm spt}(V)$. Let $\Omega_1,\dots,\Omega_j$ be the connected components of $M \setminus K$. Then there is $1\leq i \leq j$ such that ${\rm spt}(V)$ intersects $\Omega_i$. Let $\alpha$ be a connected component of $\partial \Omega_i$, and choose $\sigma$ the connected component of $\Omega_i \setminus {\rm spt}(V)$ such that
$\alpha \subset \partial\sigma$. In particular, ${\rm spt}(V)$ does not intersect $\sigma$.

Any component
of $\sigma$ that has corners can be smoothed into a strictly convex curve by   perturbing inside $\sigma$. By Proposition
\ref{no.disjoint.geodesic}, the smooth components of $\sigma$ can also be approximated by strictly convex curves inside $\sigma$. Hence there is a strictly convex domain $\Omega\subset \sigma$ such that $\partial \Omega$ is disconnected.  Then applying the curve shortening flow \cite{grayson}
to any component of $\partial \Omega$ gives a smooth closed geodesic $\gamma\subset \Omega$. This closed geodesic is disjoint from
${\rm spt}(V)$ and hence contradicts Proposition \ref{no.disjoint.geodesic}.
Therefore ${\rm spt}(V)$ is connected.
\end{proof}

\begin{prop}\label{flat.sequences}
Let $V\in {\bf C}(\{\Phi_i\})$. Then $T_V\in \mathcal{Z}_1(M,\mathbb{Z}_2)$, and  there is $\{j\}\subset \{i\}$ and $\{x_j\in X_j\}_j$ such that $\mathcal{F}(\Phi_j(x_j),T_V)+ {\bf F}(|\Phi_j(x_j)|,V)\rightarrow 0$.
\end{prop}

\begin{proof}
Let $V\in {\bf C}(\{\Phi_i\})$. Then $V$ is a stationary geodesic network with integer multiplicities. Recall that $T_V\in I_1(M,\mathbb{Z}_2)$ is the sum of the edges of $V$ which have an odd multiplicity. Let $\{j\}\subset \{i\}$ and $\{x_j\in X_j\}_j$ such that ${\bf F}(|\Phi_j(x_j)|,V)\rightarrow 0$. By (\ref{flat.varifold}), there is $\{V_j\}\subset \mathcal{V}$ such that $$\mathcal{F}(\Phi_j(x_j),T_{V_j})+{\bf F}(|\Phi_j(x_j)|,V_j)\rightarrow 0.$$
Since $\mathcal{F}(\partial S, \partial T)\leq \mathcal{F}(S,T)$ for any $S,T\in I_1(M,\mathbb{Z}_2)$, and $\partial \Phi_j(x_j)= 0$, it follows that
$\partial T_{V_j}\rightarrow 0$ in the flat topology. Also $V_j\rightarrow V$ in the varifold topology. Since each $V_j$ is a stationary integral varifold, the conditions of Theorem 1.1 of White \cite{white-currents-chains} apply. Therefore $T_{V_j}\rightarrow T_V$ in  the flat topology. Hence $\mathcal{F}(\Phi_j(x_j),T_V)\rightarrow 0$, and $T_V\in \mathcal{Z}_1(M,\mathbb{Z}_2)$.  This proves the proposition.
\end{proof}

\begin{prop}\label{flat.convergence}
 For each $\delta>0$, there are $\eta>0$ and $\tilde{i}\in \mathbb{N}$ such that if $j\geq \tilde{i}$, and ${\bf F}(|\Phi_j(x)|,V)<\eta$ for some $x\in X_j$ and $V\in {\bf C}(\{\Phi_i\})$, then $\mathcal{F}(\Phi_j(x),T_V)<\delta$. 
\end{prop}

\begin{proof}
Suppose, by contradiction, that the proposition is false. Then there are sequences $\{j\}\subset \{i\}$, $\{x_j\in X_j\}$, and $\{V_j\in {\bf C}(\{\Phi_i\})\}$, such that ${\bf F}(|\Phi_j(x_j)|,V_j)\rightarrow 0$ and $\mathcal{F}(\Phi_j(x_j),T_{V_j})\geq \delta$ for any $j$ and some $\delta>0$. By (\ref{flat.varifold}), there is $V_j'\in \mathcal{V}$ such that
$$
\mathcal{F}(\Phi_j(x_j),T_{V_j'})+{\bf F}(|\Phi_j(x_j)|,V_j')\rightarrow 0.
$$

We can suppose  that $V_j\rightarrow V\in {\bf C}(\{\Phi_i\})$ and $V_j'\rightarrow V'\in \mathcal{V}$ in the varifold topology. Then $V=V'$. Since $V_j'$ is stationary, and $\partial T_{V_j'}\rightarrow 0$ in the flat topology, by \cite{white-currents-chains} it follows that $T_{V_j'}\rightarrow T_{V'}$ in the flat topology. Hence $\Phi_j(x_j)\rightarrow T_V$ in the flat topology.  But since $V_j$ is stationary, and $\partial T_{V_j}=0$ by Proposition \ref{flat.sequences},  $T_{V_j}\rightarrow T_V$ in the flat topology by \cite{white-currents-chains}. This is a contradiction because $\mathcal{F}(\Phi_j(x_j),T_{V_j})\geq \delta$ for any $j$.
\end{proof}

Let $V\in {\bf C}(\{\Phi_i\})$. Denote by ${\rm sing}(V)$ the set of singular points of ${\rm support}(V)$, and by $E$ the set of edges of $V$. 

The proof of the following proposition uses arguments similar to those of  \cite{song-embeddedness}. 

\begin{prop}\label{multiplicity.edges}
Let $V\in {\bf C}(\{\Phi_i\})$.  The edges of $V$ have integer multiplicity either 1 or 2.
\end{prop}

\begin{proof}
Let $\{\Omega_1,\dots,\Omega_h\}$ be the connected components of $M\setminus {\rm spt}(V)$. We say that $\Omega_{k_1}$ is equivalent 
to $\Omega_{k_2}$ if any continuous path that connects  $p\in \Omega_{k_1}$ to $q\in \Omega_{k_2}$ and intersects ${\rm spt}(V)$ at the interior of the edges transversally does so an even number of times (where each edge is counted with its integer multiplicity in $V$). There cannot be two paths $c,c'$  connecting $\Omega_{k_1}$ to $\Omega_{k_2}$ such that $c$ has an even number of intersections and $c'$ has an odd number of intersections. Otherwise, there would be a closed path in $M$ with an odd number of intersections. This closed path would intersect the edges of the cycle $T_V$ an odd number of times, which is a contradiction since $T_V$ is a modulo two boundary. Therefore the set $\{\Omega_1,\dots,\Omega_h\}$ can be partitioned in either one or two equivalence classes. 

Each $\Omega_k$ has a connected boundary $\partial \Omega_k$, as otherwise we could construct a simple closed geodesic disjoint from ${\rm support}(V)$ as in Proposition \ref{connected.support.geodesic}. Since $\partial\Omega_k$ can be perturbed inside $\Omega_k$ into a strictly convex curve $\partial \Omega_k'$, by Grayson's theorem \cite{grayson} the curve shortening flow applied to $\partial \Omega_k'$ converges to a point. This is because by Proposition \ref{no.disjoint.geodesic} there is no embedded closed geodesic disjoint from ${\rm support}(V)$. Hence there is a path $t\in [0,1]\mapsto \partial \Omega_k(t)$ that is continuous in the ${\bf F}$-metric, satisfies $\Omega_k(0)=0$ and $\Omega_k(1)=\Omega_k$, and is such that ${\bf M}(\partial \Omega_k(t))\leq  {\bf M}(\partial \Omega_k)$.

By changing notation we can suppose that $\{\Omega_{1}, \dots, \Omega_{j}\}$ is an equivalence class, where $1\leq j\leq h$. We define a  map $\Psi:[0,1]\rightarrow \mathcal{Z}_1(M,\mathbb{Z}_2)$ by
$$
\Psi(t)=\sum_{i=1}^j \partial \Omega_{i}(t).
$$
Denote by $E'$ the set of edges of $V$ with odd integer multiplicity, and by $E''$ the set of edges of $V$ with
even integer multiplicity.  Then in the convergence of $\Psi(t)$ as $t\rightarrow 1$, the edges of $E'$ are limits of multiplicity one and the edges of $E''$ are either not limits or are  covered with multiplicity two.  Hence, $\Psi(0)=0$ and $\Psi(1)=T_V$.  Also,
$$
{\bf M}(\Psi(t))\leq \sum_{e\in E'}{\bf M}(e)+\sum_{e\in E''} 2 {\bf M}(e)\leq ||V||(M)
$$
for any $t\in [0,1]$. 

If $j=h$, in which case there is only one equivalence class, $T_V=0$ and $\Psi:([0,1]/\sim)\rightarrow \mathcal{Z}_1(M,\mathbb{Z}_2)$ is a one-sweepout. 

If $j<h$, we can define a map  $\Psi':[1,2]\rightarrow \mathcal{Z}_1(M,\mathbb{Z}_2)$ by
$$
\Psi'(t)= \sum_{i=j+1}^h \partial \Omega_{i}(2-t).
$$
Then $\Psi'(1)=T_V$, $\Psi'(2)=0$, and 
$$
{\bf M}(\Psi'(t))\leq \sum_{e\in E'}{\bf M}(e)+\sum_{e\in E''} 2 {\bf M}(e)\leq ||V||(M)
$$
for any $t\in [1,2]$. In this case, we can then concatenate the maps $\Psi$ and $\Psi'$ into a map $\Psi:([0,2]/\sim)\rightarrow \mathcal{Z}_1(M,\mathbb{Z}_2)$ which is a one-sweepout.

In any case, the map $\Psi$ is continuous in the flat topology and has no concentration of mass. Since $\Psi$ is a one-sweepout, we get
$$
||V||(M)=\omega_1(M,g)\leq \sum_{e\in E'}{\bf M}(e)+\sum_{e\in E''} 2 {\bf M}(e)\leq ||V||(M).
$$
Hence the edges of $V$ have multiplicity either 1 or 2, as claimed. 
\end{proof}

	The arguments to prove the following proposition are similar to those of Section 5 of  \cite{ketover-liokumovich}.

\begin{prop}\label{vertex}
Let $V\in {\bf C}(\{\Phi_i\})$. There is at most one singular point in ${\rm spt}(V)$. 
\end{prop}

\begin{proof}
Suppose that there are two singular points $p$, $q\in {\rm support}(V)$. Let $\delta>0$ be sufficiently small, and consider the geodesic circles
$C_\delta(p),C_\delta(q)$. Let $\{p_1,\dots,p_i\}$ be ${\rm spt}(V)\cap C_\delta(p)$ and $\{q_1,\dots,q_{i'}\}$ be ${\rm spt}(V)\cap C_\delta(q)$, in cyclic order. Denote $p_{i+1}=p_1$, and $q_{i'+1}=q_1$. Let $\gamma_k$ be the minimizing geodesic connecting $p_k$ to $p_{k+1}$. Similarly,  let $\eta_k$ be the minimizing geodesic connecting $q_k$ to $q_{k+1}$. Each minimizing geodesic is contained in the  closure of a connected component of 
$M \setminus {\rm spt}(V)$.  

Let $D_\delta(p),D_\delta(q)$ be the geodesic disks of radius $\delta$ centered at $p,q$, respectively.  Let $D\subset D_\delta(p)$ be the closed disk bounded by the union $\gamma_1 \cup \dots \cup \gamma_i$, and $D'\subset D_\delta(q)$ be the closed disk bounded by $\eta_1\cup \dots \cup \eta_{i'}$. Each connected component
$\Omega_k$ can be replaced by $\tilde{\Omega}_k=\Omega_k\setminus (D\cup D')$.  

Suppose $j<h$. Since $\partial \tilde{\Omega}_k$ is convex, we can apply the map construction of Proposition \ref{multiplicity.edges} to the sets $\{\tilde{\Omega}_1,\dots, \tilde{\Omega}_j\}$ instead of the sets $\{\Omega_1, \dots, \Omega_j\}$. This constructs a map
$\Psi:[0,1]\rightarrow \mathcal{Z}_1(M,\mathbb{Z}_2)$ such that $\Psi(0)=0$, $\Psi(1)=\sum_{a=1}^j \partial \tilde{\Omega}_a$, and
$$
\sup_{t\in [0,1]} {\bf M}(\Psi(t)) < ||V||(M).
$$
Similarly, by applying the construction to $\{\tilde{\Omega}_{j+1}, \dots, \tilde{\Omega}_h\}$, we get a map $\Psi':[1,2] \rightarrow \mathcal{Z}_1(M,\mathbb{Z}_2)$ such that $\Psi'(1)= \sum_{a=j+1}^h \partial \tilde{\Omega}_a$, $\Psi'(2)=0$, and
$$
\sup_{t\in [1,2]} {\bf M}(\Psi'(t)) < ||V||(M).
$$

We would like to construct a one-parameter family of cycles connecting $\Psi(1)$ to $\Psi'(1)$, and with mass bounded strictly by $||V||(M)$. Let $\alpha_k$ be the geodesic parametrized by arc length such that $\alpha_k(0)=p$ and $\alpha_k(\delta)=p_k$, and $\beta_{k}$ be the geodesic
parametrized by arc length such that $\beta_{k}(0)=q$ and $\beta_{k}(\delta)=q_{k}$. We write $\alpha_{i+1}=\alpha_1$, and $\beta_{i'+1}=\beta_1$. 

For $s\in [0,1]$, define $\gamma_k(s)$ to be the minimizing geodesic connecting $\alpha_k(\delta-\delta s)$ to $\alpha_{k+1}(\delta-\delta s)$.
Hence $\gamma_k(0)=\gamma_k$ and $\gamma_k(1)=p$. Similarly, define $\eta_k(s)$ to be the minimizing geodesic connecting $\beta_k(\delta-\delta s)$ to $\beta_{k+1}(\delta-\delta s)$, so that
 $\eta_k(0)=\eta_k$ and $\eta_k(1)=q$.

For each $1\leq k\leq i$, define the path of modulo two flat chains
$$
\tilde{\gamma}_k(s)=(\alpha_k)_{|[\delta-\delta s,\delta]}+ \gamma_k(s)+(\alpha_{k+1})_{|[\delta-\delta s,\delta]}
$$
for $s\in [0,1]$. Define $D_{k}(s)$ to be the two-chain so that $D_{k}(s)\subset D$ and $\partial D_{k}(s)=\tilde{\gamma}_k(s)+\gamma_k$.
For each $1\leq k\leq {i'}$, we define 
$$
\tilde{\eta}_k(s)=(\beta_k)_{|[\delta-\delta s,\delta]}+ \eta_k(s)+(\beta_{k+1})_{|[\delta-\delta s,\delta]}
$$
for $s\in [0,1]$. Let $D'_{k}(s)$ to be the two-chain so that $D'_{k}(s)\subset D'$ and $\partial D'_{k}(s)=\tilde{\eta}_k(s)+\eta_k$.

Let $\Omega=\cup_{i=1}^j \Omega_i$, and $\Omega'=\cup_{i=j+1}^h \Omega_i$. We define $\Gamma_1:[0,1]\rightarrow \mathcal{Z}_1(M,\mathbb{Z}_2)$ by
$$
\Gamma_1(t)=\Psi(1)+\sum_{k: \gamma_k\subset \Omega} \partial D_k(2t)
$$
for $t\in [0,1/2]$, and
$$
\Gamma_1(t)=\Psi(1)+\sum_{k: \gamma_k\subset \Omega} \partial D_k(1)
+\sum_{k: \gamma_k\subset \Omega'} \partial \big(D_k(1)- D_k(2-2t)\big)
$$
for $t\in [1/2,1]$. Since $\Gamma_1(t)\llcorner D' = \Psi(1) \llcorner D'$, it follows that
$$
\sup_{t\in [0,1]} {\bf M}(\Gamma_1(t))< ||V||(M).
$$

We can do a similar construction near $q$. Define $\Gamma_2:[0,1]\rightarrow \mathcal{Z}_1(M,\mathbb{Z}_2)$ by
$$
\Gamma_2(t)=\Gamma_1(1)+\sum_{k: \eta_k\subset \Omega} \partial D_k'(2t)
$$
for $t\in [0,1/2]$, and
$$
\Gamma_2(t)=\Gamma_1(1)+\sum_{k: \eta_k\subset \Omega} \partial D_k'(1)+\sum_{k: \eta_k\subset \Omega'} \partial \big(D'_k(1)- D'_k(2-2t)\big)
$$
for $t\in [1/2,1]$. Since $\Gamma_2(t)\llcorner D= \Gamma_1(1) \llcorner D$, we have
$$
\sup_{t\in [0,1]} {\bf M}(\Gamma_2(t))< ||V||(M).
$$

We have that
\begin{align*}
\Gamma_2(1)&=\Psi(1)+\sum_{k=1}^i \partial D_k + \sum_{k=1}^{i'} \partial D'_k \\
&= \partial \big(\sum_{a=1}^j  \tilde{\Omega}_a+D +D'\big)\\
&= \partial \big(M -\sum_{a=j+1}^h  \tilde{\Omega}_a\big)\\
&=\Psi'(1).
\end{align*}
Since $\Psi(1)=\Gamma_1(0)$, $\Gamma_1(1)=\Gamma_2(0)$ and $\Gamma_2(1)=\Psi'(1)$, the paths can be concatenated to produce
a map $\Phi:[0,1]\rightarrow \mathcal{Z}_1(M,\mathbb{Z}_2)$ such that $\Phi(0)=\Phi(1)=0$ and
$$
\sup_{t\in [0,1]}{\bf M}(\Phi(t))<||V||(M).
$$
It follows from the construction that by identifying  the endpoints we get a one-sweepout $\Phi:([0,1]/\sim)\rightarrow \mathcal{Z}_1(M,\mathbb{Z}_2)$. Since $\Phi$ is continuous in the flat topology, and has no concentration of mass, we get $\omega_1(M,g)<||V||(M)$ which is a contradiction.

The case $j=h$ is  similar.  It uses that the mass cancels near $p$  at each edge at the end of  the deformation inside $D$, as these edges are flat 
limits with multiplicity two.  Hence there is at most one singular point in ${\rm support}(V)$, which proves the proposition.
\end{proof}

\begin{prop}\label{tangent.cone}
	If $p \in {\rm spt}(V)$ is a singular point, then the tangent cone at $p$ is a sum of two lines with multiplicity one.
\end{prop}

\begin{proof}
The notation  is as in the proofs of Proposition \ref{multiplicity.edges} and Proposition \ref{vertex}.
Suppose that an edge $\alpha$ of ${\rm support}(V)$ containing $p$ has integer multiplicity two.  The number of edges containing $p$ 
is greater than or equal to three. By changing notation we can suppose that the edge connecting $p$ to $\gamma_1\cap\gamma_i$ has 
multiplicity two, and that ${\rm int}(D_1(1))\cup {\rm int}(D_i(1))\subset \Omega$. By cancellation of mass, there is a neighborhood $U$ of some $x\in \alpha$ close to $p$ such that $\overline{U} \cap D=\emptyset$, $\Psi(1) \llcorner U=\Psi'(1)\llcorner U =0$. Hence,  if $\xi=2{\bf M}(\alpha \cap U)$  then
$$
{\bf M}(\Psi(1))\leq ||V||(M)-\xi
$$
and
$$
{\bf M}(\Psi'(1))\leq ||V||(M)-\xi.
$$
Notice that $\xi$  does not depend on $\delta$.
Also ${\bf M}(\partial D_k(s))\leq O(\delta)$, for any $s\in [0,1]$.

Define $\Gamma:[0,1]\rightarrow \mathcal{Z}_1(M,\mathbb{Z}_2)$ by
$$
\Gamma(t)=\Psi(1)+\sum_{k: \gamma_k\subset \Omega} \partial D_k(2t)
$$
for $t\in [0,1/2]$, and
$$
\Gamma(t)=\Psi(1)+\sum_{k: \gamma_k\subset \Omega} \partial D_k(1)
+\sum_{k: \gamma_k\subset \Omega'} \partial \big(D_k(1)- D_k(2-2t)\big)
$$
for $t\in [1/2,1]$. Then $\Psi(1)=\Gamma(0)$, and $\Gamma(1)=\Psi'(1)$. 

For $t\in [0,1/2]$, 
$$
{\bf M}(\Gamma(t))\leq ||V||(M)-\xi+O(\delta).
$$
Since
$$
\Gamma(t)=\Psi'(1)+\sum_{k: \gamma_k\subset \Omega'} \partial \big( D_k(2-2t)\big),$$
$$
{\bf M}(\Gamma(t))\leq ||V||(M)-\xi+O(\delta)
$$
for $t\in [1/2,1]$. Hence, if $\delta$ is sufficiently small then
$$
\sup_{t\in [0,1]}{\bf M}(\Gamma(t))<||V||(M).
$$
By concatenating the maps $\Psi$, $\Gamma$ and $\Psi'$, we get a map $\Phi:[0,1]\rightarrow \mathcal{Z}_1(M,\mathbb{Z}_2)$ such that $\Phi(0)=\Phi(1)=0$ and
$$
\sup_{t\in [0,1]}{\bf M}(\Phi(t))<||V||(M).
$$
As in the Proposition \ref{vertex}  this induces a one-sweepout $\Phi:([0,1]/\sim)\rightarrow \mathcal{Z}_1(M,\mathbb{Z}_2)$, and hence 
$\omega_1(M,g)<||V||(M)$ which is a contradiction.

Therefore we can suppose that any edge of ${\rm support}(V)$ containing $p$ has integer multiplicity one. The number $i$ of edges containing $p$ cannot be odd, as $\partial T_V=0$. Hence  the number of edges is either four or is greater than or equal to six. 

Suppose $i\geq 6$. By changing notation, we can suppose that ${\rm int}(D_1(1))\subset \Omega'$.  We define $\Gamma_1:[0,1]\rightarrow \mathcal{Z}_1(M,\mathbb{Z}_2)$ by
$$
\Gamma_1(t)=\Psi(1)+\partial D_2(2t)+\partial D_i(2t)
$$
for $t\in [0,1/2]$, and
$$
\Gamma_1(t)=\Psi(1)+\partial D_2(1)+\partial D_i(1)
+ \partial \big(D_1(1)- D_1(2-2t)\big)
$$
for $t\in [1/2,1]$. Notice that $\Gamma_1(t)\llcorner D_{i-2}=\gamma_{i-2}$ for any $t\in [0,1]$, and ${\bf M}(\gamma_{i-2})<2\delta$.  Hence
$$
\sup_{t\in [0,1]} {\bf M}(\Gamma_1(t))< ||V||(M).
$$

We define $\Gamma_2:[0,1]\rightarrow \mathcal{Z}_1(M,\mathbb{Z}_2)$ by
$$
\Gamma_2(t)=\Gamma_1(1)+\sum_{k=2}^{(i-2)/2} \partial D_{2k}(2t)
$$
for $t\in [0,1/2]$, and
$$
\Gamma_2(t)=\Gamma_1(1)+\sum_{k=2}^{(i-2)/2} \partial D_{2k}(1) + \sum_{k=2}^{i/2} \partial \big(D_{2k-1}(1)- D_{2k-1}(2-2t)\big)
$$
for $t\in [1/2,1]$. Notice that $\Gamma_2(t)\llcorner D_1=\Gamma_1(1)\llcorner D_1$ for $t\in [0,1]$. Hence $\Gamma_2(t)\llcorner D_{1}=\gamma_{1}$ for any $t\in [0,1]$. Since ${\bf M}(\gamma_{1})<2\delta$, one has
$$
\sup_{t\in [0,1]} {\bf M}(\Gamma_2(t))< ||V||(M).
$$
We have $\Psi(1)=\Gamma_1(0)$,  $\Gamma_1(1)=\Gamma_2(0)$, and 
$$\Gamma_2(1)=\Psi(1)+\sum_{k=1}^{i/2} \partial D_{2k} + \sum_{k=1}^{i/2} \partial D_{2k-1} =\Psi'(1).$$
By concatenating $\Psi,\Gamma_1,\Gamma_2$ and $\Psi'$ as before it follows that
$$
\omega_1(M,g)<||V||(M),
$$
which is a contradiction. Therefore $i=4$. Since $V$ is a stationary integral varifold, $V\llcorner D_\delta(p)$ must be the varifold induced by 
two geodesics containing $p$ with multiplicity one. This proves the proposition. 
\end{proof}

	Combining the previous propositions, we can now summarize what are the properties of any $V\in {\bf C}(\{\Phi_i\})$. By Proposition \ref{flat.sequences}, $T_V\in \mathcal{Z}_1(M,\mathbb{Z}_2)$ and  there are sequences   $\{j\}\subset \{i\}$ and $\{x_j\in X_j\}_j$ such that  ${\bf F}(|\Phi_j(x_j)|,V)\rightarrow 0$ and $\mathcal{F}(\Phi_j(x_j),T_V)\rightarrow 0$. We know that ${\rm spt}(V)$ is connected by Proposition \ref{connected.support.geodesic}.
The Propositions \ref{multiplicity.edges}, \ref{vertex} and \ref{tangent.cone}  imply that the edges of $V$ have integer multiplicity either 1 or 2, that there is at most one singular point in 
${\rm spt}(V)$, and that if $p \in {\rm spt}(V)$ is a singular point, then the tangent cone of $V$ at $p$ is a sum of two lines with multiplicity one.

Therefore we have the following  possible cases: (i) $V$ is a simple closed geodesic with multiplicity one; 
(ii) $V$ is a figure-eight closed geodesic with multiplicity one; (iii) $V$
is the sum of two simple closed geodesics with multiplicity one which intersect transversely at one point, or a simple closed geodesic with multiplicity two. \\

Case (i): $V$ is a simple closed geodesic $\gamma$ with multiplicity one. \\

In this case we also have that $T_V=\gamma$. Since $T_V \in \mathcal{Z}_1(M,\mathbb{Z}_2)$, the geodesic $\gamma$ has to be two-sided.
By Proposition \ref{no.disjoint.geodesic}, the geodesic $\gamma$ is  part of a local foliation  so that the leaves   have  curvature vector pointing strictly away from  $\gamma$. If we apply curve shortening flow to a leaf on one side of $\gamma$, by Grayson's theorem \cite{grayson} the flow converges to a point since there is no closed geodesic disjoint from $\gamma$ by Proposition \ref{no.disjoint.geodesic}. Therefore $\gamma$ separates the surface $M$ into two regions which are homeomorphic to disks, which implies that $M$ is homeomorphic to  the two-sphere $S^2$.

Let $\Gamma\subset \mathcal{V}$ be the set of varifolds induced by two-sided, simple closed geodesics with multiplicity one and length $2\pi$. By Allard \cite{allard}, the set $\Gamma$ is open in $\mathcal{V}$. Since there are no one-sided simple closed curves in $S^2$, the set $\Gamma$ is also compact in the varifold topology.   Therefore  there is $\zeta>0$ such that if $V'\in \Gamma$, $\tilde{V}\in \mathcal{V}$ satisfy ${\bf F}(\tilde{V},V')\leq \zeta$, then $\tilde{V}\in \Gamma$. Hence any $\zeta$-chain of varifolds in $\mathcal{V}$ starting with $V$ is contained in $\Gamma$. This implies that ${\bf C}(\{\Phi_j\}_j)\subset \Gamma$.

Let $p$, $q\in M$, $p\neq q$. Since $\Phi_j$ is a two-sweepout, for any open sets $\Omega$, $\Omega'\subset M$, $\overline{\Omega}\cap \overline{\Omega'}=\emptyset$, $p\in \Omega$, $q\in \Omega'$, there is $x_j\in X_j$ such that $\Phi_j(x_j)$ bisects $\Omega$ and $\Omega'$ into regions of the same area. There is $\{k\}\subset \{j\}$ such that $|\Phi_k(x_k)|\rightarrow V'\in \Gamma.$ By letting $\Omega\rightarrow p$, and $\Omega'\rightarrow q$, we find a simple closed geodesic $\gamma_{p,q}\in \Gamma$ that contains $p$ and $q$. Let $v\in T_pM$, $|v|=1$. Since the geodesics 
$\gamma_i=\gamma_{p,{\rm exp}_p(v/i)}$ are embedded and converge with multiplicity one, we get that the geodesic passing by $p$ with velocity $v$ is a simple closed geodesic of length $2\pi$. This proves that in this case $g$ is a Zoll metric on $S^2$.  \\

Case (ii): $V$ is a figure-eight closed geodesic $\gamma$ with multiplicity one. \\

We will show this case leads to a contradiction. Let $\Gamma'\subset \mathcal{V}$ be the set of varifolds induced by figure-eight closed geodesics with multiplicity one and length $2\pi$. 
Let $\{V_j\}_j\subset \Gamma' \cap {\bf C}(\{\Phi_i\})$. Then there is $\{k\}\subset \{j\}$ such that $V_k\rightarrow V'\in  {\bf C}(\{\Phi_i\})$ in the varifold topology.  The varifold $V'$ is either a figure-eight with multiplicity one, or a simple closed geodesic with multiplicity two. 

Suppose $V'$ is a simple closed geodesic $\gamma$ with multiplicity two. The varifold $V_k$ can be desingularized at the self-intersection to produce two disjoint simple closed curves $\alpha_k$, $\beta_k$, such that $\alpha_k,\beta_k\rightarrow \gamma$ with multiplicity one. Therefore the  geodesic $\gamma$ has  to be two-sided. In this case the curve shortening flow argument of case (i) would produce a one-sweepout of $M$ with mass bounded by $l(\gamma)$. This is a contradiction because $\omega_1(M,g)=||V'||(M)=2\,l(\gamma)$.
Therefore $V'$ is a figure-eight with multiplicity one, which proves that $\Gamma'\cap {\bf C}(\{\Phi_i\})$ is compact.

If $\{V_j\}\subset {\bf C}(\{\Phi_i\})$ is such that $V_j\rightarrow V\in \Gamma' \cap {\bf C}(\{\Phi_i\})$, then for sufficiently large $j$, $V_j \in \Gamma'$. This is because for such $j$, $V_j$ cannot be a simple closed geodesic or a union of two simple closed geodesics intersecting transversely. Therefore $\Gamma' \cap {\bf C}(\{\Phi_i\})$ is open in ${\bf C}(\{\Phi_i\})$. As in case (i) this implies that any $V'\in {\bf C}(\{\Phi_i\})$ is a figure-eight with multiplicity one (${\bf C}(\{\Phi_i\})=\Gamma' \cap {\bf C}(\{\Phi_i\})$).

Let $0<\delta<{\rm inj}(M)$, and $V\in {\bf C}(\{\Phi_i\})$ be the varifold  of a figure-eight. Let $p$ be the singular point of ${\rm support}(V)$.
We use the notation of the Propositions \ref{multiplicity.edges}, \ref{vertex}, and \ref{tangent.cone}. By changing the notation, we can suppose that $S_V=T_V+\partial D_1(1)+\partial D_3(1)$ is a sum of two nontrivial cycles.  The function $V \mapsto S_V$ is well-defined. If $\{V_j\}\subset {\bf C}(\{\Phi_i\})$ is such that $V_j\rightarrow V$ in the varifold topology, then 
$$
\mathcal{F}(T_{V_j},T_V)+{\bf F}(|T_{V_j}|,|T_V|) \rightarrow 0,
$$
and
$$
\mathcal{F}(S_{V_j},S_V)+{\bf F}(|S_{V_j}|,|S_V|) \rightarrow 0.
$$
Since ${\bf M}(S_V)<{\bf M}(T_V)=\omega_1(M,g)$ for any $V\in {\bf C}(\{\Phi_i\})$, we can find $\xi>0$ such that
$$
{\bf M}(S_V)\leq \omega_1(M,g)-\xi
$$
for any $V\in {\bf C}(\{\Phi_i\})$. Also, $\mathcal{F}(S_V,T_V)\leq O(\delta^2)$ for any $V\in {\bf C}(\{\Phi_i\})$.

For any $0<\rho<\delta^2$, by a compactness argument we can find  $0<\eta<\rho$ such that if $V,V'\in {\bf C}(\{\Phi_i\})$  satisfy ${\bf F}(V,V')<\eta$, then 
there is a continuous map $\Gamma:[0,1]\rightarrow \mathcal{Z}_1(M,{\bf F}, \mathbb{Z}_2)$ such that $\Gamma(0)=S_V$, $\Gamma(1)=S_{V'}$,
$$
{\bf F}(\Gamma(t),\Gamma(t'))<\rho,
$$
and
$$
{\bf M}(\Gamma(t))\leq \omega_1(M,g)-\xi/2
$$
for any $t,t'\in [0,1]$.

Let $\sigma_j \subset X_j$, $\sigma_j \approx S^1$, such that $(\Phi_j)_{|\sigma_j}$ is a one-sweepout.  We identify $\sigma_j=([0,1]/\sim)$.
By Proposition \ref{flat.convergence}, for sufficiently large $j$ one has that for any $x\in [0,1]$ there is $V\in {\bf C}(\{\Phi_i\})$ with
$\mathcal{F}(\Phi_j(x),T_V)+{\bf F}(|\Phi_j(x)|,V) < \eta/5$. We can find a partition $0=t_0<t_1<\dots<t_h=1$ such that, for any $t,t'\in [t_k,t_{k+1}]$, ${\bf F}(\Phi_j(t),\Phi_j(t'))<\eta/5$. For each $k=0,\dots,h$, let $V_k\in {\bf C}(\{\Phi_i\})$ be such that $\mathcal{F}(\Phi_j(t_k),T_{V_k})+{\bf F}(|\Phi_j(t_k)|,V_k) < \eta/5$. Since we are identifying the endpoints, we can suppose that $V_h=V_0$. Hence ${\bf F}(V_k,V_{k+1})<\eta$ for any $k=0,\dots,h-1$. By reparametrizing, we can find for each $k=0,\dots,h-1$ a map
$\Gamma_k:[t_k,t_{k+1}]\rightarrow   \mathcal{Z}_1(M,{\bf F}, \mathbb{Z}_2)$ such that $\Gamma_k(t_k)=S_{V_k}$, $\Gamma_k(t_{k+1})=S_{V_{k+1}}$,
$$
{\bf F}(\Gamma_k(t),\Gamma_k(t'))<\rho,
$$
and
$$
{\bf M}(\Gamma_k(t))\leq \omega_1(M,g)-\xi/2
$$
for any $t,t'\in [t_k,t_{k+1}]$. We can concatenate the maps $\Gamma_k$ into a map $\Gamma:[0,1]\rightarrow  \mathcal{Z}_1(M,{\bf F}, \mathbb{Z}_2)$, with $\Gamma(0)=\Gamma(1)$, and such that
$$
{\bf M}(\Gamma(t))\leq \omega_1(M,g)-\xi/2
$$
for any $t\in [0,1]$. Notice that, for $t\in [t_k,t_{k+1}]$,
\begin{eqnarray*}
&&\mathcal{F}(\Gamma(t),\Phi_j(t))\leq \mathcal{F}(\Gamma_k(t),\Gamma_k(t_k))+\mathcal{F}(\Gamma_k(t_k),\Phi_j(t_k))+\mathcal{F}(\Phi_j(t_k), \Phi_j(t))\\
&&\leq \rho+\mathcal{F}(S_{V_k},\Phi_j(t_k))+\eta/5\\
&&\leq \rho+\mathcal{F}(S_{V_k},T_{V_k})+\mathcal{F}(T_{V_k},\Phi_j(t_k))+\eta/5\\
&&\leq \rho+O(\delta^2)+2\eta/5=O(\delta^2).
\end{eqnarray*}

If we choose $\delta>0$ sufficiently small, Proposition 3.5 of \cite{marques-neves-infinitely} implies that $\Phi_j$ and $\Gamma$ are homotopic in the flat topology. Therefore $\Gamma$ is a one-sweepout, but this is a contradiction because $\sup_{t\in S^1}{\bf M}(\Gamma(t))<\omega_1(M,g)$. \\

Case (iii): $V$
is the sum of two simple closed geodesics with multiplicity one which intersect transversely at one point, or a simple closed geodesic with multiplicity two. \\

Suppose there is $V\in {\bf C}(\{\Phi_i\})$ that is a simple closed geodesic $\gamma$ with multiplicity two. We have $\ell(\gamma)=\pi$, because $2\ell(\gamma)=||V||(M)=\omega_1(M,g)=2\pi$. As in the case (i), if the geodesic
is two-sided one could use curve shortening flow to prove that $\omega_1(M,g)\leq l(\gamma)$, a contradiction.  Hence the geodesic $\gamma$ has to be one-sided. By Proposition \ref{no.disjoint.geodesic}, the geodesic $\gamma$ is  part of a local foliation  such that the leaves   have  curvature vector pointing strictly away from  $\gamma$. By applying the curve shortening flow, since by Proposition \ref{no.disjoint.geodesic} there is no simple closed geodesic disjoint from $\gamma$, we get that any two-sided leaf bounds a disk. Therefore in this case the surface $M$ is diffeomorphic to the projective plane $\mathbb{RP}^2$.
 
Let $\tilde{\gamma}$ be a length-minimizing noncontractible curve. Then $\tilde{\gamma}$ is a simple closed geodesic whose length is the systole ${\rm sys}(M,g)$ of $(M^2,g)$. If $\tilde{\gamma}=\gamma$ then ${\rm sys}(M,g)=\pi$. Suppose that $\tilde{\gamma}\neq \gamma$. Since $M\approx \mathbb{RP}^2$, the cycle $T=\gamma+\tilde{\gamma}$ is in $\mathcal{Z}_1(M,\mathbb{Z}_2)$. Since there is no closed geodesic disjoint from $\gamma\cup \tilde{\gamma}$, we can proceed as in the proof of Proposition \ref{multiplicity.edges} to get a one-sweepout of $M$ with mass bounded by $l(\gamma)+l(\tilde{\gamma})$. Hence $2l(\gamma)=\omega_1(M,g)\leq l(\gamma)+l(\tilde{\gamma})$, and so $l(\gamma)\leq l(\tilde{\gamma})$. Since $l(\tilde{\gamma})\leq l(\gamma)$ as $\gamma$ is noncontractible, we get that ${\rm sys}(M,g)=l(\tilde{\gamma})=l(\gamma)=\pi$. In any case we have that  ${\rm sys}(M,g)=\pi$. 

Hence we can suppose that any $V\in {\bf C}(\{\Phi_i\})$ is a sum of two  closed geodesics with multiplicity one which are simple and intersect transversely at one point. Let $V\in {\bf C}(\{\Phi_i\})$. Hence $V=|\gamma|+|\gamma'|$, where $\gamma$ and $\gamma'$ are simple closed geodesics which intersect transversely at one point. Then $T_V=\gamma+\gamma'\in \mathcal{Z}_1(M,\mathbb{Z}_2)$ by Proposition \ref{flat.sequences}.  Using the notation of Proposition \ref{vertex} applied to $p\in \gamma \cap \gamma'$, we have that ${\rm int}(D_i(1))$ and ${\rm int}(D_{i+1}(1))$ are in different connected components of 
$M \setminus (\gamma  \cup \gamma')$, for $1\leq i\leq 4$. If we take a point $q\in {\rm int}(D_i(1))$ and move it closely along one of the geodesics $\gamma$ or $\gamma'$, we get that ${\rm int}(D_i(1))$ and ${\rm int}(D_{i+2}(1))$ are in the same connected component of $M \setminus (\gamma  \cup \gamma')$ for $1\leq i\leq 2$. Hence $\gamma \cup \gamma'$ separates $M$ into two connected components $\Omega$, $\Omega'$. It also follows that $\gamma$ and $\gamma'$ are one-sided closed geodesics. Notice that $\partial \Omega=\partial \Omega'=\gamma \cup \gamma'$.

The component $\Omega$ can be perturbed inside $\Omega$ into a strictly convex curve $\alpha$, and $\Omega'$ can be perturbed inside $\Omega'$ into a strictly convex curve $\alpha'$. Since there is no closed geodesic disjoint from $\gamma \cup \gamma'$, the curve shortening flow argument gives that $\alpha$ bounds a disk $D\subset \Omega$ and $\alpha'$ bounds a disk $D'\subset \Omega'$.  By using that $D$ and $D'$ are simply connected, and that $\gamma \cup \gamma'$ is a deformation retract of a neighborhood, we get that any loop $\sigma\in \pi_1(M,\{p\})$ is homotopic to a loop contained in $\gamma \cup \gamma'$. By performing homotopies in $D,D'$, we get that $\gamma,\gamma',\gamma^{-1},(\gamma')^{-1}$ are homotopic to each other. Hence $\pi_1(M,\{p\})=\{1,\gamma\}=\mathbb{Z}_2$, and $M$ is diffeomorphic to $\mathbb{RP}^2$.

If $\{V_j=|\gamma_j|+|\gamma_j'|\}\subset {\bf C}(\{\Phi_i\})$ is a sequence that converges in the varifold topology to $V=|\gamma|+|\gamma'|\in {\bf C}(\{\Phi_i\})$, then by changing the notation we can suppose that $\gamma_j$ converges smoothly to $\gamma$ and $\gamma_j'$ smoothly to $\gamma'$. 

Suppose $V=|\gamma|+|\gamma'|\in {\bf C}(\{\Phi_i\})$. Then $\gamma$ is part of a local foliation $\{\gamma_t\}_{t\in [0,s]}$, $\gamma_0=\gamma$, such that for any $t\in (0,s]$, $\gamma_t$ is a two-sided curve whose curvature  has a sign. 

Suppose that for some $\tilde{t}\in (0,s]$, $\gamma_{\tilde{t}}$ is a closed geodesic. Since $\gamma_{\tilde{t}}$ is two-sided, there is a neighborhood $U$ of $\gamma_{\tilde{t}}$ that is diffeomorphic to a cylinder. Let $0<t'<\tilde{t}$. By Proposition \ref{intersects.must.be.contained.geodesic} we can choose $t'$ such that any smooth closed geodesic with length bounded by $2\, \omega_1(M,g)$, disjoint from $\gamma_{\tilde{t}}$, that intersects $\bigcup_{t\in [t',\tilde{t}]} \gamma_t$, must be contained in $U$. 

Consider $\mathcal{V}'\subset {\bf C}(\{\Phi_i\})$ to be the set of varifolds $V'$ such that at least one of the closed geodesics in ${\rm spt}(V')$ is contained in $\bigcup_{t\in [0,t']} \gamma_t$. Then
$\mathcal{V}'$ is closed in ${\bf C}(\{\Phi_i\})$. Let $\{V_j=|\gamma_j|+|\gamma_j'|\}\subset {\bf C}(\{\Phi_i\})$ be a sequence that converges in the varifold topology to $V'\in \mathcal{V}'$.  By changing notation, we can suppose that $\gamma_j$ converges smoothly to a closed geodesic
that is contained in $\bigcup_{t\in [0,t']} \gamma_t$.  This implies that for sufficiently large $j$, $\gamma_j$ is disjoint from $\gamma_{\tilde{t}}$ and is contained in $\bigcup_{t\in [0,\tilde{t}]} \gamma_t$. If $\gamma_j$ intersects $\bigcup_{t\in [t',\tilde{t}]} \gamma_t$, we would have $\gamma_j\subset U$. But this is a contradiction, since $\gamma_j$ is one-sided and $U$ is a cylinder. Then, for sufficiently large $j$, $\gamma_j \subset 
\bigcup_{t\in [0,t']} \gamma_t$ and hence $V_j\in \mathcal{V}'$. This proves that $\mathcal{V}'$ is open in ${\bf C}(\{\Phi_i\}).$

Since $\gamma \subset \bigcup_{t\in [0,t']} \gamma_t$, it follows that $V\in \mathcal{V}'$. By considering $\zeta$-chains in ${\bf C}(\{\Phi_i\})$ starting with $V$, we get that ${\bf C}(\{\Phi_i\})\subset \mathcal{V}'$. Let $p$, $q\in \gamma_{\tilde{t}}$, $p\neq q$. Since $\Phi_i$ is a two-sweepout 
for each $i$, there is $V'\in {\bf C}(\{\Phi_i\})$ such that $p$, $q \in {\rm spt}(V')$.   Since $V'\in \mathcal{V}'$, it follows that one of the closed
geodesics in ${\rm spt}(V')$ does not contain either $p$ or $q$. Therefore there is a closed geodesic in ${\rm spt}(V')$ that contains $p$ and $q$. By letting $q\rightarrow p$ on $\gamma_{\tilde{t}}$, we get in the limit a one-sided closed geodesic tangent to the geodesic
$\gamma_{\tilde{t}}$ at $p$. These geodesics would have to agree but this is a contradiction since $\gamma_{\tilde{t}}$ is two-sided.

We proved that there is no $t\in (0,s]$ such that $\gamma_{t}$ is a closed geodesic. Suppose then  that for any $t\in (0,s]$, the curvature vector of $\gamma_t$ points strictly towards $\gamma$. The maximum principle implies that $\gamma$ is the only closed geodesic contained in 
$\bigcup_{t\in [0,s]} \gamma_t$.  By compactness of smooth closed geodesics with length bounds, we can find $\tilde{t}\in (0,s]$ such that there is no closed geodesic tangent to $\gamma_{\tilde{t}}$ with length bounded by $2\,\omega_1(M,g)$. Since the region $\bigcup_{t\in [0,\tilde{t}]} \gamma_t$ has strictly convex boundary, we can proceed as in the previous case by using Theorem 5 of \cite{white-maximum-principle}.

Therefore $\gamma$ is part of a local foliation such that any of the other leaves has curvature vector pointing strictly away from $\gamma$. If we apply the curvature shortening flow to such a leaf, then by \cite{grayson} it either converges to a point or to a  closed geodesic $\gamma''$ disjoint from $\gamma$. Since $M \approx \mathbb{RP}^2$, this closed geodesic is two-sided. By considering $\mathcal{V}'$  the set of varifolds $V'\in {\bf C}(\{\Phi_i\})$ such that at least one of the closed geodesics in ${\rm spt}(V')$ does not intersect a neighborhood $B_\eta(\gamma'')$  we can get a contradiction as before.  This also proves that there cannot be a closed geodesic disjoint from $\gamma$ (or similarly from $\gamma'$). Therefore we can find a one-sweepout of $(M,g)$ with mass bounded by $2l(\gamma)$.  Hence $\omega_1(M,g)\leq 2l(\gamma)$, and similarly $\omega_1(M,g)\leq 2l(\gamma')$.  Since $\omega_1(M,g)=l(\gamma)+l(\gamma')$, we get that $l(\gamma)=l(\gamma')$.

We proved that $l(\gamma)=l(\gamma')=\pi$. 
Let $\tilde{\gamma}$ be the noncontractible simple closed geodesic with length ${\rm sys}(M,g)$. If $\tilde{\gamma}=\gamma$, then ${\rm syst}(M,g)=\pi$. If  $\tilde{\gamma}\neq \gamma$, we can apply curve shortening flow to the connected components of $M \setminus (\gamma \cup \tilde{\gamma})$. Since there is no closed geodesic disjoint from $\gamma$,  we can construct as before a one-sweepout of $M$ with mass bounded by $l(\gamma)+l(\tilde{\gamma})$. Hence $2\pi=\omega_1(M,g)\leq \pi+l(\tilde{\gamma})\leq \pi+l(\gamma)=2\pi$, which implies ${\rm sys}(M,g)= \pi$.  This finishes the proof of Theorem \ref{rigidity.projective.plane2}.
\end{proof}

\begin{rmk}
We do not know of any examples of $\mathbb{RP}^2$ with $\omega_1=\omega_2=\omega_3=2sys$ that do not have constant curvature.
\end{rmk}

\begin{proof}[Proof of Theorem A]
It follows by the Weyl law \eqref{weyl-law} that the volume spectrum determines the dimension of the manifold:
$$
n+1=\lim_{k\rightarrow\infty} \frac{\ln k}{\ln \omega_k(M,g)}.
$$
Since $\omega_k(M,g)=\omega_k(\mathbb{RP}^2,\overline{g})$
 for any $k
\geq 1$, $(n+1)=2$ and hence  $M$ is two-dimensional. By the Weyl law \eqref{weyl-law} again, 
$${\rm area}(M,g)={\rm area}(\mathbb{RP}^2,\overline{g})=2\pi.$$

By Proposition \ref{widths.projective.plane}, we have that 
$$
\omega_1(M, g)=\omega_2(M, g)=\omega_3(M, g)=2\pi.
$$
Hence  Theorem \ref{rigidity.projective.plane2} implies that either $(M,g)$ is isometric to a Zoll sphere $(S^2,g)$ whose geodesics have length $2\pi$, or $M$ is diffeomorphic to  $\mathbb{RP}^2$  and ${\rm sys}(M,g)=\pi$.

The first case is not possible. In fact, by Weinstein's Theorem A in \cite{weinstein}, a Zoll sphere $(S^2,g)$ whose geodesics have length $2\pi$ has ${\rm area}(S^2,g)\geq 4\pi$.

Therefore $g$ must be a metric on $\mathbb{RP}^2$, with systole equal to $\pi$ and area equal to $2\pi$. By the case of equality in Pu's isosystolic inequality (\cite{pu}), $(M,g)$ is isometric to $(\mathbb{RP}^2,\overline{g})$. This finishes the proof of Theorem A.
\end{proof}

\section{Spherical area widths of surface Zoll metrics}

 A Zoll family in $(S^3,g)$  is a  family $\{\Sigma_\sigma\}_{\sigma\in \mathbb{RP}^{3}}$ of smoothly embedded $2$-spheres   such that for each $p\in S^3$ and each two-dimensional space $\pi\subset T_pS^3$ there exists a unique $\sigma\in \mathbb{RP}^{3}$ such that $T_p\Sigma_\sigma=\pi$. We suppose that the map
$\sigma\in \mathbb{RP}^3 \rightarrow \Sigma_\sigma \in \tilde{\mathcal{S}}$ seen as a map into the space of $C^{3,\alpha}$ embeddings of spheres is of class $C^1$. The set of equators in $S^3$ is a Zoll family.  

A Riemannian metric $g$ on $S^3$  is called a surface Zoll metric if there is a Zoll family $\{\Sigma_\sigma\}$ of minimal $2$-spheres for $(S^3,g)$. In \cite{ambrozio-marques-neves}, the authors have constructed examples of  surface Zoll metrics on $S^3$ which do not have constant sectional curvature.
There is a constant $w>0$ such that ${\rm area}(\Sigma_\sigma,g)=w$ for every $\sigma \in \mathbb{RP}^3$.  {In this case}, by elliptic theory, the map $\sigma\mapsto \Sigma_\sigma$ is also continuous with respect to the smooth topology in  $\tilde{\mathcal{S}}$. We denote by $L_\sigma$ the Jacobi operator of $\Sigma_\sigma$.

\begin{prop}\label{zoll.properties}
Let $\{\Sigma_\sigma\}_{\sigma\in \mathbb{RP}^{3}}$ be a Zoll family of minimal $2$-spheres of $(S^3,g)$. Then $\Sigma_\sigma$ has Morse index  one and nullity equal to three for every $\sigma\in \mathbb{RP}^3$. Moreover, for every $\sigma,\tau \in \mathbb{RP}^3$, the surface $\Sigma_\sigma$ intersects $\Sigma_\tau$. If $\sigma \neq \tau$, the intersection $\Sigma_\sigma \cap \Sigma_\tau$ is a smooth, connected, simple closed curve.
\end{prop}

\begin{proof}

With  $\sigma\in \mathbb{RP}^3$ fixed, let $\mathcal{W}\subset {\rm Ker} \, L_{\Sigma_\sigma}$ and  $\varphi:\mathcal{W}\rightarrow C^{k,\alpha}(\Sigma_\sigma)$ be as in Proposition \ref{approximate.minimal}. For $z\in \mathcal{W}$, consider $\Sigma(z)$ the graph over $\Sigma_\sigma$ of $\varphi(z)$. By Proposition \ref{approximate.minimal}, there exists an open set  $U\subset \mathbb{RP}^3$, $\sigma \in U$,  {and a compact set $K\subset \mathcal{W}$} such that   for every $\sigma'\in U$, the minimal surface $\Sigma_{\sigma'}$ is equal to $\Sigma(z')$  for a {unique $z'\in K$. Let $F:U\rightarrow \mathcal{W}$ denote the map which assigns $z'$ to $\sigma'$. The map $F$ is injective.
The injectivity of $\phi$ implies $F$ is continuous and so we have from the Invariance of Domain Theorem that $\text{dim}(W)\geq \text{dim}(U)=3$.  Thus ${\rm dim}({\rm Ker}(L_\sigma))\geq 3$ and because the first eigenvalue is simple we also have that ${\rm index}(L_\sigma)\geq  1$.
}

Let $Y=\{\sigma\in \mathbb{RP}^3: {\rm index}(\Sigma_\sigma)=1\}$. {The set $Y$ is closed because the function $\sigma\in \mathbb{RP}^3\mapsto{\rm dim}({\rm Ker}(L_\sigma))$ is lower semicontinuous  and ${\rm index}(L_\sigma)\geq  1$ for all $\sigma \in \mathbb{RP}^3$.  We now argue that $Y$ is also open. The upper semicontinuous function 
$$\sigma\in \mathbb{RP}^3\mapsto{\rm dim}({\rm Ker}(L_\sigma))+{\rm index}(L_\sigma)$$
is greater than or equal to $4$ for all $\sigma\in \mathbb{RP}^3$. Since $\Sigma_\sigma$ is a  topological sphere, we have from \cite{cheng} that the function above is identical to $4$ for every $\sigma\in Y$. Thus, the upper semicontinuity property implies that $Y$ is open.
}

The Simon-Smith min-max theory and the index bounds of \cite{marques-neves-index} imply the existence of a minimal sphere with Morse index at most one. By the uniqueness theorem of \cite{galvez-mira}, this minimal sphere belongs to the Zoll family. Hence ${\rm index}(\Sigma_{\tilde{\sigma}})\leq 1$ for some $\tilde{\sigma}\in \mathbb{RP}^3.$  {We know that ${\rm index}(\Sigma_{\tilde{\sigma}})\geq 1$ and thus $\tilde{\sigma}\in Y$.
Hence $Y=\mathbb{RP}^3$, which implies 
$\Sigma_\sigma$ has Morse index  one and nullity equal to three for every $\sigma\in \mathbb{RP}^3$.}

Let $\sigma\in \mathbb{RP}^3$. The set $I_\sigma=\{\tau\in \mathbb{RP}^3: \Sigma_\tau {\rm \,  intersects\, } \Sigma_\sigma \}$
is closed and contains $\sigma$. We will show that $I_\sigma$ is open, and hence by connectedness $I_\sigma=\mathbb{RP}^3$.

Let $\tau\in I_\sigma$. If $\tau \neq \sigma$, by the definition of Zoll families the intersection of $\Sigma_\tau$ and $\Sigma_\sigma$ is transversal. Since this is an open condition, $\Sigma_{\tau'}$ intersects nontrivially and  transversally $\Sigma_\sigma$ for all $\tau'$ sufficiently close to $\tau$.
We will now show that for every $\sigma'$ sufficiently close to $\sigma$, the surface $\Sigma_{\sigma'}$ intersects $\Sigma_\sigma$. Suppose, by contradiction, that there
exists a sequence $\sigma_n\rightarrow \sigma$ such that $\Sigma_{\sigma_n}$ and $\Sigma_\sigma$ are disjoint for every $n$. This implies that $\Sigma_\sigma$ admits a
positive solution $\psi$ of $L_\sigma\psi=0$, which contradicts ${\rm index}(\Sigma_\sigma)=1$.

It remains to show the intersection $\Sigma_\sigma \cap \Sigma_\tau$, $\sigma \neq \tau$, is connected. If $\sigma_n \rightarrow \sigma$, $\sigma_n \neq \sigma$, there is $j\in \mathbb{N}$ such that if $n\geq j$ then $\Sigma_{\sigma_n}$ is the normal graph over $\Sigma_\sigma$ of a smooth function $\psi_n$. Then elliptic estimates imply that there is a sequence $t_n\rightarrow 0$ such that $t_n^{-1}\psi_n$ converges, say in the $C^2$ topology, to a nontrivial solution $\psi$ of $L_\sigma\psi=0$. Since $\psi$ is an eigenfunction of $L_\sigma$ corresponding to the second eigenvalue, its number of nodal domains is equal to two and by \cite{cheng}, the gradient of $\psi$ does not vanish on  $\{\psi=0\}$. The convergence 
$t_n^{-1}\psi_n\rightarrow \psi$ implies that  the nodal set $\{\psi_n=0\}$ is a perturbation of   $\{\psi=0\}$ and hence is connected for sufficiently large $n$. Therefore if $\sigma'\neq \sigma$ is sufficiently close to $\sigma$, the intersection
$\Sigma_{\sigma'}\cap \Sigma_\sigma$ is connected.
The sets
$$
\tilde{I}_\sigma=\{\tau\in \mathbb{RP}^3\setminus \{\sigma\}: \Sigma_\tau \cap \Sigma_\sigma {\rm \, is \, connected}\}
$$
and
$$
\{\tau\in \mathbb{RP}^3\setminus \{\sigma\}: \Sigma_\tau \cap \Sigma_\sigma {\rm \, is \, disconnected}\}
$$
are  open by transversality. Therefore $\tilde{I}_\sigma=\mathbb{RP}^3\setminus \{\sigma\}$ by connectedness, which finishes the proof.
\end{proof}

	This proposition can be used to imply  the main result of this section:
\begin{thm}\label{zoll.characterization.1}
 Let $g$ be a surface Zoll metric on $S^3$.  Then
$$
\sigma_1(S^3,g)=\sigma_2(S^3,g)=\sigma_3(S^3,g)=\sigma_4(S^3,g)=w.
$$
\end{thm}

\begin{proof}
Suppose that $\{\Sigma_\sigma\}_{\sigma\in \mathbb{RP}^{3}}$ is a Zoll family of minimal two-spheres for $(S^3,g)$, and{   recall $w$ is the area of $\Sigma_\sigma$, which is independent of $\sigma$.}

We know that  there exist a disjoint family $\{\tilde{\Sigma}_1^{(k)}, \dots, \tilde{\Sigma}_{q_k}^{(k)}\}$ of smooth, embedded $g$-minimal two-spheres and $\{m_1^{(k)}, \dots, m_{q_k}^{(k)}\}\subset \mathbb{N}$ such that
$$
\sigma_k(S^3,g)=\sum_{i=1}^{q_k}m^{(k)}_{i}{\rm area}(\tilde{\Sigma}^{(k)}_{i},g), \quad k=1,\ldots,4.
$$
The uniqueness result of Galvez-Mira \cite{galvez-mira} implies that each minimal two-sphere $\tilde{\Sigma}^{(k)}_{i}$ is an element of the Zoll family, and hence has area equal to $w$. {From Proposition \ref{zoll.properties} we have that $q_k=1$  and that  $\tilde{\Sigma}^{(k)}_{1}$ is unstable. Thus it follows from Wang-Zhou \cite{wang-zhou}  that $m^{(k)}_{1}=1$ and so $\sigma_k(S^3,g)=w$ for all $ k=1,\ldots,4$ as we wanted to prove.}
\end{proof}

\section{Rigidity for the spherical area widths}

In this section we will prove the following rigidity result.

{\begin{thm}\label{zoll.characterization.2}
{Let $g$ be a smooth Riemannian metric on $S^3$ so that 
$$
\sigma_1(S^3,g)=\sigma_2(S^3,g)=\sigma_3(S^3,g)=\sigma_4(S^3,g).
$$
Then  $g$ is a surface Zoll metric.}
\end{thm}}

Theorem B follows  by combining Theorems \ref{zoll.characterization.1} and  \ref{zoll.characterization.2}

We start by deriving several results which will be used in the proof.
\subsection{Part I} Let $\{\Psi_i\}_{i\in \mathbb{N}}\subset \mathcal{P}_4'$ be such that
$$
\sup_{x\in X_i}{\bf M}(\Psi_i(x))\rightarrow \sigma_4(S^3,g),
$$  
where $X_i={\rm dmn}(\Psi_i)$. By the argument of Proposition 3.2 of \cite{li-generic} with $p=4$, we can suppose that for some $k\in\mathbb{N}$,  $X_i$ is a four-dimensional cubical complex embedded in $I(9,k)${, the $9$-dimensional cubical complex on the $9$-dimensional unit cube $I^9$ whose $p$-cells all have support on a $p$-dimensional cube of volume $3^{-pk}$.}

Let $\mathcal{V}$ be the set of integral varifolds of the form $\sum_{j=1}^q m_j|\Sigma_j|$ for some $\{m_1,\dots,m_q\}\subset \mathbb{N}$ and $\{\Sigma_1,\dots,\Sigma_q\}$ disjoint, smooth, embedded, $g$-minimal spheres satisfying $\sum_j m_j {\rm area}(\Sigma_j,g)=\sigma$ and $\sum_j{\rm index}(\Sigma_j)\leq 1$.
For each $\delta>0$, we can suppose by refining $X_i$ that for each $j$-dimensional cube $t\in X_i$ we have
$$
{\bf F}(\Psi_i(x),\Psi_i(x'))< \delta/5
$$
for every $x,x'\in t$.
Let $Z_i$ be the union of cubes $t\in X_i$ with the property that  ${\bf F}(|\Psi_i(z)|,\mathcal{V})\leq 2\delta/5$ for every $z\in t$, and $Y_i$ to be  the union of cubes
$t\in X_i$ such that  ${\bf F}(|\Psi_i(y)|,\mathcal{V})\geq \delta/5$ for every $y\in t$. Hence $Z_i,Y_i$ are subcomplexes of $X_i$ with
$X_i=Z_i \cup Y_i$. 

Suppose that we can find a sequence $\{k\}\subset \{i\}$ such that $({\Psi_k}_{|Y_k})^*(\overline{\lambda}) \neq 0 \in H^1(Y_k,\mathbb{Z}_2)$.
Since $\sigma_1(S^3,g)=\sigma$ and $\limsup_{k\rightarrow \infty} \sup_{y\in Y_k}{\bf M}(\Psi_k(y))\leq \sigma$, we would have that $\{{\Psi_k}_{|Y_k}\}\subset \mathcal{P}_1'$ is an optimal sequence for $\sigma_1$.  Deformation Theorem A of \cite{li-index} (see also \cite{marques-neves-index}) holds in the Simon-Smith setting (since all modifications are by isotopies). Applied to $\{{\Psi_k}_{|Y_k}\}$ it produces a sequence $\{{\tilde{\Psi}_k}\}\subset \mathcal{P}_1'$ which is optimal 
for $\sigma_1$ and such that no smooth element of the critical set ${\bf C}(\{{\tilde{\Psi}_k}\})$ has support with Morse index greater than or equal to two. The proof of the deformation theorem (Theorem 10 of \cite{li-index} ) allows us to choose $\{{\tilde{\Psi}_k}\}$ so that ${\bf F}(|\tilde{\Psi}_k(y)|,\mathcal{V})\geq \delta/8$ for every $y\in Y_k$. This gives a contradiction since by the Simon-Smith genus bounds, the critical set ${\bf C}(\{{\tilde{\Psi}_k}\})$ should contain a smooth element (with integer multiplicities) whose support is a union of $g$-minimal spheres (which would have to be in $\mathcal{V}$). Therefore for sufficiently large $i$
one has $({\Psi_i}_{|Y_i})^*(\overline{\lambda}) = 0$. By the properties of the cup product, it follows that for such $i$
$$
({\Psi_i}_{|Z_i})^*(\overline{\lambda}^3)  \neq 0.
$$

We have proved that  there exists a sequence $\{\Phi_i\}_{i\in \mathbb{N}} \subset \mathcal{P}_3'$
such that
\begin{equation}\label{sequence.optimal}
\lim_{i\rightarrow \infty} \sup_{x\in X_i} {\bf F}(|\Phi_i(x)|,\mathcal{V})=0,
\end{equation}
where $X_i={\rm dmn}(\Phi_i)$. We can assume the $X_i$ are connected and three-dimensional, {because the restriction to the  $3$-dimensional complex is still a $3$-sweepout.} Notice that this sequence is necessarily optimal for $\sigma_3$ because $\sigma_3(S^3,g)=\sigma$. The critical set ${\bf C}(\{\Phi_i\})$ coincides with the image set ${\bf \Lambda}(\{\Phi_i\})$ and is contained in $\mathcal{V}$.

The following Proposition is analogous to Proposition \ref{no.disjoint.geodesic}.
\begin{prop}\label{no.disjoint}
Let $V\in {\bf C}(\{\Phi_i\})\subset \mathcal{V}$. There is no closed, smooth, embedded, minimal surface $\Sigma$ disjoint from ${\rm support}(V)$. Any component $S$ of the support of $V$ is  part of a local foliation  so that the leaves on each of the sides  of $S$  have mean curvature vector pointing strictly away from  $S$.
\end{prop}

\begin{proof}
Suppose, by contradiction, that there exists a connected, closed, smooth, embedded, minimal surface $\Sigma$ disjoint from ${\rm support}(V)$. It is known that any closed, embedded, minimal hypersurface is part of a local foliation whose compact leaves have either vanishing, positive or negative mean curvature. Let $\{\Sigma_t\}_{t\in (-2\delta,2\delta)}$ be such a foliation around $\Sigma$ ($\Sigma_0=\Sigma$.)  We can suppose that $\overline{\bigcup \Sigma_t}$ is disjoint from ${\rm support}(V)$.

The first case we consider is when, for some $-\delta<t'<\tilde{t}<\delta$, the complement $\Omega$ of $\big(\cup_{t\in [t',\tilde{t}]} \Sigma_t\big)$ has strictly mean-convex boundary. By \cite{white-maximum-principle}, there exists $\eta>0$ such that every stationary $2$-varifold of $\Omega$ is disjoint from $B_\eta(\partial \Omega)$.
Consider $K=\overline{\Omega}\setminus B_\eta(\partial\Omega)$ and let $\mathcal{K}$ be the set of elements in $\mathcal{V}$ with support contained in $K$. Let $\zeta>0$ be as in Lemma \ref{convergence.supports} for $K,\mathcal{K}$,$\eta$.  {We argue that 
\begin{equation}\label{inclusion.K}
B^{\bf F}_{\zeta}(\mathcal K)\cap \mathcal V\subset \mathcal K.
\end{equation}}
If $W$ is an element of $\mathcal{V}$ 
with support contained in $\overline{\Omega}$, then ${\rm support}(W)\subset K$ ($W\in \mathcal{K}$).
Therefore any element of $\mathcal{V}$ that is $\zeta$-close in the ${\bf F}$-metric to $W$  must have support contained
in $B_\eta(K)\subset \overline{\Omega}$.  Hence its support must be contained in $K$. 

Next we prove that for all  $j$ large enough
 \begin{equation}\label{inclusion.image}
 {\bf F}(|\Phi_j(x')|,\mathcal K)\leq \zeta/5\quad\mbox{for all }x'\in X_j 
 \end{equation}
 Because
 $V \in {\bf C}(\{\Phi_i\})$, there exist sequences $\{j\} \subset \{i\}$ and $\{x_j\in X_j\}$ with $|\Phi_j(x_j)|$ converging to $V$ in the
 ${\bf F}$-metric. For sufficiently large $j$ we have that ${\bf F}(|\Phi_j(x)|,\mathcal{V})\leq \zeta/5$ for every $x\in X_j$. Given
 $x'\in X_j$, the connectedness of $X_j$ implies that we can find $\{y_1,\dots, y_{q_j}\}\subset X_j$ with $y_1=x_j$, $y_{q_j}=x'$ and
 ${\bf F}(\Phi_j(y_k),\Phi_j(y_{k+1}))\leq \zeta/5$ for every $1\leq k\leq q_j-1$. We choose $V_{j,k}\in \mathcal{V}$ so that
 ${\bf F}(|\Phi_j(y_k)|,V_{j,k})\leq \zeta/5$ and $V_{j,1}=V$. Hence $\{V_{j,1}, \dots, V_{j,q_j}\}$ is a $\zeta$-chain in $\mathcal{V}$ starting at $V$. {We conclude from \eqref{inclusion.K} that $V_{j,k}\in\mathcal K$ for every $1\leq k\leq q_j$. This implies \eqref{inclusion.image}.}
 
{It follows from  \eqref{inclusion.image} that for every such $j$ and every  $x'\in X_j$ the varifold} $|\Phi_j(x')|$ is $\zeta/5$-close in the ${\bf F}$-metric to an element of $\mathcal{V}$ with support contained in $\overline{\Omega}$.  For any $\gamma>0$ and $t'<s'<\tilde{s}<\tilde{t}$, we can choose $\zeta$ sufficiently small such that 
 $$
 {\bf M}(\Phi_j(x') \llcorner \big(\cup_{t\in (s',\tilde{s})} \Sigma_t\big))<\gamma
 $$
 for every $x'\in X_j$. If $\gamma$ is sufficiently small, this contradicts the property that $\Phi_j$ is a one-sweepout (since it is a three-sweepout). This is because of the isoperimetric inequality and the fact that any one-sweepout contains a surface that divides a prescribed region in two parts of equal volume.  This proves that either every $\Sigma_t$,  $t\in (-\delta,\delta)$, has mean curvature vector pointing towards $\Sigma$ or there exist $0<\tilde{\delta}<\delta$ such that all $\Sigma_t$, $t\in (-\tilde{\delta},\tilde{\delta})$, have mean curvature of the same sign.  
 
	The second case we consider is when, for all $t\in (-\delta,\delta)$, the surface $\Sigma_t$ has mean curvature vector pointing towards $\Sigma$. By the maximum principle, any closed minimal surface $L \subset \big(\cup_{t\in (-\delta,\delta)} \Sigma_t\big)$ must be one of the leaves. Let $\eta>0$ be such that 
$$B_\eta(\Sigma)\subset \big(\cup_{t\in (-\delta,\delta)} \Sigma_t\big),
$$
$\lambda=\sigma+1$ and $\tilde{\eta}$ be as in Proposition \ref{intersects.must.be.contained}. Hence any connected minimal surface $\tilde{\Sigma}$, disjoint from $\Sigma$, with ${\rm area}(\tilde{\Sigma})+{\rm index}(\tilde{\Sigma})\leq \lambda$, that intersects $B_{\tilde{\eta}}(\Sigma)$ must be contained in $B_\eta(\Sigma)$ and hence coincide with one of the leaves $\Sigma_t$.

Consider $0<\tilde{\delta}<\delta$ such that $\big(\cup_{t\in [-\tilde{\delta},\tilde{\delta}]} \Sigma_t\big)\subset B_{\tilde{\eta}}(\Sigma)$.
Let ${\mathcal{V}_l}\subset \mathcal{V}$ be the set of varifolds $W\in \mathcal{V}$ such that every connected component of ${\rm support}(W)$ that intersects 
$\big(\cup_{t\in [-\tilde{\delta},\tilde{\delta}]} \Sigma_t\big)$ is a leaf.

We claim that $\mathcal{V}_l$ is {an open and closed subset of $\mathcal V$.} Let $W\in \mathcal{V}_l$, and consider a sequence $\{W_j\}\subset \mathcal{V}$ that converges to $W$. Each component of $W$ is either disjoint from  $\big(\cup_{t\in [-\tilde{\delta},\tilde{\delta}]} \Sigma_t\big)$ or coincides with a leaf in $\big(\cup_{t\in [-\tilde{\delta},\tilde{\delta}]} \Sigma_t\big)$. By Lemma \ref{convergence.supports},  there exists $\alpha>0$ such that, for sufficiently large $j$, ${\rm support}(W_j)$ is disjoint from 
$$\big(\cup_{t\in [-\tilde{\delta}-2\alpha,-\tilde{\delta}-\alpha]} \Sigma_t\big) \cup \big(\cup_{t\in [\tilde{\delta}+\alpha,\tilde{\delta}+2\alpha]} \Sigma_t\big),
$$
and $-\delta<-\tilde{\delta}-2\alpha< \tilde{\delta}+2\alpha<\delta$.
Therefore each connected component of ${\rm support}(W_j)$ that intersects $\big(\cup_{t\in [-\tilde{\delta},\tilde{\delta}]} \Sigma_t\big)$ must be contained in
$\big(\cup_{t\in [-\tilde{\delta}-\alpha,\tilde{\delta}+\alpha]} \Sigma_t\big)$ and hence be a leaf. This proves $\mathcal{V}_l$ is open.  Now suppose $\{W_j\}\subset \mathcal{V}_l$ is a sequence converging to $W\in \mathcal{V}$. By the monotonicity formula for minimal surfaces, there is a bound on the number of connected components of $W_j$ that is uniform in $j$. Let $\beta>0$ be such that $\big(\cup_{t\in [-\tilde{\delta}-\beta,\tilde{\delta}+\beta]} \Sigma_t\big)\subset B_{\tilde{\eta}}(\Sigma)$.   If a connected component of $W_j$ is disjoint from $\big(\cup_{t\in [-\tilde{\delta},\tilde{\delta}]} \Sigma_t\big)$ but intersects $\big(\cup_{t\in [-\tilde{\delta}-\beta,\tilde{\delta}+\beta]} \Sigma_t\big)$, it must be a leaf (since it is disjoint from $\Sigma$). Each connected component of $W_j$ is either disjoint from $\big(\cup_{t\in [-\tilde{\delta}-\beta,\tilde{\delta}+\beta]} \Sigma_t\big)$ or is a leaf $\Sigma_t$.
This implies each connected component of $W$ that intersects $\big(\cup_{t\in [-\tilde{\delta},\tilde{\delta}]} \Sigma_t\big)$ is a leaf. Hence $W\in \mathcal{V}_l$. This proves that $\mathcal{V}_l$ is closed. 

Since $\mathcal{V}$ is compact, $\mathcal{V}_l$ is compact. Therefore there exists $\zeta>0$ such that 
\begin{equation}\label{inclusion.K2}
B^{\bf F}_{\zeta}(\mathcal V_l)\cap \mathcal V\subset \mathcal V_l.
\end{equation}
Suppose that for some interval $[a,b]\subset (-\tilde{\delta},\tilde{\delta})$, one has that $\Sigma_t$ is not a minimal surface for any $t\in [a,b]$.
{It follows that the support of any element in $\mathcal V_l$ is disjoint from $\big(\cup_{t\in [a,b]} \Sigma_t\big)$. From \eqref{inclusion.K2} it follows that for every $\zeta$-chain $\{V_1,\ldots,V_q\}$ in $\mathcal{V}$ with $V_1=V\in  \mathcal V_l$ we have $\{V_1,\ldots,V_q\}\subset \mathcal V_l$ and thus all elements in the  $\zeta$-chain have support disjoint from $\big(\cup_{t\in [a,b]} \Sigma_t\big)$. } As before this contradicts the fact that $\Phi_j$ is a one-sweepout. Therefore we have proved that {if the second case holds then }$\Sigma_t$ is a minimal surface for any $t\in [-\tilde{\delta},\tilde{\delta}]$.

{The third case to consider is when all } $\Sigma_t$,  $t\in (-\delta,\delta)$, have mean curvature of the same sign. Using again  the maximum principle we see that any closed minimal surface $\Sigma \subset \big(\cup_{t\in (-\delta,\delta)} \Sigma_t\big)$ must be one of the leaves. {We can proceed as in the second case to} get a contradiction if there is no foliation of $\Sigma$ by minimal surfaces.

{The conclusion is that if $\Sigma$ is a minimal surface disjoint from $V$ then $\Sigma$ is part of a local foliation by minimal surfaces. Each leaf is  stable since it admits a positive Jacobi field  and they all have the same area.}

 Notice that {the arguments used in the second and third case} can also be applied to a component
$S$ of the support of $V$ instead of $\Sigma$. They prove  that  $S$  is part of a local foliation $\{S_t\}_{t\in(-a,a)}$ such that on each of its sides the mean curvature vector of any leaf points away from $S$.  {We now show that either every leaf is minimal or every $S_t$ is strictly mean convex for all $0<|t|<a$.  

Let $I\subset(-a,a)$ the subset of those $t$ for which $S_t$ is minimal. If $t\in I$ and $t\neq 0$, then $S_{t}$ is a minimal surface disjoint from $V$ which admits a local foliation as in the third case considered above. Therefore $S_{t}$ admits a local foliation by minimal surfaces and the maximum principle implies that this local foliation has to coincide with $\{S_{t'}\}_{t'\in(t-r,t+r)}$ for some small $r$. So $I-\{0\}$ is open.  The space of stable, connected, minimal surfaces with a given area is compact by the curvature estimates of Schoen \cite{schoen} (also \cite{schoen-simon}, \cite{sharp}) and so $I$ is  closed. Therefore $I$ is either $\{0\}$, $(-a,0]$, $[0,a)$, or $(-a,a)$. If $I=(-a,0]$ or $I=[0,a)$ then $S$ would fall into the third case considered above and we would get a contradiction. Thus either  $I=\{0\}$ or $I=(-a,a)$, which is what we wanted to show.
}

{By showing that every minimal surface must intersect $V$, we obtain that on each side of a connected component $S$ of  ${\rm support}(V)$ every leaf has mean curvature vector pointing strictly away from $S$ and this completes the proof of the Proposition. 

The work we have done shows the existence of } a foliation $\{\Sigma_t\}_{t\in [-s,s]}$, $\Sigma_0=\Sigma$, around $\Sigma$, by stable minimal surfaces.  Notice that $\Sigma$ separates $S^3$ into two connected components.  {From the fact that the space of stable, connected, minimal surfaces with a given area is compact we see that }
the foliation can be extended until it touches ${\rm support}(V)$, in which case the extreme leaves coincide with connected components of 
${\rm support}(V)$ by the maximum principle, or until it converges to a  closed one-sided minimal surface with multiplicity two.  Since there are no one-sided closed embedded surfaces in $S^3$, we have constructed a foliation $\{\Sigma_t\}_{t\in [-s,s]}$, $\Sigma_0=\Sigma$, by minimal surfaces such that $\Sigma_{-s}$ and $\Sigma_s$ are connected components of ${\rm support}(V)$. {These connected components of ${\rm support}(V)$ are the endpoints of a foliation by minimal surfaces and so the foliation can be extended} past any component of the support of $V$, which is a contradiction since the ambient space is compact.  This proves that there is no closed minimal surface disjoint from the support of $V$.
\end{proof}

\begin{prop}\label{connected.support}
Every element of the critical set ${\bf C}(\{\Phi_i\})\subset \mathcal{V}$ has connected support. 
\end{prop}

\begin{proof}
Suppose, by contradiction, that some $V \in {\bf C}(\{\Phi_i\})$ has disconnected support.   Then  it follows that $V=\sum_{j=1}^q m_j|\Sigma_j|$ for some $\{m_1,\dots,m_q\}\subset \mathbb{N}$ and $\{\Sigma_1,\dots,\Sigma_q\}$ disjoint, smooth, embedded, minimal spheres satisfying $\sum_j m_j {\rm area}(\Sigma_j,g)=\sigma$ and $\sum_j{\rm index}(\Sigma_j)\leq 1$ with $q\geq 2$. Proposition \ref{no.disjoint} implies  that the surface $\Sigma_j$, $1\leq j \leq q$, is part of a local foliation such that on each of its sides the mean curvature vector of every leaf 
points strictly away from $\Sigma_j$. Hence any connected component $\Omega$ of $S^3\setminus \cup_j\Sigma_j$ with disconnected $\partial \Omega$ must  contain a minimal surface in its interior (by area-minimization). This contradicts Proposition \ref{no.disjoint}. Hence ${\rm support}(V)$ is connected.
\end{proof}

\begin{prop}
Every element of the critical set ${\bf C}(\{\Phi_i\})\subset \mathcal{V}$ has multiplicity one.
\end{prop}

\begin{proof}
Let $V \in {\bf C}(\{\Phi_i\})$. Then, by Proposition \ref{connected.support}, $\Sigma= {\rm support}(V)$ is connected. Hence $V=m \cdot |\Sigma|$ for some $m\in \mathbb{N}$. Proposition \ref{no.disjoint} implies  that a neighborhood of  $\Sigma$ admits  a foliation $\{\Sigma_t\}_{t\in [-1,1]}$, $\Sigma_0=\Sigma$, such that the mean curvature vector of $\Sigma_t$ points strictly away from $\Sigma$ for $t\neq 0$. Let $Q$ be the three-ball disjoint from $\Sigma$ and bounded by $\Sigma_1$. We can consider one-parameter sweepouts of $Q$ by two-spheres starting at $\partial Q$ as in  Section 2 of \cite{marques-neves-duke}. We can use Theorem 2.1 of \cite{marques-neves-duke} to conclude the associated min-max width is at most ${\rm area}(\partial Q)$, since there cannot be a minimal surface disjoint from $\Sigma$. We can repeat this for the three-ball $\tilde{Q}$ disjoint from $\Sigma$ and bounded by $\Sigma_{-1}$. By concatenating almost optimal sweepouts for these min-max schemes with the foliation $\{\Sigma_t\}_{t\in [-1,1]}$, we construct an element of $\mathcal{P}_1'$ with maximal area equal to ${\rm area}(\Sigma)$. Therefore $\sigma=\sigma_1\leq {\rm area}(\Sigma)$. But
$\sigma={\bf M}(V)=m\cdot  {\rm area}(\Sigma)$. Hence we have proved that $m=1$.
\end{proof}

\begin{prop}\label{limit.connected.multiplicity.one}
Every element of the critical set ${\bf C}(\{\Phi_i\})\subset \mathcal{V}$ is of the form $V=|\Sigma|$, where $\Sigma$ is a smooth, embedded, $g$-minimal two-sphere with Morse index equal to one and nullity at least one and at most three.
\end{prop}

\begin{proof}
Choose $V\in {\bf C}(\{\Phi_i\})$. We have proved that $V=|\Sigma|$, where $\Sigma$ is a smooth, embedded, $g$-minimal two-sphere with Morse index at most one. Suppose, by contradiction, that  $\Sigma$ is stable. By Proposition \ref{no.disjoint}, there is a foliation $\{\Sigma_t\}_{t\in [-1,1]}$, $\Sigma_0=\Sigma$, such that the mean curvature vector of $\Sigma_t$ points strictly away from $\Sigma$ for $t\neq 0$. If $V_j=|S_j|\in {\bf C}(\{\Phi_i\})$  is a sequence converging to $V$ in varifold sense, $S_j$ converges smoothly to $\Sigma$ by Allard's Regularity Theorem. Suppose there is  $\{S_k\}_k\subset \{S_j\}_j$ such that $S_k\neq \Sigma$   for every $k$. Using the foliation and the maximum principle, we conclude that, for sufficiently large $k$, $S_k$ intersects $\Sigma$. This would imply the existence of a nontrivial solution $v$  of the linearized equation 
$L_\Sigma v=0$ that changes sign. This is not possible since we assumed that  $\Sigma$ is stable. Hence there is $\zeta>0$ such that any $W\in {\bf C}(\{\Phi_i\})$ with 
${\bf F}(W,V)\leq \zeta$ satisfies $W=V$. Hence any $\zeta$-chain $\{V_1,\dots,V_l\}\subset {\bf C}(\{\Phi_i\})$ with $V_1=V$ must satisfy $V_q=V$ for every $1\leq q\leq l$. As before, this contradicts the fact that $\{\Phi_i\}$ are one-sweepouts. This finishes the proof that $\Sigma$ is unstable, and hence has Morse index one. The nullity of
$\Sigma$ is at most three by \cite{cheng}. The nullity has to be at least one because if $\Sigma$ is nondegenerate, then again there is $\zeta>0$ such that any $W\in {\bf C}(\{\Phi_i\})$ with 
${\bf F}(W,V)\leq \zeta$ satisfies $W=V$. 
\end{proof}

\subsection{Part II}\label{part2}

{Fix  a compact set in the varifold topology $\mathcal{V}_1\subset \mathcal{V}$}  such that every element of $\mathcal{V}_1$  is of the form $V=|\Sigma|$, where $\Sigma$ is a smooth, embedded, minimal two-sphere with Morse index equal to one and nullity at least one and at most three. We can consider  the elements of $\mathcal{S}$ and $\mathcal{V}_1$ also as mod two flat cycles. Recall the definition of the ${\bf F}$-metric: ${\bf F}(T_1,T_2)=\mathcal{F}(T_1,T_2)+{\bf F}(|T_1|,|T_2|)$. Then $\mathcal{V}_1$ is compact for the ${\bf F}$-metric by Allard's Regularity Theorem.

 Let  $k\geq 2$ be an integer and $\alpha\in (0,1)$.  Let $\mathcal{S}^{k,\alpha}$ be the space of $C^{k,\alpha}$ embedded two-spheres in $S^3$.  For each $|\Sigma|\in \mathcal{V}_1$, let $\mathcal{U}_\delta^{k,\alpha}(\Sigma)$ be the set of spheres in $\mathcal{S}^{k,\alpha}$ that can be written as a graph over $\Sigma$ of a function $f$
with $|f|_{k,\alpha}<\delta$. Here we use the induced Riemannian structure on $\Sigma$ to define $|f|_{k,\alpha}$.

\begin{prop}\label{limit.flat.topology}
Let $\Phi_i:X_i\rightarrow \mathcal{S}^{k,\alpha}$ be a sequence of continuous maps, $X_i$ a three-dimensional compact simplicial complex, which are three-sweepouts. Suppose ${\bf \Lambda}(\{\Phi_i\})\subset \mathcal{V}_1$.
The sequence $\{\Phi_i\}_i$ satisfies
$$
\lim_{i\rightarrow \infty} \sup_{x\in X_i} {\bf F}(\Phi_i(x),\mathcal{V}_1)=0.
$$
\end{prop}

\begin{proof}
Let $T_i$ be  a sequence of mod two flat cycles with $|T_i|\rightarrow \Sigma \in \mathcal{V}_1$ in the varifold topology. If $\{T_j\}_j$ is a subsequence with $T_j\rightarrow T$ in the flat topology, then ${\rm support}(T)\subset \Sigma$. By the Constancy Theorem either $T=\Sigma$ or $T=0$. Therefore, by compactness in the flat topology of the space of cycles with uniformly bounded mass, for every $\eta>0$ there exists $0<\delta<\eta$ such that if ${\bf F}(|T|,\Sigma)<\delta$ for some $\Sigma \in \mathcal{V}_1$ then either $\mathcal{F}(T,\Sigma)<\eta/2$ or $\mathcal{F}(T,0)<\eta/2$. Let  $\eta>0$ be  such that
$\mathcal{F}(\mathcal{V}_1,0)\geq 2\eta$. Let $i$ be such that ${\bf F}(|\Phi_i(x)|,\mathcal{V}_1)<\delta$ for every $x\in X_i$. Suppose there exists $\tilde{x}\in X_i$
such that $\mathcal{F}(\Phi_i(\tilde{x}),0)<\eta/2$. Since $X_i$ is path-connected and $\Phi_i$ is continuous in the ${\bf F}$-metric, it follows that $\mathcal{F}(\Phi_i(x),0)<\eta/2$ for every $x\in X_i$. This is a contradiction if $\eta$ is sufficiently small, by Lemma \ref{close.implies.homotopic}, because $\Phi_i$ is a three-sweepout. Hence ${\bf F}(\Phi_i(x),\mathcal{V}_1)<\delta+\eta/2$  for every $x\in X_i$.
This proves the  proposition.
\end{proof}

From Allard's Regularity Theorem and compactness of $\mathcal{V}_1$ in the varifold topology, for each $\delta>0$ there exists $0<\xi=\xi(\delta)<\delta$ such that if 
$|\Sigma|,|\Sigma'|\in \mathcal{V}_1$ satisfy ${\bf F}(|\Sigma|,|\Sigma'|)<\xi$ then $\mathcal{U}_\xi^{k,\alpha}(\Sigma) \subset \mathcal{U}_\delta^{k,\alpha}(\Sigma')$ and
$\mathcal{U}_\xi^{k,\alpha}(\Sigma')\subset \mathcal{U}_\delta^{k,\alpha}(\Sigma)$. Also, for each $\xi>0$ there exists $0<\theta=\theta(\xi)<\xi$ such that if $\Sigma'\in \mathcal{S}^{k,\alpha}$ satisfies $\Sigma'\in \mathcal{U}_\theta^{k,\alpha	}(\Sigma)$ for some $\Sigma \in \mathcal{V}_1$, then ${\bf F}(\Sigma',\Sigma)<\xi$.

For {$\Sigma\in \mathcal{V}_1$}, Proposition \ref{approximate.minimal} gives a smooth embedding ${\varphi}_\Sigma:\overline{B}^{3}\rightarrow \mathcal{S}^{k,\alpha}$, where $\overline{B}^{3}$ is the closed unit ball in $\mathbb{R}^3$, and $\eta_\Sigma>0$, such that ${\varphi}_\Sigma(0)=\Sigma$ and so that every closed minimal surface $\Sigma'$ with ${\bf F}(|\Sigma'|,|\Sigma|) \leq  \eta_\Sigma$  must belong to 
${\mathcal{W}}_\Sigma={\varphi}_\Sigma(\overline{B}^{3})$.  We can choose $\eta_\Sigma$ so that  ${\bf F}(|\Sigma'|,|\Sigma|) \geq  2\eta_\Sigma$ for every $\Sigma'\in \partial {\mathcal{W}}_\Sigma={\varphi}_\Sigma(\partial \overline{B}^{3})$. 

\begin{prop}\label{ruling.out.1.0}
Let $\Sigma \in \mathcal{V}_1$. Given $\xi>0$ there is $0<\bar \xi<\xi/5$ so that for every continuous three-sweepout $\Phi:X\rightarrow \mathcal{S}^{k,\alpha}$  defined on  a three-dimensional compact simplicial complex $X$ with 
$$
 \sup_{x\in X} {\bf F}(\Phi(x),\mathcal{V}_1)<\bar \xi, 
$$
we can find a continuous three-sweepout $\Phi':X\rightarrow \mathcal{S}^{k,\alpha}$ so that
\begin{itemize}
\item ${\bf F}(\Phi'(x),\Phi(x))<2\xi$ for every $x\in X$;
\item $
\Phi'(X)\subset \mathcal{Z}=\{S\in \mathcal{S}^{k,\alpha}:{\bf F}(|S|,|\Sigma|)>\eta_\Sigma/3\}   \, \bigcup \, {\mathcal{W}}_\Sigma.
$
\end{itemize}
\end{prop}

\begin{proof}

{\bf Claim:} Every continuous map $\Psi:\partial B^q\rightarrow \mathcal{S}^{k,\alpha}$, $q\geq 1$ an integer, so that for some $\Sigma\in \mathcal{V}_1$ and $\eta>0$, it holds that $\Psi(x)\in \mathcal{U}_\eta^{k,\alpha}(\Sigma)$ for every $x\in \partial B^q$, has  a continuous extension
$\tilde{\Psi}:B^q \rightarrow \mathcal{S}^{k,\alpha}$ with $\tilde{\Psi}(x)\in \mathcal{U}_\eta^{k,\alpha}(\Sigma)$ for every $x\in  B^q$.

A continuous map $\Psi:\partial B^q\rightarrow \mathcal{U}_\eta^{k,\alpha}(\Sigma)$ can be written as 
 $$\Psi(x)={\rm graph}_\Sigma(h(x))$$ with $h:\partial B^q \rightarrow \{f\in C^{k,\alpha}(\Sigma): |f|_{k,\alpha}<\eta\}$. For $x\in \partial B^q$ and $0<t\leq 1$, we define $\tilde{\Psi}(tx)={\rm graph}_\Sigma(th(x))$ and
$\tilde{\Psi}(0)=\Sigma$. This defines a continuous  extension $\tilde{\Psi}:B^q\rightarrow \mathcal{S}^{k,\alpha}$ of $\Psi$. Note that $\tilde{\Psi}(B^q)\subset \mathcal{U}_\eta^{k,\alpha}(\Sigma)$.  This proves the claim.

{For $\kappa>0$, there exists $0<\tilde{\theta}=\tilde{\theta}(\kappa)<\kappa$ such that if ${\varphi}_\Sigma(w)\in \mathcal{U}^{k,\alpha}_{\tilde{\theta}}({\varphi}_\Sigma(v))$ then $|v-w|<\kappa$. For $\theta>0$, there exists $0<\kappa=\kappa(\theta)<\theta$ such that if $v,w\in \overline{B}^{j(\Sigma)}$ satisfy 
$|v-w|<\kappa$ and $|{\varphi}_\Sigma(w)| \in  \mathcal{V}_1$, then ${\varphi}_\Sigma(v)\in \mathcal{U}^{k,\alpha}_\theta({\varphi}_\Sigma(w))$.}

 Let $X'=\{x\in X: {\bf F}( |\Phi(x)|, |\Sigma|)< \eta_\Sigma/2\}$ and $\delta>0$ be as in Lemma \ref{close.implies.homotopic}. Let $0<\xi<\delta/2$ and  $\theta=\theta(\xi)$. Let $\xi_1=\xi(\theta)$ and $\theta_1=\theta(\xi_1)$.  Let $\xi_2=\xi(\theta_1)$ and $\theta_2=\theta(\xi_2)$. Let $\kappa=\kappa(\theta_2)$.   Let $0<\theta_3<\min\{\eta_\Sigma, \tilde{\theta}(\kappa)\}$, and $\xi_3=\xi(
\theta_3)$. {Assume that } 
$$
 \sup_{x\in X} {\bf F}(\Phi(x),\mathcal{V}_1)<\xi_3/5=\bar \xi<\xi/5. 
$$
We can consider a barycentric subdivision $X(l)$ of $X$ such that 
$$
\mathcal{F}(\Phi(x),\Phi(y))+ {\bf F}(|\Phi(x)|,|\Phi(y)|)<\xi_3/5
$$
for every face $\alpha\in X(l)$, $x,y\in \alpha$.  If $x$ is a vertex of $X(l)$, choose $\Sigma_x\in \mathcal{V}_1$ with 
${\bf F}(\Phi(x),\Sigma_x)<\xi_3/5$ and define $\Phi'(x)=\Sigma_x$. Notice that for every vertex $x$, ${\bf F}(\Phi'(x),\Phi(x))<\xi_3/5$.

 If $x,y$ are vertices of the same face, then ${\bf F}(\Sigma_x,\Sigma_y)<\xi_3<\xi_2<\xi_1$.  Hence $\mathcal{U}^{k,\alpha}_{\xi_2}(\Sigma_x)\subset \mathcal{U}^{k,\alpha}_{\theta_1}(\Sigma_y)$, $\mathcal{U}^{k,\alpha}_{\xi_2}(\Sigma_y)\subset \mathcal{U}^{k,\alpha}_{\theta_1}(\Sigma_x)$ and $\mathcal{U}^{k,\alpha}_{\xi_1}(\Sigma_x)\subset \mathcal{U}^{k,\alpha}_{\theta}(\Sigma_y)$ and $\mathcal{U}^{k,\alpha}_{\xi_1}(\Sigma_y)\subset \mathcal{U}^{k,\alpha}_{\theta}(\Sigma_x)$. Also if $\sigma$ is a face that is contained in a face $\tilde{\sigma}$ that intersects $X'$, then for any vertex
$x$ of $\sigma$, ${\bf F}(|\Sigma_x|,|\Sigma|)<\eta_\Sigma$. This implies that $\Sigma_x\in {\mathcal{W}}_\Sigma$. Define
$\tilde{X}$ to be the complex formed by the union of all faces $\sigma$ which are contained in some face that intersects $X'$.

Let $[x,y]$ be a one-dimensional face of $X(l)$ ($x,y$ are vertices). Notice that ${\bf F}(\Sigma_x,\Sigma_y)<\xi_3$ implies $\Sigma_y\in \mathcal{U}_{\theta_3}^{k,\alpha}(\Sigma_x)$. If $[x,y]\in \tilde{X}$, then $\{\Sigma_x,\Sigma_y\}\subset {\mathcal{W}}_\Sigma$. Hence $\Sigma_x={\varphi}_\Sigma(v)$, $\Sigma_y={\varphi}_\Sigma(w)$, and we set
$\Phi'((1-t)x+ty)={\varphi}_\Sigma((1-t)v+tw)$, $t\in [0,1]$. Since $|v-w|<\kappa$, we have that $\Phi'(z) \in \mathcal{U}^{k,\alpha}_{\theta_2}(\Sigma_x)$ for every $z\in [x,y]$.  If $[x,y]$ is not a  one-dimensional face of $X'$, since $\Sigma_y\in \mathcal{U}_{\theta_3}^{k,\alpha}(\Sigma_x)$ we can use the claim to extend
$\Phi'$ continuously to $[x,y]$ so that $\Phi'([x,y])\subset  \mathcal{U}_{\theta_3}^{k,\alpha}(\Sigma_x)$. In any case,
${\bf F}(\Phi'(z),\Sigma_x)<\xi_2$ for every $z\in [x,y]$. 

Let $\sigma$ be a two-dimensional face of $X(l)$.  Suppose $\sigma \in \tilde{X}$. Then $\partial\sigma \subset \tilde{X}$.
If the vertices $x_k$ of $\sigma$, $1\leq k\leq 3$, satisfy $\Sigma_{x_k}={\varphi}_\Sigma(v_k)$, we can define
$\Phi'(\sum_k t_k x_k)={\varphi}_\Sigma(\sum_k t_k v_k)$ for $t_k\in [0,1]$, $\sum_kt_k=1$. This is an extension of $\Phi'$ to $\sigma$. Notice that $|\sum_k t_k v_k - v_1|< \kappa$. Hence $\Phi'(z) \in \mathcal{U}^{k,\alpha}_{\theta_2}(\Sigma_{x_1})$ for every $z\in\sigma$. Suppose $\sigma \notin \tilde{X}$. Since $\theta_2<\xi_2$, if we pick a vertex $x$ of $\sigma$ it follows that
$\Phi'(z)\in \mathcal{U}^{k,\alpha}_{\theta_1}(\Sigma_x)$ for every $z\in \partial \sigma$. Hence by the claim we can extend
$\Phi'$ to $\sigma$ such that $\Phi'(z)\in \mathcal{U}^{k,\alpha}_{\theta_1}(\Sigma_x)$ for every $z\in \sigma$. In any case, 
there is a vertex $x$ of $\sigma$ such that ${\bf F}(\Phi'(z),\Sigma_x)<\xi_1$ for every $z\in \sigma$.

Suppose $\sigma$ is a three-dimensional face of $X(l)$. If $\sigma \in \tilde{X}$, then $\partial\sigma \subset \tilde{X}$.
If the vertices $x_k$ of $\sigma$, $1\leq k\leq 4$, are such that  $\Sigma_{x_k}={\varphi}_\Sigma(v_k)$, we can define
$\Phi'(\sum_k t_k x_k)={\varphi}_\Sigma(\sum_k t_k v_k)$ for $t_k\in [0,1]$, $\sum_kt_k=1$. This is an extension of $\Phi'$ to $\sigma$. Notice that $|\sum_k t_k v_k - v_1|< \kappa$. Hence $\Phi'(z) \in \mathcal{U}^{k,\alpha}_{\theta_2}(\Sigma_{x_1})$ for every $z\in\sigma$. Suppose then that $\sigma \notin \tilde{X}$. Since $\theta_1<\xi_1$, if we pick a vertex $x$ of $\sigma$ then
$\Phi'(z)\in \mathcal{U}^{k,\alpha}_{\theta}(\Sigma_x)$ for every $z\in \partial \sigma$. By the claim we can extend
$\Phi'$ to $\sigma$ such that $\Phi'(z)\in \mathcal{U}^{k,\alpha}_{\theta}(\Sigma_x)$ for every $z\in \sigma$. In both cases, 
there is a vertex $x$ of $\sigma$ such that ${\bf F}(\Phi'(z),\Sigma_x)<\xi$ for every $z\in \sigma$.

Hence we have defined a continuous map $\Phi':X \rightarrow \mathcal{S}^{k,\alpha}$. If $z\in \sigma$, $\sigma \in X(l)$, there is a vertex
$x$ of $\sigma$ such that ${\bf F}(\Phi'(z),\Sigma_x)<\xi$. Hence
$$
{\bf F}(\Phi'(z),\Phi(z))\leq {\bf F}(\Phi'(z),\Sigma_x)+ {\bf F}(\Sigma_x, \Phi(x))+{\bf F}(\Phi(x),\Phi(z))<2\xi
$$
for every $z\in X$. Since $2\xi<\delta$, the map $\Phi'$ is a three-sweepout.

Notice that $\Phi'(\tilde{X})\subset {\mathcal{W}}_\Sigma$. If $z\in X\setminus \tilde{X}$, let $\sigma\in X(l)$ be a face with $z\in \sigma$. Then $\sigma$ is disjoint from $X'$.  Hence $z\notin X'$ which implies ${\bf F}(|\Phi(z)|,|\Sigma|)\geq \eta_\Sigma/2$. Therefore
$$
\Phi'(X)\subset \mathcal{Z}=\{S\in \mathcal{S}^{k,\alpha}:{\bf F}(|S|,|\Sigma|)>r\}   \, \bigcup \, {\mathcal{W}}_\Sigma,
$$
with $r=\eta_\Sigma/3$.
\end{proof}

Let $\mathcal{G}\subset \mathcal{V}$ be  the set of $\Sigma\in \mathcal{V}$ such that ${\varphi}_\Sigma$ and $\eta_\Sigma$ can be chosen so that ${\varphi}_\Sigma(z)$ is a minimal surface for all  $z\in \overline{B}^3$.  The set $\mathcal{G}$ is open in $\mathcal{V}$, and hence ${\mathcal{B}}=\mathcal{V}_1\setminus \mathcal{G}$  is compact in the $C^{k,\alpha}$-topology. If $\Sigma\in \mathcal{B}$, there is a sequence $v_i\in \overline{B}^3$, $v_i\rightarrow 0$, depending on $\Sigma$, such that ${\varphi}_\Sigma(v_i)$ is not a minimal surface for every $i\geq 1$.

For $\Sigma \in \mathcal{B}$, choose $t_\Sigma$ to be a radius $s$ so that every  $\Sigma'\in {\varphi}_\Sigma(\overline{B}^{3}_s)$ satisfies ${\bf F}(|\Sigma'|,|\Sigma|)\leq \eta_\Sigma/8$ and there exists $z_\Sigma\in \partial \overline{B}^3_s$ such that ${\varphi}_\Sigma(z_\Sigma)$ is not a minimal surface. Let $0<\kappa_\Sigma<\eta_\Sigma$ such that  
${\bf F}(|{\varphi}_\Sigma(z_\Sigma)|,\mathcal{V}_1)\geq \kappa_\Sigma$. Let $0<\gamma_\Sigma<\eta_\Sigma/8$ be such that ${\bf F}(|\Sigma'|,|\Sigma|)\geq \gamma_\Sigma$ for every $\Sigma'$ satisfying
$\Sigma'\in {\varphi}_\Sigma(\overline{B}^{3}\setminus B^{3}_{t_\Sigma/2})$.

\begin{prop}\label{ruling.out.2}
Let $\Sigma \in \mathcal{B}$. Suppose $\Phi_i:X_i\rightarrow \mathcal{S}^{k,\alpha}$ is a sequence of continuous maps, $X_i$ a three-dimensional compact simplicial complex, which are three-sweepouts, with  ${\bf \Lambda}(\{\Phi_i\})\subset \mathcal{V}_1$. There exists a sequence of continuous maps $\Psi_i:Y_i\rightarrow \mathcal{S}^{k,\alpha}$, $Y_i$ a  three-dimensional compact simplicial complex, which are  three-sweepouts and such that
$$
{\bf \Lambda}(\{\Psi_i\}_i) \subset {\bf \Lambda}(\{\Phi_i\}_i)  \cap \left(B_{\bf F}(|\Sigma|,\gamma_\Sigma)\right)^c.
$$ 
\end{prop}
 
\begin{proof}
Let $\xi>0$ {and apply Proposition \ref{ruling.out.1.0} to each $\Phi_i$ with $i$ sufficiently large to obtain a sequence of continuous maps
$\Phi_i':X_i\rightarrow \mathcal{S}^{k,\alpha}$ which are three-sweepouts.}  The sequence satisfies
$$
{\bf F}(\Phi_i'(z),\Phi_i(z))< 2\xi
$$
for every $z\in X_i$, and 
$$
\Phi_i'(X_i)\subset \mathcal{Z}=\{S\in \mathcal{S}^{k,\alpha}:{\bf F}(|S|,|\Sigma|)>\eta_\Sigma/3\}   \, \bigcup \, {\mathcal{W}}_\Sigma.
$$
Suppose $\xi<\kappa_\Sigma/3$, so that ${\varphi}_\Sigma(z_\Sigma)\notin \Phi_i'(X_i).$

 We consider the space $\mathcal{Z}$ with the $C^{k,\alpha}$ induced topology.  Let   $\mathcal{X}_1$ be a sufficiently small neighborhood of $\Phi_i'(X_i) \cap \left(\mathcal{Z}\setminus  {\varphi}_\Sigma(B^{3}_{t_\Sigma})\right)$ in $\mathcal{Z}$, $\mathcal{X}_2={\varphi}_\Sigma(B^{3}_{t_\Sigma})$ and $\mathcal{X}=\mathcal{X}_1\cup \mathcal{X}_2$. Note that $\mathcal{X}_1$ and $\mathcal{X}_2$ are open subsets of $\mathcal{Z}$ and hence of $\mathcal{X}$. We can construct $\mathcal{X}_1$ such that 
 $$
 {\varphi}_\Sigma^{-1}(\mathcal{X}_1\cap\mathcal{X}_2)=\{ s\cdot z: z\in \Omega, s\in (1-q,1)\}
 $$
 for some $0<q<1/2$ and some neighborhood $\Omega$ of  ${\varphi}_\Sigma^{-1}(\Phi_i'(X_i))\cap \partial \overline{B}^3_{t_\Sigma}$ in 
 $\partial \overline{B}^3_{t_\Sigma}$ satisfying $z_\Sigma \notin \Omega$.  We assume without loss of generality that $\Omega\neq \emptyset$.
  
    The Mayer-Vietoris sequence in cohomology with $\mathbb{Z}_2$ coefficients reads as
$$
\cdots \rightarrow H^2(\mathcal{X}_1\cap \mathcal{X}_2)\rightarrow H^3(\mathcal{X}) \rightarrow H^3(\mathcal{X}_1) \oplus H^3(\mathcal{X}_2)\rightarrow H^3(\mathcal{X}_1\cap \mathcal{X}_2)\rightarrow \cdots
$$
Since $\mathcal{X}_1\cap \mathcal{X}_2$ has the homotopy type of $\Omega$, which is homeomorphic to a noncompact manifold of dimension $2$, it follows that $H^q(\mathcal{X}_1\cap\mathcal{X}_2)=0$ for every $q\geq 2$. Hence the restriction homomorphism
$$
H^3(\mathcal{X}) \rightarrow H^3(\mathcal{X}_1) \oplus H^3(\mathcal{X}_2)
$$
is an isomorphism. But $H^3(\mathcal{X}_2)=0$, so $i_1^*:H^3(\mathcal{X}) \rightarrow H^3(\mathcal{X}_1)$ is an isomorphism ($i_1:\mathcal{X}_1\rightarrow\mathcal{X}$ is the inclusion map).

Notice that $\Phi_i'(X_i)\subset \mathcal{X}$. As before, we can construct a continuous map
$\Psi_i:Y_i\rightarrow \mathcal{X}_1\subset \mathcal{S}^{k,\alpha}$, from some three-dimensional compact simplicial complex $Y_i$, that is a three-sweepout. A diagonal argument proves the proposition.
\end{proof}

\subsection{Part III} Consider the sequence $\{\Phi_i\}_{i\in \mathbb{N}} \subset \mathcal{P}_3'$ given by \eqref{sequence.optimal}. From Proposition \ref{limit.connected.multiplicity.one} we see that we can take the compact set $\mathcal{V}_1$ considered in Section \ref{part2}  to be $\mathcal{V}_1={\bf \Lambda}(\{\Phi_i\}_i)={\bf C}(\{\Phi_i\}_i)$.

\begin{prop}\label{refining.sequence.2}
There exists a sequence of continuous three-sweepouts  $\{\tilde{\Psi}_i:\tilde{Y}_i\rightarrow \mathcal{S}^{k,\alpha}\}_i$, where $\tilde{Y}_i$ are three-dimensional compact simplicial complexes,  such that for  a compact set $\tilde{\mathcal{G}}\subset \mathcal{G}$,
$$
\lim_{i\rightarrow \infty} \sup_{x\in \tilde{Y}_i} {\bf F}(|\tilde{\Psi}_i(x)|,\tilde{\mathcal{G}})=0.
$$
\end{prop}

\begin{proof}
The set $\mathcal{B}\subset \mathcal{V}_1$ is compact in the varifold topology. Hence there exists $\{\Sigma_1,\dots,\Sigma_q\}\subset \mathcal{B}$ such that
$$
\mathcal{B} \subset B_{\bf F}(|\Sigma_1|,\gamma_{\Sigma_1}) \cup \cdots \cup B_{\bf F}(|\Sigma_q|,\gamma_{\Sigma_q}).
$$
By applying Proposition \eqref{refining.sequence.2} to $\Sigma_1, \dots, \Sigma_q$ consecutively, starting with the sequence $\{\Phi_i\}_i$, we get a sequence  of continuous maps $\tilde{\Psi}_i:\tilde{Y}_i\rightarrow \mathcal{S}^{k,\alpha}$, $\tilde{Y}_i$ a  three-dimensional compact simplicial complex, which are  three-sweepouts and such that
$$
{\bf \Lambda}(\{\tilde{\Psi}_i\}_i) \subset {\bf \Lambda}(\{\Phi_i\}_i)  \cap \left(B_{\bf F}(|\Sigma_1|,\gamma_{\Sigma_1})\right)^c \cap \cdots \cap \left(B_{\bf F}(|\Sigma_q|,\gamma_{\Sigma_q})\right)^c.
$$
The proposition is proved by setting $\tilde{\mathcal{G}}={\bf \Lambda}(\{\tilde{\Psi}_i\}_i)$.
\end{proof}

 Let $\Sigma \in {\mathcal{G}}$. By the first variation formula the area of ${\varphi}_\Sigma(z)$ is equal to $\sigma$ for every $z\in \overline{B}^3$. 
Since ${\varphi}_\Sigma$ is an embedding, the nullity of ${\varphi}_\Sigma(z)$ is greater than or equal to three for every $z\in B^3$. Therefore ${\varphi}_\Sigma(z)$ cannot be stable, and hence has Morse index greater than or equal to one. Let $Y\subset B^3$ be the set of $z\in B^3$  such that the Morse index of ${\varphi}_\Sigma(z)$ is one. By \cite{cheng}, the second eigenvalue of the Jacobi operator has multiplicity at most three. Hence, for $z\in Y$, the nullity of ${\varphi}_\Sigma(z)$ is equal to three for every $z\in B^3$. By continuity of the eigenvalues the set $Y$ is open. Since $Y\subset B^3$ is also closed, we get $Y=B^3$.
Therefore ${\varphi}_\Sigma(B^3)\subset {\mathcal{G}}$.

The set ${\mathcal{G}}$ has the structure of a smooth three-manifold induced by the charts $({\varphi}_\Sigma)_{|B^3}$. Let $\{\tilde{\Psi}_i:\tilde{Y}_i\rightarrow \mathcal{S}^{k,\alpha}\}_i$ and $\tilde{{\mathcal{G}}}$ be as in Proposition \ref{refining.sequence.2}. Note that
$\tilde{\mathcal{G}}\subset {\mathcal{G}}$.

\begin{prop}\label{approximation.final}
There exists a sequence of continuous three-sweepouts  $\Psi_i:\tilde{Y}_i\rightarrow {{\mathcal{G}}}$ such that
$$
\lim_{i\rightarrow \infty} \sup_{x\in \tilde{Y}_i} {\bf F}(\Psi_i(x),\tilde{\mathcal{G}})=0.
$$
\end{prop}

\begin{proof}
Let $h$ be a complete Riemannian metric on $\mathcal G$. Consider $\bar \rho$ so that,  for all $\Sigma\in\tilde{\mathcal G}$ and $\rho\leq \bar\rho$ the geodesic ball of radius $\rho$ centered at $\Sigma$ with respect to $h$, $B_{\rho}(\Sigma,h)$, is geodesically convex, meaning that any two points $B_{\rho}(\Sigma,h)$ can be connected by a unique geodesic contained in $B_{\rho}(\Sigma,h)$. We can suppose that for some constant $c$ we have
\begin{equation}\label{distance.metrics}
c{\bf F}(\Sigma,\Sigma')\leq d_h(\Sigma,\Sigma')\leq {\bf F}(\Sigma,\Sigma')\quad\text{for all }\Sigma \in\tilde{\mathcal G},\, \Sigma'\in B_{\bar \rho}(\Sigma,h)
\end{equation}
where $d_h$ is the distance with respect to the metric $h$. 

Choose $\rho<\bar \rho/2$. We can apply Proposition \ref{limit.flat.topology} to $\{\tilde{\Psi}_i\}$ and so for sufficiently large $i$
$$
 \sup_{x\in \tilde{Y}_i} {\bf F}(\tilde{\Psi}_i(x),\tilde{\mathcal{G}})<\rho/5. 
$$
Consider a barycentric subdivision $\tilde{Y}_i(l_i)$ of $\tilde{Y}_i$ such that 
$$
\mathcal{F}(\tilde{\Psi}_i(x),\tilde{\Psi}_i(y))+ {\bf F}(|\tilde{\Psi}_i(x)|,|\tilde{\Psi}_i(y)|)<\rho/5
$$
for every face $\alpha\in \tilde{Y}_i(l_i)$, $x,y\in \alpha$.  If $x$ is a vertex of $\tilde{Y}_i(l_i)$, choose $\Sigma_x\in \tilde{\mathcal{G}}$ with 
${\bf F}(\tilde{\Psi}_i(x),\Sigma_x)<\rho/5$ and define $\Psi_i(x)=\Sigma_x$.  For every vertex $x$, ${\bf F}(\Psi_i(x),\tilde{\Psi}_i(x))< \rho/5$ and  if $x,y$ are vertices of the same face, then ${\bf F}(\Sigma_x,\Sigma_y)< \rho$. 

Let $[x,y]$ be a one-dimensional face of $\tilde{Y}_i(l_i)$ ($x,y$ are vertices). Notice that $d_h(\Sigma_x,\Sigma_y)<\rho$ from  \eqref{distance.metrics} and hence we can use the geodesic convexity of $B_{\rho}(\Sigma_x,h)\subset {\mathcal{G}}$ to extend $\Psi_i$ to a continuous map $\Psi_i:[x,y]\rightarrow B_{\rho}(\Sigma_x,h)$. From \eqref{distance.metrics} we have ${\bf F}(\Psi_i(v),\tilde{\Psi}_i(v))<2c^{-1} \rho$ for every $v\in [x,y]$.

Let $\sigma$ be a two-dimensional face of $X_i(l_i)$. If $x$ is a vertex of $\sigma$ we have $\Psi_i(\partial\sigma)\subset B_{\rho}(\Sigma_x,h)$ from geodesic convexity. Connecting $\Sigma\in B_{\rho}(\Sigma_x,h)-\Psi_i(\partial\sigma)$  to every $\Sigma_v$, $v\in\partial\sigma$, with a geodesic we can extend $\Psi_i$ to a continuous map $\Psi_i:\sigma\rightarrow B_{\rho}(\Sigma_x,h)$ that sends the barycenter of $\sigma$ to some $S\in B_{\rho}(\Sigma_x,h)-\Psi_i(\partial\sigma)$ and  maps lines departing from the barycenter to geodesics  departing from $S$.
We have
${\bf F}(\Psi_i(v),\tilde{\Psi}_i(v))<2c^{-1}\rho$ for every $v\in \sigma$.

Suppose $\sigma$ is a three-dimensional face of $X_i(l_i)$. If $x$ is a vertex of $\sigma$, then for every other vertex $y$ of $\sigma$ we have $\Sigma_y\in B_{\rho}(\Sigma_x,h)$ and so $\Psi_i(\partial\sigma)\subset B_{2\rho}(\Sigma_x,h)$ from geodesic convexity. Therefore we can extend $\Psi_i$ to a continuous map $\Psi_i:\sigma\rightarrow \mathcal{G}$ that
sends the barycenter of $\sigma$ to some $S\in B_{2\rho}(\Sigma_x,h)-\Psi_i(\partial\sigma)$ and maps lines departing from the barycenter to geodesics  departing from $S$. 
 We have
${\bf F}(\Psi_i(v),\tilde{\Psi}_i(v))<3c^{-1}\rho$ for every $v\in \sigma$.

The continuous map $\Psi_i:\tilde{Y}_i\rightarrow {\mathcal{G}}$ satisfies ${\bf F}(\Psi_i(x),\tilde{\Psi}_i(x))<3c^{-1}\rho$ for every $x\in \tilde{Y}_i$. Since $\rho>0$ is arbitrary, a diagonal argument finishes the proof.
\end{proof}

\subsection{Proof of Theorem \ref{zoll.characterization.2}}

{Consider a map $$\Psi_i:\tilde{Y}_i\rightarrow {{\mathcal{G}}}$$ given by Proposition \ref{approximation.final}.} We can suppose that $\tilde{Y}_i$ is connected (if not then restrict to the connected component that detects $(\Psi_i)^*(\overline{\lambda}^3)$).
There exists $\alpha\in H_3(\tilde{Y}_i,\mathbb{Z}_2)$ such that $(\Psi_i)^*(\overline{\lambda}^3) \cdot \, \alpha=1$. Then 
$\overline{\lambda}^3 \cdot (\Psi_i')_*(\alpha)=1$. Hence for any open set $Z\subset {\mathcal G}$ with $\Psi_i(\tilde{Y}_i)\subset Z$,  it holds that  
$\overline{\lambda}^3_{|Z}\neq 0\in H^3(Z,\mathbb{Z}_2)$. Notice that $\Psi_i(\tilde{Y}_i)$ is compact and connected. {Now a $3$-manifold $X$ has  $H^3(X,\mathbb{Z}_2)=0$ unless it is closed}.  Hence $Z_i=\Psi_i(\tilde{Y}_i)$ is a {closed} three-manifold and the map  $\Psi_i:\tilde{Y}_i\rightarrow Z_i$ is surjective. By passing to a subsequence, we can suppose   that $Z_i$ is a connected component $Z$ of ${\mathcal G}$.

Hence the inclusion map $i_{Z}:Z\rightarrow \mathcal{S}^{k,\alpha}$ is a three-sweepout. By Allard, we can consider $i_{Z}$ as a continuous map $i_{Z}:Z\rightarrow \mathcal{S}$ in the smooth topology. Let $Z'$ be the set of all oriented spheres $S$ such that the corresponding unoriented sphere $\Sigma=|S|\in Z$.  Then $Z'$ is a compact three-manifold and the map $\pi:Z'\rightarrow Z$, $\pi(S)=|S|$, is a two-cover. Since $i_Z$ is in particular a one-sweepout, there is a continuous path in $Z'$ connecting an oriented sphere $S$ to the same sphere $\tilde{S}$ with the opposite orientation. Hence $Z'$ is connected.

Let $\mathcal{P}Z'$ be the space of pointed two-spheres $(S,q)$, with $S\in Z'$ and $q\in S$. If $(S,q)\in Z'$, and $\Sigma=|S|$, let ${\varphi}_\Sigma:\overline{B}^3\rightarrow \mathcal{S}^{k,\alpha}$ be as before.
Let $h:\overline{B}^3\rightarrow C^{k,\alpha}(\Sigma)$ be the smooth embedding such that ${\varphi}_\Sigma(z)={\rm graph}_\Sigma(h(z))$, where the graph is in the direction of the unit normal $\nu$ to $S$.  We can consider ${\varphi}_\Sigma(z)$ with the orientation {compatible with the} orientation of $S$. For a neighborhood $U\subset \Sigma$ of $q$, consider the map $\psi_\Sigma: B^3\times U\rightarrow \mathcal{P}Z'$ defined by
$$
\psi_\Sigma(z,x)=({\varphi}_\Sigma(z),{\rm exp}_x(h(z)(x)\nu(x))).
$$
Notice that $\psi_\Sigma$ is continuous and $\psi_\Sigma(B^3\times U)$ is an open set of $\mathcal{P}Z'$. The inverse map is
$$
(S',q')\mapsto    ({\varphi}_\Sigma^{-1}(|S'|),\pi_\Sigma(q')),
$$
where $\pi_\Sigma:V_\Sigma\rightarrow \Sigma$ is the nearest point projection on $\Sigma$, defined on a tubular neighborhood $V_\Sigma$ of $\Sigma$. We can suppose 
that ${\varphi}_\Sigma(z)\subset V_\Sigma$ for every $z\in \overline{B}^3$. Hence $\mathcal{P}Z'$ is a topological
compact five-dimensional manifold.

There is a natural map $T: \mathcal{P}Z' \rightarrow Gr_{2}(S^3)$, where $Gr_{2}(S^3)$ denotes the set of oriented two-dimensional tangent planes  to $S^3$, defined by 
$$
T(S,q)=T_qS.
$$
The map $T$ is continuous. Given $(S,q)\in \mathcal{P}Z'$,  we claim there is a neighborhood $\mathcal{U}\subset \mathcal{P}Z'$ such that $T_{|\mathcal{U}}$ is injective. If not, by contradiction,  there are sequences 
$\{(S_i^1,q_i^1)\}_i, \{(S_i^2,q_i^2)\}_i$ of pointed surfaces in $\mathcal{P}Z'$  converging to $(S,q)$ and such that
$(S_i^1,q_i^1)\neq (S_i^2,q_i^2)$ and 
$T(S_i^1,q_i^1)=T(S_i^2,q_i^2)$ for every $i$. Hence $q_i^1=q_i^2$ and $T_{q_i^1}S_i^1=T_{q_i^2}S_i^2$.  By writing $S_i^2$ as a graph over $S_i^1$, and using elliptic  theory applied to the minimal surface equation, one can extract a nontrivial Jacobi field $h$ of $S$ in the limit.  Since the Morse index of $S$ is one, $h$ is a second eigenfunction for the Jacobi operator of $S$. But the
function $h$ satisfies $h(q)=0$ and $\nabla_Sh(q)=0$, which contradicts the proof of Theorem 3.2 in \cite{cheng}. This proves that $T$ is locally injective.

Since ${\rm dim}\, \mathcal{P}Z' ={\rm dim} \, Gr_{2}(S^3)$, by Invariance of Domain the map $T$ is a local homeomorphism.  In fact, $T$ is a covering map because $\mathcal{P}Z' $ is compact.  Because $S^3$ is parallelizable,  $Gr_{2}(S^3)$ is homeomorphic to $S^3\times S^2$. Hence $Gr_{2}(S^3)$ is simply connected, which implies $T$ is a homeomorphism. This implies $Z'$ is a transitive family of minimal spheres in the sense of \cite{galvez-mira}.  

We claim that $Z'$ is simply connected. Let $\sigma:[0,1]\rightarrow Z'$ be a continuous path with $\sigma(0)=\sigma(1)=S$.  By the homotopy lifting property of fiber bundles, there exists a continuous family of smooth embeddings $t\in [0,1]\mapsto i(t)$ of $S^2$ in $S^3$ such that $\sigma(t)=i(t)(S^2)$ with the orientation induced by $i(t)$. Let $x\in S^2$, and choose $y\in S^2$ such that $i(1)(y)=i(0)(x)$. Choose a smooth path $x:[0,1]\rightarrow S^2$ with $x(0)=x$ and $x(1)=y$. Then $q(t)=i(t)(x(t))\in \sigma(t)$ for $t\in [0,1]$. Notice that $q(0)=q(1)$. Then we define $\eta:[0,1]\rightarrow \mathcal{P}Z' $ by $\eta(t)=(\sigma(t),q(t))$. Hence $\eta(0)=\eta(1)$. 
But $\mathcal{P}Z'$ is simply connected. Therefore there exists a continuous homotopy $H:[0,1]\times [0,1]\rightarrow \mathcal{P}Z'$ such that
$H(t,0)=\eta(t)$, $H(t,1)=\eta(0)=\eta(1)$ for every $t\in [0,1]$, and $H(0,s)=\eta(0)$, $H(1,s)=\eta(1)$ for every $s\in [0,1]$. If $\pi(S,q)=S$, then $h(t,s)=\pi(H(t,s))$ is
a homotopy between $\sigma$ and a constant map with the endpoints fixed. This proves $Z'$ is simply connected. Hence $\pi_1(Z)=\mathbb{Z}_2$, which
implies $Z$ is diffeomorphic to $\mathbb{RP}^3$ by geometrization. It follows that  $Z$ is a Zoll family of minimal spheres for the metric $g$, finishing the proof of Theorem \ref{zoll.characterization.2}.

\section{Systolic inequalities on $\RP^3$}

We denote by $\mathcal{Z}_2^{'}(\mathbb{RP}^3,\mathbb{Z}_2)$ the space of mod two flat two-dimensional chains $T$ with $T-S=\partial U$, where $U \in {\bf I}_3(\mathbb{RP}^3,\mathbb{Z}_2)$ and $S\subset \mathbb{RP}^3$ is a linear projective plane,  endowed with the flat metric. Hence $\mathcal{Z}_2^{'}(\mathbb{RP}^3,\mathbb{Z}_2)=\mathcal{Z}_2(\mathbb{RP}^3,\mathbb{Z}_2)+S$.

We use  $\mathcal{S}$ to denote  the space of embedded two-dimensional projective spaces in $\RP^3$ with the smooth topology.  Any $\Sigma \in \mathcal{S}$ is $\pi_1$-injective \cite{bben} and  homologous mod 2  to a linear projective plane.
The inclusion $i:\mathcal{S}\rightarrow \mathcal{Z}_2(\RP^3,{\bf F},\mathbb{Z}_2)$ is a continuous map. 
A continuous map $\Phi:X\rightarrow \mathcal{S}$ is called  a smooth $k$-sweepout (by projective planes) if the composition $i\circ \Phi$ is a $k$-sweepout. We denote $\Phi \in \mathcal{P}'_k$.  Notice that in contrast to the case of $(S^3,g)$,  there is a positive constant $c$ so that ${\bf M}(S)\geq c$ for any $S\in \mathcal{S}$.

We define for $(\RP^3,g)$  the projective area widths
$$
\sigma_k(\RP^3,g)=\inf_{\Phi\in  \mathcal{P}'_k} \sup_{x\in {\rm dmn}(\Phi)} {\bf M}(\Phi(x)),
$$
for $0\leq k\leq 3$. 
We have $\mathcal{P}'_k\subset \mathcal{P}'_l$  for $l\leq k$ and 
\begin{equation}\label{sigma0}
\sigma_0(\RP^3,g)=\inf_{\Sigma\in\mathcal S}{\bf M}(\Sigma).
\end{equation}

The systole of a Riemannian manifold $(X,h)$, $\text{sys}(X,h)$, is defined as the least length of a non-contractible loop in $X$.

We prove that a constant curvature metric on $\RP^3$ is determined modulo isometries by the invariants ${\rm sys}$ and $\sigma_2$:

\begin{thm}\label{systolic.inequality} Let $g$ be a smooth Riemannian metric on $\RP^3$. Then 
$$  {\rm sys}^2(\RP^3,g)\leq \frac{\pi}{2}\,\sigma_2(\RP^3,g),$$
with equality if and only if $g$ has constant sectional curvature.
\end{thm}
\begin{proof}
We start by checking that the equality holds for constant sectional curvature metrics. It suffices to consider the canonical metric $\overline{g}$ on $\RP^3$. Set $\Psi:\RP^3\rightarrow \mathcal S$ given by
$$\Psi([a_1:a_2:a_3:a_4])=\{[x]\in \RP^3:a_1x_1+a_1x_1+a_3x_3+a_4x_4=0\}.$$
The map $\Psi$  is a three-sweepout and ${\bf M}(\Phi(x))=2\pi$  for any $x\in \RP^3$. Thus $\sigma_3(\RP^3,\overline{g})\leq 2\pi$.  Any linear projective space in $\RP^3$ is area-minimizer for the canonical metric and so $\sigma_0(\RP^3,\overline{g})=2\pi$. Thus $\sigma_2(\RP^3,\overline{g})=2\pi$. The systole of $(\RP^3,\overline{g})$ is $\pi$ which implies the equality in  Theorem \ref{systolic.inequality} for $\overline{g}$.

Pu's systolic inequality is used in  \cite{bben} to show that for every $\Sigma\in\mathcal S$ we have 
\begin{equation}\label{identity2.holds}\text{area}(\Sigma)\geq\frac{2}{\pi}{\rm sys}^2(\Sigma,g_{\Sigma})\geq  \frac{2}{\pi}{\rm sys}^2(\RP^3,g).
\end{equation} Hence
\begin{equation}\label{identity.holds}
\sigma_2(\RP^3,g)\geq\sigma_0(\RP^3,g)=\inf_{\Sigma\in\mathcal S} {\bf M}(\Sigma)\geq  \frac{2}{\pi}{\rm sys}^2(\RP^3,g).
\end{equation}

Suppose that $\sigma_2(\RP^3,g)=\sigma_0(\RP^3,g)$ and let $\{\Phi_i\}_{i\in \mathbb{N}}\subset \mathcal{P}_3'$ be such that
$$
\lim_{i\to\infty}\sup_{x\in X_i}{\bf M}(\Phi_i(x))=\sigma_2(\RP^3,g),
$$  
where $X_i={\rm dmn}(\Phi_i)$ is a two-dimensional cubical complex.

Let $\mathcal V$ be the set of stationary integral varifolds of the form $V=|\Sigma|$, where $\Sigma$ is a minimal embedded projective plane in $(\RP^3,g)$ with $\text{area}(\Sigma)=\sigma_0(\RP^3,g)$.   Note that $\Sigma$ minimizes area in its isotopy class by \eqref{sigma0} and is thus stable. Hence we have from Sharp's compactness theorem that $\mathcal V$ is compact in the varifold topology  and so compact in the smooth topology  by Allard.

\begin{lem}\label{varifold.approximation} Then 
$\lim_{i\to\infty} \sup_{x\in X_i}{\bf F}(|\Phi_i(x)|,\mathcal V)=0$.
\end{lem}
\begin{proof}
Let $x_i\in X_i$. We have that 
$$\inf_{\Sigma\in\mathcal S} {\bf M}(\Sigma)\leq \limsup_{i\to\infty}{\bf M}(\Phi(x_i))\leq \sigma_2(\RP^3,g)=\sigma_0(\RP^3,g)=\inf_{\Sigma\in\mathcal S} {\bf M}(\Sigma).$$
Thus $\Phi_i(x_i)$ is a minimizing sequence in the sense of \cite{msy} and thus  contains a convergent subsequence to a varifold $V$ whose support is a minimal embedded projective plane in $(\RP^3,g)$ with $\text{area}(\Sigma)=\sigma_0(\RP^3,g)$.
\end{proof}
\begin{prop}\label{plane.constant} For any  $v\in T_p\RP^3$, there is $V=|\Sigma|\in \mathcal V$ so that $v\in T_p\Sigma$.
\end{prop}
\begin{proof}

 We claim that for any two   points $p_1, p_2 \in \RP^3$, there is $V\in \mathcal V$ whose support contains $p_1,p_2$.
 
Let  $r>0$ be sufficiently small so that $B_r(p_j)$ are  disjoint, $j=1,2$.  Since $\Phi_i$ is a two-sweepout, it follows there is $x_i\in X_i$ so that $\text{area}(\Phi_i(x_i)\cap B_r(p_j))\geq c$
for any $i\in \mathbb{N}$, $j=1,2$, for some constant $c>0$   (see \cite{marques-neves-infinitely}). Thus we obtain from Lemma \ref{varifold.approximation} the existence of a minimal embedded projective plane $\Sigma_r$ so that $|\Sigma_r|\in\mathcal V$ and $\Sigma_r\cap B_r(p_j)\neq\emptyset$, $j=1,2$ .  Letting $r\to 0$, we obtain from the compactness of $\mathcal V$  in the smooth topology the existence of $V\in \mathcal V$ whose support contains $p_1,p_2$.

If $p_1=p$, and $p_2=\exp_p(tv)$, for $t>0$,
it follows that there is a stable minimal embedded projective plane $\Sigma(t)$ containing $\{p, \exp_p(tv)\}$. There is a sequence  $t_i\to 0$ such that $\Sigma(t_i)$ converges strongly to a minimal embedded projective plane $\Sigma$. The surface $\Sigma$ cannot be transversal to the geodesic $t\mapsto \exp_p(tv)$ at $p$.  Hence $v\in T_p\Sigma$. 
\end{proof}

Suppose that  $\sigma_2(\RP^3,g)= \frac{2}{\pi}{\rm sys}^2(\RP^3,g)$. From \eqref{identity2.holds}  and \eqref{identity.holds} we see that $\sigma_2(\RP^3,g)=\sigma_0(\RP^3,g)$, every  $V=|\Sigma|\in\mathcal V$ has support satisfying the  case of equality  in Pu's systolic inequality, and ${\rm sys}(\Sigma,g_{\Sigma})={\rm sys}(\RP^3,g)$. 

Thus for every $V=|\Sigma|\in\mathcal V$ we have that $\Sigma$ has constant Gaussian curvature $K=2\pi/\sigma_0(\RP^3,g)$.  Every closed geodesic of $\Sigma$ is a  systole of $(\Sigma,g_{\Sigma})$ and hence a closed geodesic of $(\RP^3,g)$. By Proposition \ref{plane.constant}, for any $p\in \mathbb{RP}^3$ and $v\in T_p\mathbb{RP}^3$, there is $V=|\Sigma|\in\mathcal V$ such that $v\in T_p\Sigma$. Since $(\Sigma,g_\Sigma)$ is a Zoll surface, it follows that any  geodesic of $(\mathbb{RP}^3,g)$ is closed, embedded, with length $\pi/\sqrt{K}$. Hence $(\mathbb{RP}^3,g)$ is a Zoll metric, and by Besse \cite{besse} (Appendix D) the metric $g$ has constant sectional curvature. 
\end{proof}

\begin{rmk}
The invariants ${\rm sys}$ and $\sigma_1$ are not sufficient to determine the canonical metric of $\mathbb{RP}^3$ modulo isometries. Hence Theorem \ref{systolic.inequality} is also sharp in this sense.

The map $\Phi: S^1=([0,\pi]/\sim) \rightarrow \mathcal S$ defined by $\Sigma(t)=\{[x]\in \mathbb{RP}^3: (\cos t)x_1+(\sin t)x_2=0\}$ is a one-sweepout.  Let $Y(x)=(x_2,-x_1,0,0)$ for $x\in S^3$. Then $Y$ induces a tangent vector field on $\mathbb{RP}^3$ such that if $x\in \Sigma(t)$, then $Y(x)\perp T_x\Sigma(t)$. If $z\in \mathbb{RP}^3$ is such that $Y(z)\neq 0$, let $\{e_1,e_2\}$ be an orthonormal frame for the constant curvature metric $\overline{g}$  in an open set $U\subset \mathbb{RP}^3$, $z\in U$, such that  $\langle e_i,Y\rangle=0$, $i=1,2$. If $\psi$ is a nonnegative smooth function with support contained in $U$, positive somewhere, the metric 
$$g_\psi=e_1 \cdot e_1+e_2\cdot e_2 + (1+\psi) (Y/|Y|_{\overline{g}})\cdot (Y/|Y|_{\overline{g}})$$
extends to a smooth metric on $\mathbb{RP}^3$ such that $\overline{g}\leq g_\psi$ and $(\Sigma(t),g_\psi)=
(\Sigma(t),\overline{g})$ for $t\in S^1$. The metric $g_\psi$ is such that ${\rm sys}(g_\psi)={\rm sys}(\overline{g})$,  $\sigma_1(g_\psi)=\sigma_1(\overline{g})$, and does not  have constant sectional curvature.
\end{rmk}

\end{document}